\documentclass[final,times,authoryear,1p]{elsarticle}

\usepackage{amsmath,amsfonts,amssymb}
\usepackage{amsthm}
\usepackage{graphicx}
\usepackage{color}
\usepackage{framed}
\usepackage{fancyhdr}
\newtheorem{theorem}{Theorem}[section]
\newtheorem{lemma}[theorem]{Lemma}
\newtheorem{proposition}[theorem]{Proposition}
\newtheorem{corollary}[theorem]{Corollary}
\newtheorem{definition}[theorem]{Definition}
\newtheorem{remark}[theorem]{Remark}
\newtheorem{example}{\it Example\/}

\usepackage{amssymb}

\journal{Transportation Research Part B}

\begin{document}

\begin{frontmatter}



 \begin{center}
 \textcolor{blue}{ARTICLE LINK:  http://www.sciencedirect.com/science/article/pii/S0191261515000971
\\  PLEASE CITE THIS ARTICLE AS\\ 
Han, K., Szeto, W.Y., Friesz, T.L., 2015. Formulation, existence, and computation of boundedly rational dynamic user equilibrium with fixed or endogenous user tolerance. \\Transportation Research Part B 79, 16-49.}
 \line(1,0){469}
 \end{center}

\title{Formulation, existence, and computation of boundedly rational dynamic user equilibrium with fixed or endogenous user tolerance}


\author[ic]{Ke Han \corref{cor}}
\ead{k.han@imperial.ac.uk}

\author[hk]{W. Y. Szeto}
\ead{ceszeto@hku.hk}

\author[psu]{Terry L. Friesz}
\ead{tfriesz@psu.edu}

\cortext[cor]{Corresponding author}

\address[ic]{Department of Civil and Environmental Engineering, Imperial College London, United Kingdom.}
\address[hk]{Department of Civil Engineering, the University of Hong Kong, China.}
\address[psu]{Department of Industrial and Manufacturing Engineering, Pennsylvania State University, USA.}

\begin{abstract}
This paper analyzes {\it simultaneous route-and-departure-time} (SRDT) {\it dynamic user equilibrium} (DUE) that incorporates the notion of {\it boundedly rational} (BR) user behavior in the selection of departure time and route choices. Intrinsically, the {\it boundedly rational} dynamic user equilibrium (BR-DUE) model we present assumes that travelers do not always seek the least costly route-and-departure-time choice. Rather, their perception of travel cost is affected by an indifference band describing travelers' tolerance of the difference between their experienced travel costs and the minimum travel cost. An extension of the BR-DUE problem is the so-called {\it variable tolerance dynamic user equilibrium} (VT-BR-DUE) wherein endogenously determined tolerances may depend not only on paths, but also on the established path departure rates.

This paper presents a unified approach for modeling both BR-DUE and VT-BR-DUE, which makes significant contributions to model formulation, analysis of existence, solution characterization, and heuristic numerical computation of such problems. The VT-BR-DUE problem, together with the BR-DUE problem as a special case, is formulated as a variational inequality. We provide a very general existence result for VT-BR-DUE and BR-DUE that relies on assumptions weaker than those required for mere DUE models. Moreover, a characterization of the solution set is provided based on rigorous topological analysis. Finally, three computational algorithms are proposed based on the VI and DVI formulations. Numerical studies are conducted to assess the proposed algorithms in terms of solution quality, convergence, and computational efficiency.
\end{abstract}


\end{frontmatter}

\section{\label{Intro}Introductory remarks}

This paper studies an extension of the {\it simultaneous route-and-departure-time  dynamic user equilibrium} (SRDT DUE)  \citep{Friesz1993}. Namely we incorporate the concept of {\it bounded rationality} (BR) proposed by \cite{Simon1957, Simon1990, Simon1991} for the modeling of travel behavior. As such, BR-DUE models are developed under the assumption that travelers, viewed as Nash agents, do not behave in a completely rational manner.

In the literature of traffic user equilibrium, the modeling of travelers' route and/or departure time choices has been greatly influenced by Wardrop's first principle \citep{Wardrop}, which states that road users behave in a rational way and seek to minimize their own travel times/costs by making route (or departure time) choices. There are multiple means of expressing the dynamic notion of Wardropian user equilibrium, such as  variational inequality \citep{Friesz1993, existence, SL}, differential variational inequality \citep{FBST, FKKR, FHNMY}, and nonlinear complementarity problem \citep{HUD, WTC}. While enjoying a number of canonical mathematical representations, the notion of completely rational user equilibrium is not entirely in line with realistic driving behavior and empirical observations. That is,  travelers may not always choose the departure time and/or route that yield the minimum travel cost. Such behavior may be due to (1) imperfect travel information; and (2) certain ``inertia" in decision-making. Indeed, empirical studies suggest that in reality, drivers do not always follow the least costly route-and-departure-time choice \citep{AP}. Accordingly, one may model such boundedly rational behavior by introducing travelers' tolerances towards the difference in travel costs, and postulating a range of costs acceptable to travelers, rather than the minimum travel cost.

The main subjects of this paper include {\it boundedly rational dynamic user equilibrium} (BR-DUE), which relies on fixed (exogenous) tolerances that often depend on the O-D pair in the literature. We also consider the so-called {\it variable tolerance boundedly rational dynamic user equilibrium} (VT-BR-DUE), in which the tolerances may depend on the path and the established (actual) path departure rates. Clearly, BR-DUE is simply a special case of VT-BR-DUE. The latter concept is relevant in situations where drivers' tolerances may be affected by not only their O-D pairs (or more generally, their static attributes such as socio-economic status) but also by a number of external factors. For example, these tolerances may vary by path, depending on road quality, travel distance, scenic quality \citep{BenAkiva}, and personal familiarity \citep{Bonsall}, all of which are associated with a given link or path. In addition, the user tolerances may depend on prevailing traffic conditions \citep{Ueberschaer, Huchingson}, which are completely determined by the actual path departure rates. However, it is important for us to clarify that whether and how the tolerances should depend on all the aforementioned factors is not the focus here. Rather, this paper is meant to provide the most general modeling framework that can accommodate all these factors if and when they become relevant. According to the literature review presented below, the notion of bounded rationality has received much attention in (dynamic) traffic assignment; and our current paper makes a significant contribution to this area by providing the first complete theory of BR-DUE and VT-BR-DUE.

\subsection{Some literature on bounded rationality}

As a relaxation of the perfect rationality assumption commonly made in Nash games, the notion of bounded rationality is first proposed by \cite{Simon1957, Simon1990, Simon1991} and introduced to traffic modeling by \cite{MC}. In prose, the notion of bounded rationality postulates a range of acceptable travel costs that, when achieved, do not incentivize travelers to change their departure times or route choices.  Such a range is phrased by \cite{MC} as ``indifference band". The width of such a band, usually denoted by $\varepsilon$, is either derived through a behavioral study of road users (for example, by surveys) or calibrated from empirical observation through inverse modeling techniques.  In general, $\varepsilon$ could depend on a specific origin-destination pair and/or travel commodity. Since \cite{MC}, boundedly rational user equilibrium (BR-UE) has received considerable attention in {\it static traffic assignment} (STA), with an incomplete list of research papers including \cite{DLPB,  GC, HT, KA, LYL} and \cite{Marsden}. It is also investigated using simulation-based approaches in the venue of dynamic modeling \citep{HM1997, MJ1991, ML1999, MZL}.

The notion of bounded rationality (BR) was used somewhat imprecisely during the early days of dynamic traffic assignment research. In particular, BR was studied in a so-called laboratory setting by \cite{MC}, without a mathematical articulation of BR for dynamic traffic assignment. BR was used in a similarly fashion for simulations by \cite{JM, PM, ML1999}; and \cite{CM}, which efforts were again limited by the lack of a complete mathematical model of BR for DUE. Recognizing the lack of a theory of traffic assignment that directly incorporates BR, \cite{Ridwan} tried to apply the theory of fuzzy systems to the study of BR. \cite{Bogers}, again driven by the lack of a suitable theory, conducted more laboratory studies of BR.  \cite{Szeto2003} and \cite{SL2006} propose a mathematical model for \emph{route-choice} (RC) boundedly rational dynamic user equilibrium (RC BR-DUE). The RC BR-DUE is  formulated as a discrete-time nonlinear complementarity problem in \cite{SL2006}, where a heuristic route-swapping algorithm is proposed to solve the problem. \cite{GZ} consider RC BR-DUE with endogenously determined tolerances by allowing the width of the indifference band $\varepsilon$ to depend on time and the actual path departure rates. However, no problem formulation, solution existence or computational method are provided in that paper. Contributions by \cite{Szeto2003, SL2006} and \cite{GZ} achieve enhanced (yet partial) integration of BR and DUE but do not establish and analyze a complete theory, where by ``complete" we mean a mathematical formulation consistent with known empirical results and surmised behaviors; qualitative properties; and a computational approach that is demonstrably effective.

To the best of our knowledge, there has not been an analytical treatment of the BR-DUE problem with exogenous or variable (endogenous) tolerances in the literature, in terms of formulation, qualitative analyses, and computation. This paper bridges this gap by establishing the first complete analytical framework capable of formulating BR-DUE problems into familiar mathematical forms (such as variational inequalities), analyzing solution existence, characterizing the solution set, and computing solutions with convergent algorithms. These specific contributions are elaborated below.

\subsection{Contribution made in this paper}

This paper makes contributions in four major areas:

\begin{enumerate}
\item problem formulation;

\item existence theory;

\item characterization of solution set; and

\item algorithm development and assessment.
\end{enumerate}

\subsubsection{Problem formulation}
Study of the BR-DUE problems is most facilitated by formulating them into a canonical mathematical form so that existing analytical and computational tools may apply. This paper proposes, for the first time in the literature, three equivalent mathematical formulations of the SRDT BR-DUE problem with exogenously given tolerances, namely a variational inequality (VI), a differential variational inequality (DVI), and a fixed-point problem (FPP). Remarkably, the same formulations exist for the VT-BR-DUE problem, wherein the tolerances depend endogenously on the path departure rates; this will be shown in our paper as well \footnote{Since the VT-BR-DUE problem subsumes BR-DUE problem as a special case, throughout this paper we will establish results for VT-BR-DUE, while it is understood that the same results apply to BR-DUE automatically.}.  The three equivalent formulations proposed by this paper are demonstrably beneficial to the modeling, analysis, and computation of BR-DUE and VT-BR-DUE. A key innovation that is crucial to these formulations is the introduction of a new operator that simultaneously encapsulates the network performance model (i.e. dynamic network loading) and the endogenous user tolerance. As we will show in Section \ref{secpropoperator}, this new operator is well defined and continuous. It plays a key role in the existence theory, solution characterization, and computation of the problem.

\subsubsection{Solution existence}

Relative to existence of boundedly rational dynamic user equilibria, we
provide an existence theory that rests on new proofs of the existence and
continuity of the newly introduced operators. Our existence theory makes use of the unique mathematical structure of the VT-BR-DUE problems and properties of the
effective delay operator. To our knowledge, this is the first existence result for boundedly rational dynamic user equilibrium, and it applies equally to RC BR-DUE, SRDT BR-DUE and SRDT VT-BR-DUE.  In particular, our existence result for VT-BR-DUE relies on conditions weaker than those employed for normal DUEs \citep{existence} \footnote{The existence of SRDT DUE immediately implies the existence of a VT-BR-DUE or a BR-DUE. Thus a non-trivial discussion of the existence of boundedly rational dynamic user equilibria should rely on conditions weaker than those for DUEs.}. In fact, we will articulate in Section \ref{subsecDUEnotexist} conditions under which the DUE does not exist  or the existence proof is difficult. We then show that VT-BR-DUEs exist under these same conditions, thereby showing that the existence of VT-BR-DUE or BR-DUE is indeed more general than normal DUEs.

The existence proof provided in this paper utilizes the VI representation of SRDT VT-BR-DUE, developed in the same paper. Our existence result does not invoke the {\it a priori} boundedness on the path departure rates. Notice that the issue  of {\it a priori} boundedness is less of a concern for route-choice DUEs by virtue of problem formulation; but it will impact the existence proof for SRDT DUEs due to  compactness in an infinite-dimensional space. This subtle point will be demonstrated in detail in Section \ref{secexistence}. Thus, having established solution existence in the most general setting, we offer an existence theory that subsumes RC BR-DUE, SRDT BR-DUE and SRDT VT-BR-DUE.

\subsubsection{Characterization of the solution set}

This paper provides a characterization of the VT-BR-DUE solution set by establishing rigorous results on the non-uniqueness and behavior of VT-BR-DUE solutions. In particular, we show the compactness of the solution set,
analyze its interior points, and illustrate a procedure to find infinitely
many solutions around a given solution, which leads to the construction of a
connected component of the solution set.

Our characterization constitutes the first result on the non-uniqueness of VT-BR-DUE solutions. In particular, by working within a discrete-time setting, we first show that the solution set, which is nonempty by virtual of the existence theory, is
closed and compact; this is achieved by using the VI formulation. We then
analyze the interior points of the solution set with respect to various
subset topologies; the result gives us a sense of how dense the solution set
is in a very high-dimensional space. Moreover, we construct a precise
procedure to find, in the neighborhood of any VT-BR-DUE solution, a convex
set consisting of infinitely many other solutions. Notably, such a procedure
is purely analytical and does not resort to numerical computations. Last but not least, we utilize such a procedure to construct a connected and possibly nonconvex subset of the solution set. The characterization of the solution set presented in this paper is nontrivial, resting on topological analysis. Moreover, this and other results of similar kind have not been reported previously, to our knowledge, neither in a static nor a dynamic context, and they address important properties of the solution set related to compactness, infiniteness, convexity, and connectedness.

\subsubsection{Computation}

On the computational side, prior to this paper there has not been any algorithm proposed in the literature for computing VT-BR-DUE or BR-DUE, when both route and departure time choices are within the purview of drivers. The formulation of VT-BR-DUE put forward in this paper has a significant impact on the computation of these problems, given many existing computational algorithms for variational inequalities (VIs), differential variational inequalities (DVIs) and fixed-point problems (FPP) in the literature. The proposed VI formulation for BR-DUE, when properly discretized, admits a number of existing solution algorithms \citep{HL, HL2002, SL, Ukkusuri}. On the other hand, the computation of a continuous-time VT-BR-DUE is most facilitated by the mathematical paradigm of DVI \citep{PS} and emerging computational algorithms associated therein.

As a demonstration of the computational significance of the proposed formulations, three computational algorithms, based on the VI and DVI formalisms, are proposed and tested on several networks. The first one is a fixed-point method; the second one is a self-adaptive projection method adapted from \cite{HL2002}; and the third one is a proximal point method \citep{Konnov}. These three algorithms enjoy rigorous convergence results given certain notions of monotonicity of the operator; namely strong monotonicity (fixed-point method),  pseudo monotonicity (self-adaptive projection), and quasi monotonicity (proximal point method). Our paper's intent is to: (i) document how far the available mathematics can take us in assuring convergence, and (ii) illustrate what can be done computationally when proceeding heuristically by relaxing monotonicity assumptions needed to assure
convergence. As we demonstrate in the numerical studies, all three algorithms show empirical convergence even though the underlying monotonicity properties are not verified against the network performance model.

In summary, contributions made by this paper include:
\begin{enumerate}
\item  an expression of the simultaneous route-and-departure-time (SRDT) dynamic user equilibrium with bounded rationality (BR-DUE) or with variable tolerance (VT-BR-DUE) as  variational inequalities,  differential variational inequalities, and fixed-point problems;

\item a major existence theory for the VT-BR-DUE, which relies on conditions considerably weaker than those employed for DUEs; 

\item a characterization of the compact and infinite nature of the solution set, and a procedure to find an infinite and connected solution set without numerical computations of VT-BR-DUE solutions; and 

\item three computational algorithms for computing SRDT VT-BR-DUEs with provable convergence results, which are tested on several networks in terms of solution characteristics, convergence, and computational times.
\end{enumerate}

It should be noted that although this paper considers {\it simultaneous route-and-departure-time} DUE problems incorporating bounded user rationality, the proposed analytical framework is easily transferrable to route-choice BR-DUE and even to static BR-UE (boundedly rational user equilibrium) problems. This will be demonstrated in a future publication.

The rest of this paper is organized as follows. Section \ref{secNEB} introduces basic notations and background materials necessary for the presentation of subsequent analyses. We will also articulate the definitions of SRDT DUE, BR-DUE, and VT-BR-DUE. In Section \ref{secbrdue} we show the VI formulation of the SRDT VT-BR-DUE, by introducing a new operator that simultaneously encapsulates the network loading model and the (endogenous) user tolerances. Such an operator will be analyzed in depth in Section \ref{secpropoperator}, in terms of its well-definedness and continuity. In Section \ref{secexistence}, we provide an existence theory for SRDT VT-BR-DUE problems. Section \ref{secsolutionchara} is concerned with the characterization of the solution set for VT-BR-DUE problems. Three computational algorithms will be presented in Section \ref{seccomputation}, together with their convergence results based on various notions of generalized monotonicity.  Several numerical studies of the proposed computational algorithms will be presented in Section \ref{secnumerical}. Finally, Section \ref{secconclusion} offers some concluding remarks.

\section{\label{secNEB}Notation, essential background, and definition}

Throughout this paper, the time interval of analysis is a single commuting period or ``day" expressed as $[t_{0},\, t_{f}] \subset \mathbb{R}$ where $\mathbb{R}$ denotes the set of real numbers. We let $\mathcal{P}$ be the set of all paths employed by travelers. For each path $p\in\mathcal{P}$ we define the path departure rate which is a function of departure time $t\in[t_0,\,t_f]$ \footnote{Throughout this paper, we employ the convention that expresses a function/functional/operator $g$ of certain argument to be $g(\cdot)$, and expresses the image of a particular point $x$ to be $g(x)$.}:
$$
h_p(\cdot):~[t_0,\,t_f]~\rightarrow~\mathbb{R}_+
$$ 
where $\mathbb{R}_+$ denotes the set of non-negative real numbers. Each path departure rate $h_p(\cdot)$ is measured at the entrance of the first arc of the relevant path; and the unit for the path departure rates is {\it vehicles per unit time}.  We next define $h(\cdot)=\{h_p(\cdot): p\in\mathcal{P}\}$ to be a vector of departure rates (path flows). Therefore, $h=h(\cdot)$ can be viewed as a vector-valued function of $t$, the departure time \footnote{For notation convenience and without causing any confusion, we will sometimes use $h$ instead of $h(\cdot)$ to denote the path departure rate vectors.}. 

We denote the space of square-integrable functions on the time horizon $[t_0,\,t_f]$ by $L^2[t_0,\,t_f]$, which is a Hilbert space. Its subset consisting of non-negative functions is denoted $L_+^2[t_0,\,t_f]$.  We stipulate that each path departure rate is square integrable; that is
$$
h_p(\cdot)\in L_+^2[t_0,\,t_f],\qquad \qquad h(\cdot)\in\big(L_+^2[t_0,\,t_f]\big)^{|\mathcal{P}|}
$$
where $\big(L_+^2[t_0,\,t_f]\big)^{|\mathcal{P}|}$ is a subset of the $|\mathcal{P}|$-fold product space $\big(L^2[t_0,\,t_f]\big)^{|\mathcal{P}|}$ consisting of non-negative vector-valued functions. The inner product on the Hilbert space $\big(L^2[t_0,\,t_f]\big)^{|\mathcal{P}|}$ is defined as
\begin{equation}\label{ipdef}
\left<u,\, v\right>~=~\int_{t_0}^{t_f}\left(u(s)\right)^T\,v(s)\,ds~=~\sum_{i=1}^{|\mathcal{P}|}\int_{t_0}^{t_f} u_i(s) v_i(s)\,ds
\end{equation}
where the superscript $T$ denotes the transpose of vectors. Moreover, the norm 
\begin{equation}\label{l2normdef}
\left\|u\right\|_{L^2}~=~\left<u,\,u\right>^{1/2}
\end{equation}
is induced by the inner product \eqref{ipdef}.

Here, as in all
DUE modeling, the single most crucial ingredient is the path delay operator, 
which maps a given vector of departure rates $h$ to a vector of path travel times. More specifically, we let
\begin{equation*}
D_{p}(t,\,h)\qquad \forall t\in[t_0,\,t_f],\quad  \forall p\in \mathcal{P}
\end{equation*}
be the path travel time of a driver departing at time $t$ and following path $p$, given the departure rates $h$ associated with all the paths in the network. We then define the path delay operator $D(\cdot)$ by letting $D(h)=\{D_p(\cdot,\,h):\,p\in\mathcal{P}\}$, which is a vector of time-dependent path travel times.  $D$ is an operator defined on $\big(L_+^2[t_0,\,t_f]\big)^{|\mathcal{P}|}$, which maps a vector of path departure rates $h(\cdot)$ to the vector of path travel times $\{D_p(\cdot,\,h):\,p\in\mathcal{P}\}$. Mathematically, we have 
\begin{equation}\label{Ddef}
D:~\big(L^2_+[t_0,\,t_f]\big)^{|\mathcal{P}|}~\rightarrow~\big(L^2_+[t_0,\,t_f]\big)^{|\mathcal{P}|},\qquad h(\cdot)=\{h_p(\cdot),~p\in\mathcal{P}\}~\mapsto~D(h)=\{D_p(\cdot,\,h),~p\in\mathcal{P}\}
\end{equation}
\noindent The {\it effective path delay operator}, $\Psi$, is similarly defined by including arrival penalties in addition to travel times. As such, the effective path delay is a more general notion of ``travel cost" than path delay. The effective delay operator is defined as follows. 
\begin{equation}\label{rvPsidef}
\Psi:~\big(L^2_+[t_0,\,t_f]\big)^{|\mathcal{P}|}~\rightarrow~\big(L^2_+[t_0,\,t_f]\big)^{|\mathcal{P}|},\qquad h(\cdot)=\{h_p(\cdot),~p\in\mathcal{P}\}~\mapsto~\Psi(h)=\{\Psi_p(\cdot,\,h),~p\in\mathcal{P}\}
\end{equation}
\noindent where
\begin{equation}\label{cost}
\Psi _{p}(t,\,h)~=~D_{p}(t,\,h)+f\big( t+D_{p}(t,h)-T_{A}\big) \qquad \forall t\in[t_0,\,t_f],\quad \forall p\in \mathcal{P}
\end{equation}
$T_{A}$ is the desired arrival time; and the term $f\big( t+D_{p}(t,h)-T_{A}\big)$ assesses a non-negative arrival penalty whenever
\begin{equation}
t+D_{p}(t,\,h)~\neq~T_{A}  \label{pgt}
\end{equation}
as $t+D_{p}(t,\,h)$ is the clock time at which departing traffic arrives at
the destination of path $p\in \mathcal{P}$. Note that $T_A$ is often assumed to depend on the origin-destination pair.

\begin{remark}
Our formulation requires minimal assumptions on the functional form of the penalty function $f(\cdot)$, and can thus accommodate very general arrival time windows, arrival preferences, and special circumstances. For example, if early arrivals are encouraged, then the function $f(\cdot)$ can be selected such that $f(s)$ is increasing for $s<0$. In general, however, the function $f(\cdot)$ needs to be carefully calibrated using behavioral assumptions and empirical data. 
\end{remark}

We interpret $\Psi_p(t,\,h)$ as the {\it perceived} travel cost of drivers departing at time $t$ following path $p$, given the vector of path departure rates $h$.  We stipulate that each path effective delay
\begin{equation*}
\Psi _{p}(\cdot,\,h):~ [t_{0},\,t_{f}]~ \longrightarrow~ \mathbb{R} _{++}\qquad \forall p\in \mathcal{P}
\end{equation*}
is measurable, positive, and square integrable, where $\mathbb{R}_{++}$ denotes the set of positive real numbers. The notion of positive functions, as we employ throughout this paper, refers to measurable functions that are positive {\it almost everywhere}.  The notation
\begin{equation*}
\Psi(h)~\doteq~\{\Psi _{p}(\cdot ,h):~p\in \mathcal{P}\} \in \big(L_+^2[t_0,\,t_f]\big)^{|\mathcal{P}|}
\end{equation*}
is used to express the complete vector of effective delays.

The (effective) path delay operator is a key component of analytical dynamic user equilibrium (DUE) models. It is usually not available in closed form and has to be numerically evaluated from dynamic network loading (DNL), which is a subproblem of a complete DUE model. The DNL sub-problem aims at describing and predicting the spatial-temporal evolution of traffic flows on a network that is consistent with established route and departure time choices of travelers, by introducing appropriate dynamics to flow propagation, flow conservation, and travel delays on a network level.  Any DNL must be consistent with the established path departure rates and link delay model, and is usually performed under the {\it first-in-first-out} (FIFO) rule. A few link flow models commonly employed for the DNL procedure include the link delay model \citep{Friesz1993}, the Vickrey model \citep{GVM1, GVM2}, the cell transmission model \citep{CTM1, CTM2}, the link transmission model \citep{LTM, LKWM}, and the Lighthill-Whitham-Richards model \citep{LW, Richards}.

Studies of the dynamic network loading models date back to the 1990's with a significant number of publications. \citep{FKKR, FHLY, FHNMY, Handissertation,  LS, NZ2010, Szeto2003, Szeto2011, SL, SL2006, Ukkusuri}. Notably, despite the absence of closed-form representations of the delay operators, it has been reported that certain dynamic network loading models can be explicitly expressed as a system of {\it differential algebraic equations} (DAEs) or {\it partial differential algebraic equations} (PDAEs). Those results include: the DAE system formulation of the DNL procedure for the link delay model \citep{FKKR}; the DAE system formulation of the DNL procedure for the Vickrey model \citep{Handissertation}; the DAE system formulation of the DNL procedure for the LWR-Lax model \citep{FHNMY}; and the PDAE system formulation of the LWR model incorporating spillback \citep{LWRcont}.

\subsection{Simultaneous route-and-departure choice dynamic user equilibrium (SRDT DUE)}\label{subsecSRDTDUE}
We recap the notion of SRDT DUE originally articulated by \cite{Friesz1993}.  This DUE problem is based on a fixed trip matrix $\big(Q_{ij}: (i,\,j)\in\mathcal{W}\big)$, where each $Q_{ij}\in \mathbb{R} _{++}$ is the fixed travel demand between origin-destination (O-D) pair $\left( i,j\right) \in \mathcal{W}$, and $\mathcal{W}$ is the set of O-D pairs. Note that in the context of simultaneous route and departure time choices, $Q_{ij}$ represents traffic volume, not flow \footnote{This is in contrast to route-choice DUEs where the demand for each O-D pair is specified as a time-varying departure rate, i.e. in the form of flow, not volume.}. Finally we let $\mathcal{P}_{ij}\subset \mathcal{P}$  be the set of paths connecting O-D pair $\left( i,j\right) \in \mathcal{W}$.

The demand satisfaction constraint for the path departure rates $h$ is written as
\begin{equation}\label{cons}
\sum_{p\in \mathcal{P}_{ij}}\int_{t_0}^{t_f}h_p(t)\,dt~=~Q_{ij}\qquad\forall (i,\,j)\in\mathcal{W}
\end{equation}
\noindent With the notation and concepts we have thus far introduced, the feasible region for the DUE problem is
\begin{equation}\label{chapVI:lambda}
\Lambda~\doteq~\left\{ h(\cdot)\geq 0:\sum_{p\in \mathcal{P}_{ij}}
\int_{t_{0}}^{t_{f}}h_{p}\left( t\right) dt=Q_{ij}\quad \forall
\left( i,j\right) \in \mathcal{W}\right\} \subseteq \left( L_{+}^{2}[t_{0},\,t_{f}] \right) ^{\left\vert \mathcal{P}\right\vert }
\end{equation}

We now define the minimum travel cost within a given O-D pair. In order to do this in a continuous-time context, we require the measure-theoretic analog of the infimum of a set of numbers.  In particular, for  any measurable function $g: ~[t_0,\,t_f]\rightarrow \mathbb{R}$, the {\it essential infimum} of $g(\cdot)$ on $[t_0,\,t_f]$ is given by
\begin{equation}\label{chapVI:essinf}
\underset{s\in[t_0,\,t_f]}{\hbox{essinf}}\left\{g(s)\right\}~\doteq~\sup\left\{ x\in\mathbb{R}:~meas\{s\in [t_0,\,t_f]:~g(s)<x\}~=~0\right\}
\end{equation}
\noindent In other words, the essential infimum is the minimum value of $g(\cdot)$ over all $t$ except a set with zero measure. The essential infimum of the effective travel delays, or simply the minimum travel cost, is expressed as:
\begin{equation}\label{essinfdef1}
v_p(h)~\doteq~\underset{t\in[t_0,\,t_f]}{\hbox{essinf}}\left\{\Psi_p(t,\,h)\right\}~>~0\qquad\forall p\in\mathcal{P}
\end{equation}
\begin{equation}\label{essinfdef2}
v_{ij}(h)~\doteq~\min_{p\in\mathcal{P}_{ij}}\left\{v_p(h)\right\} \qquad \forall \left(i,\,j\right) \in \mathcal{W}
\end{equation}

 The following definition of dynamic user equilibrium is first articulated by \cite{Friesz1993}:
\begin{definition}\label{duedef}
{\bf (SRDT dynamic user equilibrium)} A vector of departure rates $h^{\ast }\in \Lambda$ is a dynamic user equilibrium with simultaneous route and departure time choices if
\begin{equation}\label{chapVI:comp}
h_{p}^{\ast }\left( t\right) >0,~~p\in \mathcal{P}_{ij}~~\Longrightarrow~~ \Psi _{p}(t,\,h^{\ast }) ~=~v_{ij}(h^*)\qquad \hbox{for almost every}~ t\in[t_0,\,t_f]
\end{equation}
We denote this equilibrium by $DUE\big(\Psi,\,\Lambda,\, [t_{0},t_{f}] \big)$.
\end{definition}
\noindent In prose, \eqref{chapVI:comp} means that, within a given O-D pair $(i,\,j)$, if the departure rate along path $p$ at time $t$ is positive, then the corresponding travel cost is equal to the minimum travel cost within this O-D.

 Using measure-theoretic arguments, \cite{Friesz1993} establish that
a dynamic user equilibrium is equivalent to the following variational
inequality under suitable regularity conditions:
\begin{equation} \label{duevi}
\left. 
\begin{array}{c}
\text{find }h^{\ast }\in \Lambda\text{ such that} \\ 
\displaystyle \sum_{p\in \mathcal{P}}\displaystyle \int\nolimits_{t_{0}}^{t_{f}}\Psi _{p}(t,h^{\ast
})(h_{p}(t)-h_{p}^{\ast }(t))dt\geq 0 \\ 
\forall h\in \Lambda
\end{array}%
\right\} VI \big(\Psi, \,\Lambda,\,[t_{0},\, t_{f}] \big) 
\end{equation}
The above variational inequality can be re-written in a more compact form by using the notation introduced in \eqref{ipdef}: 
$$
\left<\Psi(h^*)~,~h-h^*\right>~\geq~0\qquad\forall h\in\Lambda
$$

\subsection{Boundedly rational dynamic user equilibrium (BR-DUE)}

The notion of bounded rationality (BR) is a relaxation of Wardrop's first principle \citep{Wardrop}. The Wardropian principle incorporating the simultaneous route-and-departure-time notion requires identical travel costs among all utilized routes and departure time choices between an origin-destination pair; see \eqref{chapVI:comp}. The boundedly rational dynamic user equilibrium (BR-DUE), on the other hand, requires that the experienced travel costs are within the interval $[v_{ij}(h^*),\,v_{ij}(h^*)+\varepsilon_{ij}]$, where $v_{ij}(h^*)$ denotes the minimum travel cost between O-D pair $(i,\,j)$. The fixed constant $\varepsilon_{ij}\in \mathbb{R}_{+}$ specifies the range of acceptable differences in the travel costs experienced by travelers between O-D pair $(i,\,j)$.

Following closely the definition of DUE \eqref{chapVI:comp}, we present the definition of BR-DUE with route and departure time choices.

\begin{definition}\label{brduedef}{\bf (SRDT BR-DUE)} Given the vector of tolerances $\varepsilon=\big(\varepsilon_{ij}:~(i,\,j)\in\mathcal{W}\big)\in \mathbb{R}_{+}^{|\mathcal{W}|}$, a vector of departure rates $h^*\in\Lambda$ is a boundedly rational dynamic user equilibrium associated with $\varepsilon$\, if for all $(i,\,j)\in\mathcal{W}$,
\begin{equation}\label{brduedef1}
h_p^*(t)~>~0,~~ p\in\mathcal{P}_{ij}~~\Longrightarrow~~ \Psi_p(t,\,h^*)\in [v_{ij}(h^*),\,v_{ij}(h^*)+\varepsilon_{ij}]\qquad \hbox{for almost every} ~t\in[t_0,\,t_f]
\end{equation}
where $v_{ij}(h^*)$ is the essential infimum of the effective path delays between origin-destination pair $(i,\,j)$, and is defined in \eqref{essinfdef1}-\eqref{essinfdef2}. We denote this equilibrium by {\it BR-DUE}$\big(\Psi,\,\varepsilon,\,\Lambda,\, [t_{0},\,t_{f}] \big)$.
\end{definition}

The tolerances in the definition of BR-DUE are given exogenously. They can usually be estimated through some static attributes of travelers between an O-D pair, such as age, gender, and socio-economic status.

\subsection{Variable tolerance boundedly rational dynamic user equilibrium (VT-BR-DUE)}

As we explained earlier in the introduction, the tolerances may sometimes depend on endogenous variables such as the departure rate vector. Further extension can be made regarding their dependencies on the path $p$. Specifically, we allow the tolerances to depend on path $p$ and path departure rates $h$, by introducing 
$$
\varepsilon_{ij}^p(h)\in\mathbb{R}_+,\quad \forall p\in\mathcal{P}_{ij},\quad \forall (i,\,j)\in\mathcal{W},\quad\forall h\in\Lambda,
$$
\noindent and letting $\varepsilon(h)\doteq\left(\varepsilon_{ij}^p(h):\,p\in\mathcal{P}_{ij},\,(i,\,j)\in\mathcal{W}\right)$ be the concatenation of such tolerances. Note that $\varepsilon(\cdot)$ is viewed as a mapping from $\Lambda$, the set of feasible path departure rates, to $\mathbb{R}_+^{|\mathcal{P}|}$.

\begin{definition}\label{vtbrduedef}{\bf (SRDT VT-BR-DUE)}
Given the vector of variable tolerances $\varepsilon(\cdot): \Lambda\rightarrow\mathbb{R}_+^{|\mathcal{P}|}$, a vector of departure rates $h^*\in\Lambda$ is a VT-BR-DUE associated with $\varepsilon(\cdot)$ if, for all $(i,\,j)\in\mathcal{W}$,
\begin{equation}\label{vtbrduedef1}
h_p^*(t)~>~0,~~ p\in\mathcal{P}_{ij}~~\Longrightarrow~~ \Psi_p(t,\,h^*)\in [v_{ij}(h^*),\,v_{ij}(h^*)+\varepsilon_{ij}^p(h^*)]\qquad \hbox{for almost every} ~t\in[t_0,\,t_f]
\end{equation}
where $v_{ij}(h^*)$ is the essential infimum of the effective path delays between origin-destination pair $(i,\,j)$, and is defined in \eqref{essinfdef1}-\eqref{essinfdef2}. We denote this equilibrium by {\it VT-BR-DUE}$\big(\Psi,\,\varepsilon,\,\Lambda,\, [t_{0},\,t_{f}] \big)$.
\end{definition}

\begin{remark}\label{rmk1}
It can be easily seen that the BR-DUE is just one special case of the VT-BR-DUE problem, in which the dependence of $\varepsilon_{ij}^p(h^*)$ on $p$ and $h^*$ are dropped. Thus, to simultaneously analyze both models  using the proposed methodological framework, it suffices for us to treat the VT-BR-DUE only, and the established results will automatically hold for the BR-DUE problem.
\end{remark}

\section{The VI formulation for the SRDT VT-BR-DUE and BR-DUE problems}\label{secbrdue}

In this section, we present the infinite-dimensional variational inequality (VI) formulation  for the BR-DUE and VT-BR-DUE problems, when both route choice and departure time choice are considered. These original formulations will benefit the  analyses presented later on the existence and computation of these two models. For the reason stated in Remark \ref{rmk1}, we will state and prove the VI formulation for the VT-BR-DUE problem, followed by a corollary that expresses the VI formulation for the BR-DUE problem.

Essential to the variational inequality formulation of the VT-BR-DUE problem is the following new operator: 
\begin{equation}\label{Phidef}
\Phi^{\varepsilon}: ~ \Lambda~\rightarrow~ \big(L_+^2[t_0,\,t_f]\big)^{|\mathcal{P}|},\qquad h~\mapsto~\big(\Phi_p^{\varepsilon}(\cdot,\,h):~p\in\mathcal{P}\big)
\end{equation}
where 
\begin{equation}\label{Phidef1}
\Phi^{\varepsilon}_p(t,\,h)~=~\max\left\{\Psi_p(t,\,h),~ v_{ij}(h)+\varepsilon_{ij}^p(h)\right\}-\left(\varepsilon_{ij}^p(h)-\min_{q\in\mathcal{P}_{ij}}\left\{\varepsilon_{ij}^q(h)\right\}\right)\qquad \forall p\in\mathcal{P}_{ij}
\end{equation}

\noindent Given any vector of path departure rates $h\in\Lambda$, by performing the dynamic network loading procedure, one obtains the effective path delays $\Psi_p(t,\,h), \forall t\in[t_0,\,t_f],\,\forall p\in\mathcal{P}$. Thus the essential infimum $v_{ij}(h)$ can be determined through definition \eqref{essinfdef1}-\eqref{essinfdef2}. Moreover, the path-specific variable tolerances $\varepsilon_{ij}^p(h)$ can be also determined with the given $h$.  Consequently, one can readily construct the quantities $\Phi^{\varepsilon}_p(t,\,h),\,\forall t\in[t_0,\,t_f],\,\forall p\in\mathcal{P}$ according to \eqref{Phidef1}. Thus, the operator $\Phi^{\varepsilon}$ stated above can indeed be viewed as a mapping from $\Lambda$ into $\big(L_+^2[t_0,\,t_f]\big)^{|\mathcal{P}|}$. We indicate the dependence of such an operator on the variable tolerances $\varepsilon(\cdot)$ by a superscript. More rigorous analysis regarding the existence, well-definedness, and continuity of this new operator will be presented in Section \ref{secpropoperator}.

Theorem \ref{vtbrduevithm} below casts the VT-BR-DUE problem as an infinite-dimensional variational inequality, in the most generic form. In comparison with the VI formulation \eqref{duevi} of DUE, this new variational inequality relies on the new principal operator $\Phi^{\varepsilon}$, which encapsulates both the DNL procedure and the variable tolerances. Notably, the such a canonical formulation allows known methodologies regarding variational inequalities to be directly applied to VT-BR-DUE problems.

\begin{theorem}\label{vtbrduevithm}{\bf (VT-BR-DUE equivalent to a variational inequality)} 
Given $\varepsilon(\cdot): \Lambda\rightarrow\mathbb{R}_+^{|\mathcal{P}|}$, define $\Phi^{\varepsilon}(\cdot)$ according to \eqref{Phidef}-\eqref{Phidef1}. Then, a vector of path departure rates $h^*\in\Lambda$ is a VT-BR-DUE solution if and only if it solves the following variational inequality.
\begin{equation}\label{newvtbrduevi}
\left.
\begin{array}{c}
\hbox{find}~h^*\in\Lambda~~\hbox{such that}
\\
\displaystyle \sum_{p\in\mathcal{P}}\int_{t_0}^{t_f}\Phi^{\varepsilon}_p(t,\,h^*)(h_p(t)-h_p^*(t))\,dt~\geq~0
\\
\forall~h\in\Lambda
\end{array}
\right\}VI\big(\Phi^{\varepsilon},\,\Lambda,\,[t_0,\,t_f]\big)
\end{equation}
\end{theorem}
\begin{proof}
The proof is postponed until  \ref{secapp1}.
\end{proof}

As a special case of Theorem \ref{vtbrduevithm}, we present the VI formulation for the BR-DUE problem with exogenously given tolerances $\varepsilon=\big(\varepsilon_{ij}:\,(i,\,j)\in\mathcal{W}\big)\in\mathbb{R}_+^{|\mathcal{W}|}$.

\begin{corollary}\label{brduevilemma}{\bf (BR-DUE equivalent to a variational inequality)}
Given the fixed tolerance vector $\varepsilon=\big(\varepsilon_{ij}:\,(i,\,j)\in\mathcal{W}\big)\in\mathbb{R}_+^{|\mathcal{W}|}$, define 
\begin{equation}\label{smallphidef}
\phi^{\varepsilon}_p(t,\,h)~\doteq~\max\left\{\Psi_p(t,\,h),\,v_{ij}(h)+\varepsilon_{ij}\right\}\qquad \forall p\in\mathcal{P}_{ij},\quad\forall h\in\Lambda
\end{equation}
Then, a path departure rate vector $h^*\in\Lambda$ is a BR-DUE solution if and only if  it solves the following variational inequality
\begin{equation}\label{newbrduevi}
\left.
\begin{array}{c}
\hbox{find}~h^*\in\Lambda~~\hbox{such that}
\\
\displaystyle \sum_{p\in\mathcal{P}}\int_{t_0}^{t_f}\phi^{\varepsilon}_p(t,\,h^*)(h_p(t)-h_p^*(t))\,dt~\geq~0
\\
\forall~h\in\Lambda
\end{array}
\right\}VI\big(\phi^{\varepsilon},\,\Lambda,\,[t_0,\,t_f]\big)
\end{equation}
\end{corollary}

\begin{proof}
In the case with fixed tolerances $\varepsilon_{ij}$ for all $(i,\,j)\in\mathcal{W}$, the operator $\Phi_p^{\varepsilon}$ defined in \eqref{Phidef1} reduces to \eqref{smallphidef}, and \eqref{newvtbrduevi} yields \eqref{newbrduevi}.
\end{proof}

We provide below an intuitive and graphical illustration of the VI formulations for DUE and BR-DUE. For simplicity, we consider just one O-D pair with one path $p$.  Figure \ref{figDUEVI} depicts a DUE solution $h^*_p(\cdot)$ and the corresponding effective path delay $\Psi_p(\cdot, \, h_p^*)$, both as functions of departure time $t$. The VI formulation for DUE is equivalent to 
$$
\int_{t_0}^{t_f}\Psi_p(t,\,h^*) h_p(t)\,dt~\geq~\int_{t_0}^{t_f}\Psi_p(t,\,h^*) h^*_p(t)\,dt\qquad\forall h\in\Lambda
$$
In other words, $h^*$ is the minimizer of 
\begin{equation}\label{minimizeqt1}
\int_{t_0}^{t_f} \Psi_p(t,\,h^*)\,h_p(t)\,dt 
\end{equation}
among all $h_p(\cdot)$ that satisfies $\int_{t_0}^{t_f}h_p(t)\,dt=Q_{ij}$. Therefore, the part where $h^*_p(\cdot)>0$ should be located within the ``flat bottom" of the effective delay curve; see Figure \ref{figDUEVI}. This is consistent with the definition of DUE.

On the other hand, the BR-DUE condition requires that whenever the solution $h^*_p(t)$ is non-zero,  $\Psi_p(t,\,h^*)$ must be within an indifference band $[v_{ij}(h^*),\,v_{ij}(h^*)+\varepsilon_{ij}]$. This means that the positive portion of $h_p^*(\cdot)$ must reside within the time interval $[a,\,b]$ (see Figure \ref{figBRDUEVI}, left). This immediately implies that the graph of $h_p^*(\cdot)$ must reside within the ``flat bottom" of the function $\phi_p^{\varepsilon}(\cdot,\,h^*)$ defined in \eqref{smallphidef}; see the right part of Figure \ref{figBRDUEVI}. This leads to a new VI whose principal operator is $\phi_p^{\varepsilon}(\cdot,\, h^*)$.

The case with the variable tolerances can be interpreted in a similar way, but is more difficult to visualize and is omitted from this paper.

\begin{figure}[h!]
\centering
\includegraphics[width=0.6\textwidth]{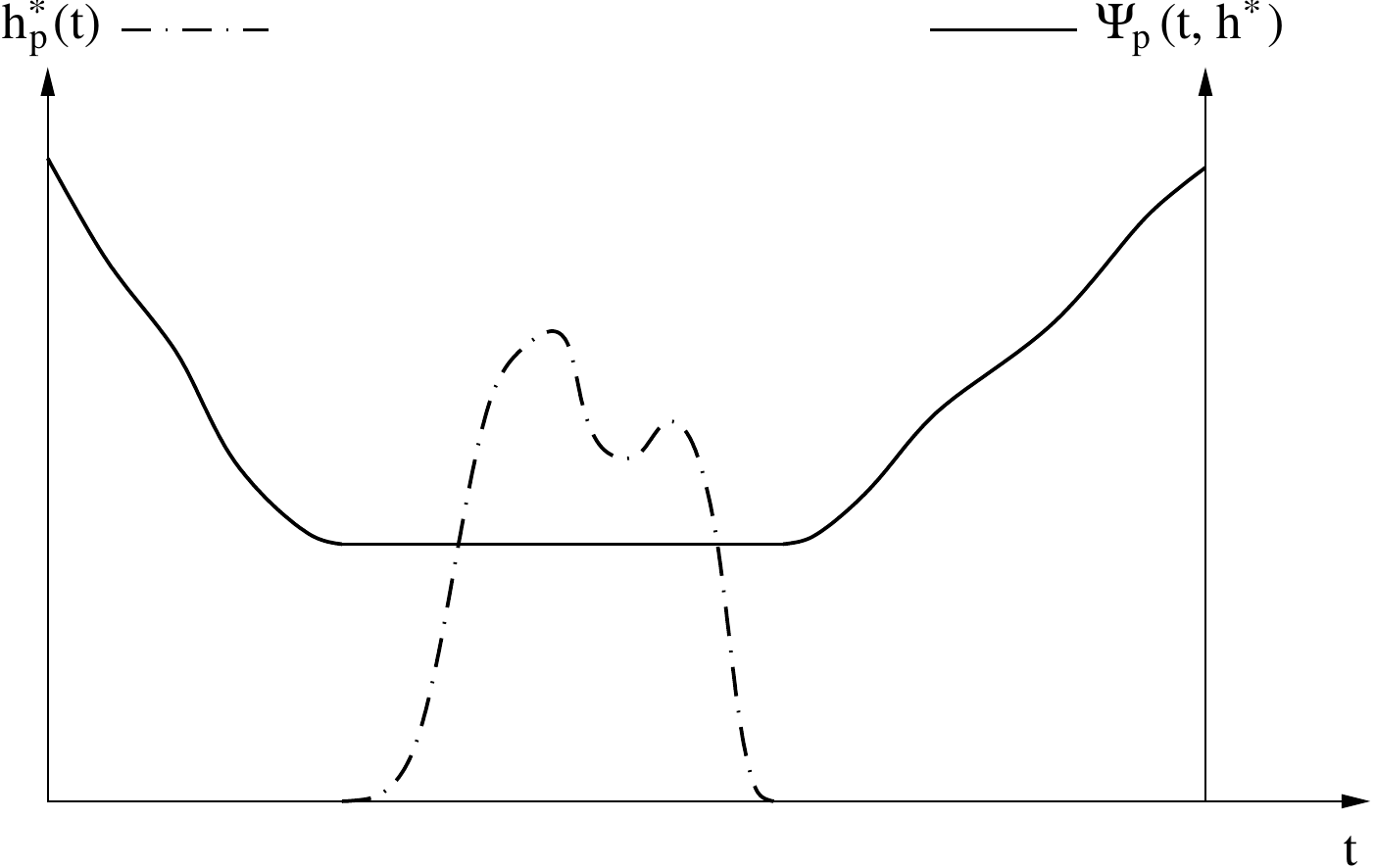}
\caption{\small An illustration of a DUE solution $h_p^*(\cdot)$ and the associated effective path delay $\Psi_p(\cdot,\,h^*)$. In order to minimize the quantity \eqref{minimizeqt1} over all $h_p(\cdot)\in\Lambda$, the equilibrium solution $h^*_p(\cdot)$ must be located at the ``flat bottom" of the effective delay curve. This coincides with the definition of DUE.}
\label{figDUEVI}
\end{figure}

\begin{figure}[h!]
\centering
\includegraphics[width=\textwidth]{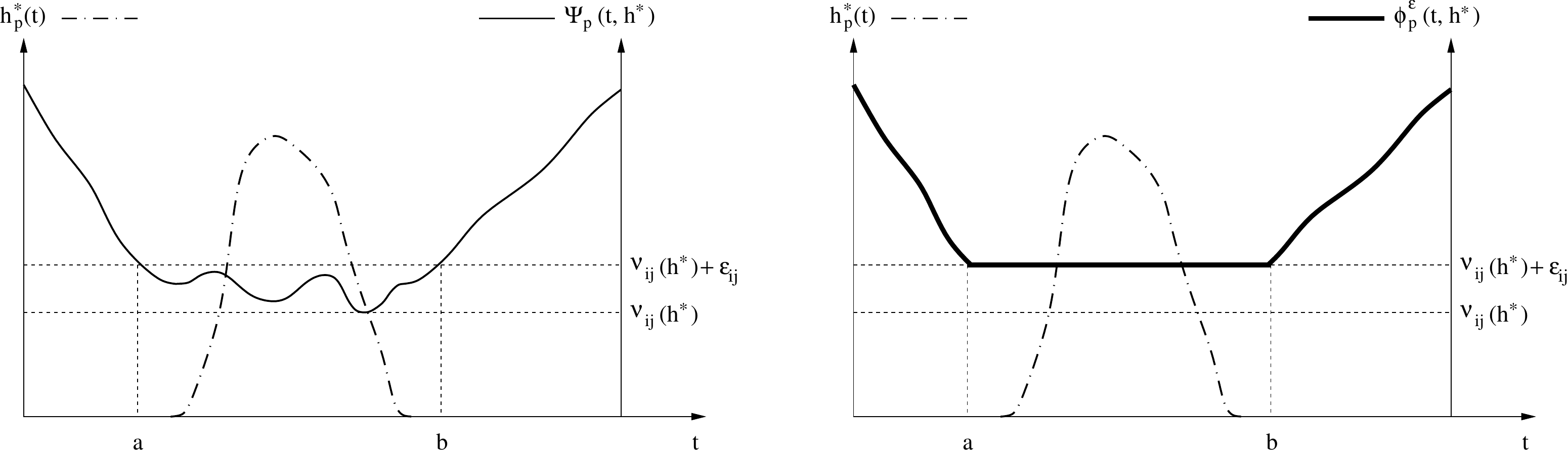}
\caption{\small An illustration of a BR-DUE solution $h_p^*(\cdot)$ and the associated effective path delay $\Psi_p(\cdot,\,h^*)$ (left). To satisfy the BR-DUE condition, $h_p^*(\cdot)$ must reside within the time interval $[a,\,b]$ (left). On the other hand, the function $\phi_p^{\varepsilon}(t,\,h^*)$, defined in \eqref{smallphidef}, is represented by the thick black curve in the figure on the right. Thus, a departure rate $h_p^*(\cdot)$ satisfies the BR-DUE condition if and only if it is located at the ``flat bottom" of the curve $\phi_p^{\varepsilon}(\cdot,\,h^*)$. In view of Figure \ref{figDUEVI}, we see that the BR-DUE condition can be alternatively expressed as a variational inequality with the principal operator $\phi_p^{\varepsilon}(\cdot,\, h^*)$.}
\label{figBRDUEVI}
\end{figure}

\section{Properties of the new operators}\label{secpropoperator}
The proposed variational inequality formulation relies on the new operators $\Phi^{\varepsilon}$ (the VT-BR-DUE case) or $\phi^{\varepsilon}$ (the BR-DUE case). Thus, it is important for us to understand these operators by establishing their qualitative properties required for further analysis and computation. In this section, we will show that the delay operator $\Phi^{\varepsilon}$, given by \eqref{Phidef} and \eqref{Phidef1}, exists, is well defined, and is continuous. Again, these results will automatically hold for $\phi^{\varepsilon}$.

\begin{proposition}\label{propwelldefine}{\bf (Existence and well-definedness of $\Phi^{\varepsilon}$ as an operator)}
Let $\Psi: \Lambda\to \big(L_+^2[t_0,\,t_f]\big)^{|\mathcal{P}|}$ be the effective path delay operator, and $\varepsilon(\cdot)=\big(\varepsilon_{ij}^{p}(\cdot),\,p\in\mathcal{P}_{ij},\,(i,\,j)\in\mathcal{W}\big)$ be the variable tolerance. Assume that the tolerances are uniformly bounded; that is,
\begin{equation}\label{epsilonunib}
0~<~\sup\left\{\varepsilon_{ij}^p(h):~ p\in\mathcal{P}_{ij},~(i,\,j)\in\mathcal{W},\, h\in\Lambda\right\} ~<~+\infty
\end{equation}
Then the operator $\Phi^{\varepsilon}$, given by \eqref{Phidef} and \eqref{Phidef1}, exists and is a well-defined mapping from $\Lambda$ to $\big(L_+^2[t_0,\,t_f]\big)^{|\mathcal{P}|}$.
\end{proposition}
\begin{proof}
The proof is postponed until  \ref{subsecappthmwelldefined}.
\end{proof}

Notice that Proposition \ref{propwelldefine} is based on the well-accepted premise that an effective path delay operator, as defined in \eqref{rvPsidef}-\eqref{cost}, exists. Rigorous existence results for the effective path delay operators are presented by \cite{Friesz1993} for the DNL based on the link delay model; by \cite{existence} for the DNL based on the Vickrey model; by \cite{BH1} and \cite{FHNMY} for the DNL based on the LWR-Lax model; and by \cite{GP} for the DNL based on the classical LWR model.

It remains to show the continuity of $\Phi^{\varepsilon}$, whose proof relies on an alternative expression of the effective delays. This expression is also used in the existence proof provided later in Section \ref{secexistence}. Such an expression is proposed in \cite{existence} but will be recapped here for the completeness of our presentation.

\subsection{Alternative expression of the effective delays}\label{secaltcost}

Let us recall the effective path delay
\begin{equation} \label{ed}
\Psi _{p}(t,\,h)~\doteq ~D_{p}(t,\,h)+f\big(t+D_{p}(t,\,h)-T_{A}%
\big) 
\end{equation}%
\noindent We now re-write \eqref{ed} in a slightly different, yet equivalent, form. Given each O-D pair $(i,\,j)\in \mathcal{W}$, let us
introduce the cost function $\phi _{ij}(\cdot ):[t_{0},\,t_{f}]\rightarrow 
\mathcal{\mathbb{R} }$, which is a function of departure time, and $\,\psi _{ij}(\cdot ):[t_{0},\,t_{f}]\rightarrow \mathcal{\mathbb{R} }$, which is a function of arrival time. Here, the departure time and arrival time are expressed respectively as $t$ and $t+D_p(t,\,h)$. Any driver who departs from the origin at $t$ and arrives at destination at $t+D_p(t,\,h)$ has a travel cost (effective path delay) expressed as 
\begin{equation}\label{moregeneralcost}
\phi_{ij}(t)+\psi_{ij}(t+D_p(t,\,h))
\end{equation}
\noindent where $h\in\Lambda$ is a vector of feasible path departure rates.

The two cost components $\phi_{ij}$ and $\psi_{ij}$ represent the cost for (early) departure and the cost for (late) arrival, respectively. The following mild assumption is made on these two cost functions.\\

\noindent {\bf A0.}  {\it For each $(i,\,j)\in\mathcal{W}$, $\phi _{ij}(\cdot )$ and $\psi
_{ij}(\cdot )$ are continuous on $[t_{0},\,t_{f}]$. $\psi_{ij}(\cdot)$ is monotonically increasing; and $\phi_{ij}(\cdot)$ is Lipschitz continuous with constant $L_{ij}$.}\\

The combination of the two functions, along with assumption {\bf (A0)}, can represent rather general travel cost structures. For example, the effective path delay given in \eqref{ed} is a special case of \eqref{moregeneralcost}, in which 
$$
\phi_{ij}(t)~\doteq~-t,\qquad\psi_{ij}(t+D_p(t,\,h))~\doteq~t+D_p(t,\,h)+f\left(t+D_p(t,\,h)-T_A\right)
$$
Moreover, assumption {\bf A0} is easily satisfied if, in addition, $f'(\cdot)\geq -1$, which is reasonable since the unit cost of early arrival is less than or equal to the unit cost of elapsed travel time \citep{Small}. 

As a second example, the following frequently used form of travel cost 
\begin{equation}\label{abg}
\alpha D_p(t,\,h)+
\begin{cases}
\beta\left(T_A-(t+D_p(t,\,h)\right)\qquad &\hbox{if }~ t+D_p(t,\,h)~\leq~T_A
\\
\gamma\left(t+D_p(t,\,h)-T_A\right)\qquad &\hbox{if }~ t+D_p(t,\,h)~>~T_A
\end{cases}
\end{equation}
where $\gamma>\alpha>\beta>0$, is also a special case of \eqref{moregeneralcost}. This is easily seen by writing 
\begin{align*}
\phi_{ij}(t)&~\doteq~-\alpha t
\\
\psi_{ij}(t+D_p(t,\,h))&~\doteq~\alpha \big(t+D_p(t,\,h)\big)+
\begin{cases}
\beta\left(T_A-(t+D_p(t,\,h)\right)\qquad &\hbox{if }~ t+D_p(t,\,h)~\leq~T_A
\\
\gamma\left(t+D_p(t,\,h)-T_A\right)\qquad &\hbox{if }~ t+D_p(t,\,h)~>~T_A
\end{cases}
\end{align*}
\noindent One can easily check that assumption {\bf A0} is automatically satisfied.

\subsection{Continuity of the operator $\Phi^{\varepsilon}$}

The next theorem establishes the continuity of the new operator $\Phi^{\varepsilon}$ when viewed as a mapping from $\Lambda$ to $\big(L_+^2[t_0,\,t_f]\big)^{|\mathcal{P}|}$. Our argument relies on the already established results regarding the continuity of the effective delay operator $\Psi$ when certain link dynamics and network model are employed. These  include the continuity result for the {\it link delay model} \citep{Friesz1993} established in \cite{ldm}; the continuity result for the {\it Vickrey model} \citep{Vickrey}  established in \cite{existence}; the continuity result for the LWR-Lax model without spillback \citep{FHNMY} established in \cite{BH1}; and the continuity result for the LWR model with spillback established in \cite{LWRcont}. Therefore, Theorem \ref{propcontphi} below, at the very least, applies to these traffic network models. If other types of flow dynamics and/or network models are considered, new continuity result for $\Psi$ needs to be shown in order for Theorem \ref{propcontphi} to hold; but this is beyond the scope of this paper.

\begin{theorem}{\bf (Continuity of $\Phi^{\varepsilon}$ as an operator)}\label{propcontphi}
Assume that the effective path delay operator $\Psi: \Lambda\to \big(L^2_+[t_0,\,t_f]\big)^{|\mathcal{P}|}$ is continuous. In addition, let the functionals $\varepsilon_{ij}^p(\cdot): \Lambda\to \mathbb{R}_{++}$ be continuous $\forall p\in\mathcal{P}_{ij},\, \forall(i,\,j)\in\mathcal{W}$. Then the operator 
\begin{equation}\label{Phidef2}
\Phi^{\varepsilon}: ~ \Lambda~\rightarrow~ \big(L_+^2[t_0,\,t_f]\big)^{|\mathcal{P}|},\qquad h~\mapsto~\big(\Phi_p^{\varepsilon}(\cdot,\,h):~p\in\mathcal{P}\big)
\end{equation}
where 
\begin{equation}\label{Phidef3}
\Phi^{\varepsilon}_p(t,\,h)~=~\max\left\{\Psi_p(t,\,h),~ v_{ij}(h)+\varepsilon_{ij}^p(h)\right\}-\left(\varepsilon_{ij}^p(h)-\min_{q\in\mathcal{P}_{ij}}\left\{\varepsilon_{ij}^q(h)\right\}\right)\qquad \forall p\in\mathcal{P}_{ij}
\end{equation}
is continuous.
\end{theorem}
\begin{proof}
The proof is postponed until  \ref{subsecappcontproof}.
\end{proof}

\begin{remark}\label{rmkcontinuity}
Theorem \ref{propcontphi} highlights the fact that the continuity of $\Phi^{\varepsilon}$ or $\phi^{\varepsilon}$ is more general than the continuity of the effective path delay operator $\Psi$, simply because the latter implies the former. The reverse is false, and one can easily envisage situations where the effective delay operator is discontinuous but the corresponding $\Phi^{\varepsilon}$ and $\phi^{\varepsilon}$ are continuous. Consider, for example, a sequence of points $h^{(n)}\to h^*$ but $\displaystyle\lim_{n\to+\infty}\Psi(h^{(n)}) \neq \Psi(h^*)$. If the tolerance $\varepsilon$ is large enough to cover such a jump discontinuity, the functions $\Phi^{\varepsilon}(h^{(n)})$ can still converge to $\Phi(h^*)$ in the $L^2$-norm. A more concrete example is provided below.
\end{remark}

\begin{example}

\begin{figure}[h!]
\centering
\includegraphics[width=\textwidth]{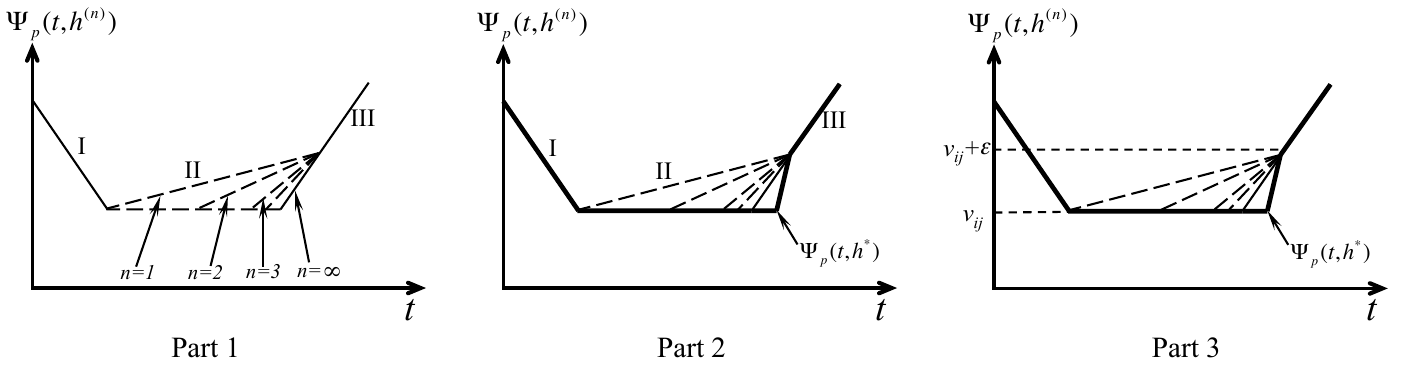}
\caption{\small Illustration that $\Phi^{\varepsilon}$ may be continuous even if $\Psi$ is not.}
\label{figContepsilon}
\end{figure}
In Part 1 of Figure \ref{figContepsilon}, we show the sequence $\Psi_p(\cdot,\,h^{(n)})$ and their limit $\displaystyle\lim_{n\to\infty}\Psi_p(\cdot,\,h^{(n)})$, where each function consists of three components, labeled by I, II and III in the figure. Moreover, these functions share the same components $I$ and $III$, and differ only in $II$. In Part 2, the function $\Psi_p(\cdot,\,h^*)$ is shown with thick line segments, and it is different from $\displaystyle\lim_{n\to\infty}\Psi_p(\cdot,\, h^{(n)})$, indicating a discontinuity of $\Psi$ at the point $h^*$. In Part 3, if all the differences in these  functions are within the indifference band, $\Phi^{\varepsilon}_p(\cdot,\,h^{(n)}),\,\displaystyle\lim_{n\to\infty}\,\Phi^{\varepsilon}_p(\cdot,\,h^{(n)})$, and $\Phi^{\varepsilon}_p(\cdot,\,h^*)$ would coincide. As a result, $\Phi^{\varepsilon}$ is continuous at $h^*$.
\end{example}

This example shows that the continuity of $\Phi^{\varepsilon}$ is indeed a weaker regularity condition than the continuity of $\Psi$.

\section{Existence of VT-BR-DUE and BR-DUE}\label{secexistence}
The existence of VT-BR-DUE or BR-DUE can be conveniently analyzed through their established VI formulations and Browder's fixed-point existence theorem:

\begin{theorem}
\label{mainthm}{\bf \citep{Browder}} Let $K$ be a compact convex subset of the locally convex
topological vector space $E$, $T$ a continuous (single-valued) mapping of $K$
into $E^{\ast }$, the dual space of $E$. Then, there exists $u_{0}$ in $K$ such that 
\begin{equation*}
\Big<T(u_{0}),\,u-u_0\Big>~\geq ~0\qquad \forall u\in K
\end{equation*}
\end{theorem}
\begin{proof}
See \cite{Browder}.
\end{proof}

\begin{remark}
Browder's theorem is one generalization of the famous Brouwer's fixed-point theorem. Existence results for DUE have been established in a number of papers based on other form of Brouwer's fixed-point theorem, either explicitly or implicitly. Those papers include \cite{MW2010, Mounce, Mounce2007, MS2007, WTC} and \cite{ZM}.
\end{remark}

Approaches based on Browder's theorem require the set of feasible path departure rates to be compact and convex in a Hilbert space, and typically involve an {\it a priori} bound on all the path departure rates \citep{ZM}. The invocation of the {\it a priori} boundedness on departure rates is used to secure compactness needed for a topological argument. This is particularly relevant in continuous-time models due to the fact that closedness and boundedness together no longer guarantee compactness in infinite-dimensional spaces.

Note should be taken on the following fact: the boundedness assumption is less of an issue for the route-choice (RC) DUE by virtue of problem formulation; that is, for RC\ DUE, the
travel demand constraints are of the following form: 
\begin{equation}
\sum_{p\in \mathcal{P}_{ij}}h_{p}(t)~=R_{ij}(t)\qquad \forall ~t,\quad
\forall ~(i,\,j)\in \mathcal{W}  \label{introeqn1}
\end{equation}%
where $\mathcal{W}$ is the set of origin-destination pairs, $\mathcal{P}_{ij}
$ is the set of paths connecting $\left( i,j\right) \in \mathcal{W}$ and $
h_{p}(\cdot)$ is the departure rate along path $p$. $R_{ij}(t)$
represents the rate (not volume) at which travelers leave origin $i$ with
the intent of reaching destination $j$ at time $t$; each such trip rate is
assumed to be bounded from above. Since (\ref{introeqn1}) is imposed
point-wise and every path departure rate $h_{p}(\cdot)$ is nonnegative, we are assured that any feasible vector $h(\cdot)=\left( h_{p}(\cdot):~p\in \mathcal{P}_{ij},~\left(i,\,j\right) \in \mathcal{W}\right)$ is uniformly bounded. On the other hand, the
SRDT user equilibrium imposes the following constraints on the path departure rates: 
\begin{equation}
\sum_{p\in \mathcal{P}_{ij}}\int_{t_{0}}^{t_{f}}h_{p}(t)\,dt~=~Q_{ij}\qquad
\forall ~(i,\,j)\in \mathcal{W}  \label{introeqn2}
\end{equation}%
where $Q_{ij}\in \mathcal{\mathbb{R} }_{++}^{1}$ is the volume (not rate) of
travelers departing node $i$ with the intent of reaching node $j$. The
integrals in (\ref{introeqn2}) are interpreted as Lebesgue; hence, \eqref
{introeqn2} alone is not enough to assure bounded path departure rates. This
observation has been the major hurdle to proving existence without the 
{\it a priori} invocation of bounds on the path departure rates.

The existence of a dynamic user equilibrium trivially implies the existence of a boundedly rational dynamic user equilibrium. Therefore, a meaningful discussion of the existence of VT-BR-DUEs or BR-DUEs, as we aim in this paper, should rely on assumptions weaker than those made for normal DUEs.  In particular, we will first review the existence conditions for SRDT DUE \citep{existence}, which are arguably the most general existence conditions in the literature. In fact, as we later show, a violation of these conditions may result in the non-existence of SRDT DUE. Then, we will show the existence of VT-BR-DUE under conditions weaker than those assumed by \cite{existence} for the SRDT DUE case, thereby establishing that the existence of VT-BR-DUE is indeed more general than DUE.

\subsection{Review of existence results for SRDT DUE}

The existence of SRDT DUE is shown in \cite{existence}. In that paper, the effective path delay is expressed in the same way as we did in Section \ref{secaltcost}. That is, for any $(i,\,j)\in\mathcal{W}$, 
$$
\Psi_p(t,\,h)~=~\phi_{ij}(t)+\psi_{ij}\left(t+D_p(t,\,h)\right)\qquad\forall p\in\mathcal{P}_{ij}
$$
where $t$ denotes the departure time, and $t+D_p(t,\,h)$ represents the arrival time. The following sufficient conditions for the existence of DUE have been established. 
\\
\noindent {\bf A1.} For each $(i,\,j)\in\mathcal{W}$, $\phi_{ij}(\cdot)$ and $\psi_{ij}(\cdot)$ are continuous on $[t_0,\,t_f]$. $\phi_{ij}(\cdot)$ is monotonically decreasing while $\psi_{ij}(\cdot)$ is monotonically increasing. Moreover, $\phi_{ij}(\cdot)$ is Lipschitz continuous with constant $L_{ij}$; and there exists $\Delta_{ij}>0$ such that 
\begin{equation}\label{a1eqn}
\psi_{ij}(t_2)-\psi_{ij}(t_1)~\geq~\Delta_{ij}(t_2-t_1)\qquad\forall t_0~\leq~t_1~<~t_2~\leq~t_f
\end{equation}

\noindent {\bf A2.} Each link $a\in\mathcal{A}$ of the network has a finite exit flow capacity $M_a<\infty$.\\

\noindent {\bf A3.} The effective delay operator $\Psi$ is continuous from $\Lambda$ into $\left(L^2_+[t_0,\,t_f]\right)^{|\mathcal{P}|}$.

\begin{remark}
 Assumption {\bf A1} holds true for the cost function \eqref{abg} with $\gamma>\alpha>\beta>0$. A proof is provided in \cite{existence}. 
\end{remark}

\begin{remark}
Assumption {\bf A2} obviously applies to a large class of traffic flow models, including the Lighthill-Whitham-Richards model \citep{LW, Richards}, the cell transmission model \citep{CTM1, CTM2}, the link transmission model \citep{LTM, LKWM}, the LWR-Lax model \citep{FHNMY}, and the Vickrey model \citep{Vickrey, GVM1, GVM2}.  However, minor exception of {\bf A2} does exist, e.g. the link delay model proposed by \cite{Friesz1993}. 
\end{remark}

\begin{remark}
Assumption {\bf A3} is the most recognized condition for existence, and it has been shown to be true when related to the link delay model \citep{ldm}, the Vickrey model \citep{existence}, and the LWR model \citep{LWRcont}.  Assumptions {\bf A1}-{\bf A2} together are meant to tackle the issue of compactness mentioned earlier, which would otherwise have to be handled using the {\it ad hoc} boundedness on the path departure rates. 
\end{remark}

In the remainder of this section, we will first illustrate the non-existence or difficulties in proving existence of SRDT DUE without assumptions {\bf A1} and {\bf A2} (Section \ref{subsecDUEnotexist}). Then, we will show the existence of VT-BR-DUE and BR-DUE in the absence of {\bf A1} and {\bf A2} (Section \ref{subsecVTBRDUEexistence}). This allows us to reach the conclusion that the existence of boundedly rational dynamic user equilibria is indeed more general than that of DUEs.

\subsection{Existence of SRDT DUE in violation of {\bf A1} or {\bf A2}}\label{subsecDUEnotexist}

\subsubsection{Non-existence of SRDT DUE with the Vickrey model}\label{subsubsecnoexist}
\cite{BH} show the existence and uniqueness of a Nash-like traffic equilibrium on a network with one link and a single bottleneck. That paper employs a hydrodynamic model with a point-queue at the traffic bottleneck, which subsumes the Vickrey model \citep{Vickrey} as a special case. A closed-form representation of a Nash equilibrium among the drivers is provided based on the following choices of travel cost functions
\begin{equation}\label{costfcnviolation}
\phi(t)~=~-t,\qquad \psi\left(t+D_p(t,\,h)\right)~=~
\begin{cases}
0 \qquad & t+D_p(t,\,h)~\leq~T_A
\\
\left(t+D_p(t,\,h)-T_A\right)^2 \qquad & t+D_p(t,\,h)~>~T_A
\end{cases}
\end{equation}
where $T_A$ is the target arrival time, $t$ denotes departure time, and $t+D_p(t,\,h)$ is the path exit time. The subscript $ij$ is omitted here since there is only one O-D pair in the network. Clearly, the definition of $\psi$ violates assumption {\bf A1} and, in particular, \eqref{a1eqn}. That paper proves the uniqueness of such a Nash equilibrium, in which the path departure rate is not a square-integrable function. Instead, it is a distribution containing a Dirac-delta. In other words, no square-integrable path departure rate can yield equal and minimum cost for all the drivers, thus no solution in the sense of Definition \ref{duedef} exists in this case. Again, this is due to the violation of {\bf A1} in the definition \eqref{costfcnviolation}.

\subsubsection{Difficulty in proving existence for SRDT DUE with the link delay model}\label{subsubsecdiffexistence}

The link delay model (LDM) is originally proposed by \cite{Friesz1993} based on a link travel time function. In particular, it is assumed that the link traversal time, denoted by $D_a(t)$, is expressed as an affine function of the link occupancy $X(t)$ where $t$ denotes the link entry time:
$$
D_a(t)~=~\alpha X(t)+\beta
$$
for some $\alpha,\,\beta>0$. Clearly, the LDM allows arbitrary link exit flow -- a violation of assumption {\bf A2}. As a result, the existence of SRDT DUE with LDM cannot be established in the framework proposed by \cite{existence}. When this happens, one may rely on the  restrictive assumption that all the path departure rates are {\it a priori} bounded in order to prove existence, but such a proof does not guarantee the existence of any DUE solution characterized solely by the set $\Lambda$, which contains unbounded functions.

\subsection{Existence result for VT-BR-DUE}\label{subsecVTBRDUEexistence}

We present our main existence result for the VT-BR-DUE problem in the following theorem. 

\begin{theorem}\label{existencethembrdue}{\bf (Existence of VT-BR-DUE)}
Assume that 
\begin{enumerate}
\item Assumption {\bf A0} holds.
\item The operator $\Phi^{\varepsilon}$: $\Lambda\rightarrow \left(L_+^2[t_0,\,t_f]\right)^{|\mathcal{P}|}$ is continuous.
\item The functionals $\varepsilon^{p}_{ij}(\cdot): \Lambda\rightarrow \mathbb{R}_+$ are continuous and bounded away from zero. That is, there exists $\varepsilon^{min}>0$ such that $\varepsilon_{ij}^p(h)\geq \varepsilon^{min}, \forall h\in\Lambda,\,\forall p\in\mathcal{P}_{ij},\,\forall (i,\,j)\in\mathcal{W}$.
\end{enumerate}
Then, the variable tolerance boundedly rational dynamic user equilibrium with simultaneous route and departure time choices, as defined in Definition \ref{vtbrduedef}, exists.
\end{theorem}
\begin{proof}
The proof is postponed until \ref{secapp4}.
\end{proof}

\begin{remark}\label{rmkmoregeneralassumption}
Theorem \ref{existencethembrdue} is significant in that it relaxes all three conditions {\bf A1} - {\bf A3} required for the existence of normal DUEs. First of all, compared to {\bf A1}, {\bf A0} drops the assumptions that $\phi_{ij}(\cdot)$ is monotonically decreasing and \eqref{a1eqn}. Secondly, assumption {\bf A2} is completely omitted in this theorem. Last but not least, Theorem \ref{existencethembrdue} relies on the continuity of $\Phi^{\varepsilon}$ instead of the continuity of $\Psi$. According to Remark \ref{rmkcontinuity}, the former is more general than the latter.  
\end{remark}

As a side note, item 3 from Theorem \ref{existencethembrdue} is a reasonable assumption necessary for the existence of VT-BR-DUE. Otherwise, if we allow $\varepsilon_{ij}^p(h^{(n)})$ to tend to zero for a sequence of points $h^{(n)}$, then we are back to the normal DUE case and cannot expect a more general existence result to hold.

According to our discussions in Section \ref{subsubsecnoexist} and Section \ref{subsubsecdiffexistence}, the two unresolved DUE existence problems, due to the violation of {\bf A1} and {\bf A2}, respectively, are fully addressed in the BR case by Theorem \ref{existencethembrdue}. Such an observation, along with Remark \ref{rmkmoregeneralassumption}, substantiate our claim that the existence of VT-BR-DUEs is more general than the existence of DUEs; and this paper provides a precise mathematical interpretation of this statement.

\section{Characterization of the solution set of VT-BR-DUE problems}\label{secsolutionchara}

This section provides a mathematical characterization of the solution set of the VT-BR-DUE problems. Theorem \ref{existencethembrdue} shows that such a set is non-empty under mild conditions, which will continue to hold throughout this section.

We will restrict our analysis to a finite-dimensional space (i.e., discrete-time problems) for the following reasons: (1) it is easier to describe and depict various geometric and topological properties in a finite-dimensional space, which is effectively an Euclidean space; (2) working with finite-dimensional cases sheds light on the numerical solutions of VT-BR-DUEs as almost all calculations are done in a discrete-time setting. The discrete-time VT-BR-DUE problem will be rigorously defined in the next section.

\subsection{Discrete-time VT-BR-DUE problem}
We consider, for each $n\geq 1$, a uniform partition of $[t_0,\,t_f]$ into $n$ subintervals $I_1,\,\ldots,\,I_n$. Define 
\begin{equation}\label{dtfeasible}
\Lambda^n~\doteq~\left\{ h\in\Lambda:~~ h_p(\cdot) \hbox{ is constant on }~ I_k,\quad\forall 1\leq k\leq n, \quad \forall p\in\mathcal{P}\right\}\qquad\forall n\geq 1
\end{equation}
\noindent Each $\Lambda^n$ is the intersection of $\Lambda$ and the space of piecewise-constant functions. Notice that each $\Lambda^n$ still consists of continuous-time functions. We next introduce the discrete-time counterpart of the departure rate vectors. Given $n\geq 1$ and $h\in\Lambda^n$, define a vector $\bar h_p\in\mathbb{R}_+^{n}$ such that $\bar h_p(k)=h_p(t),\,t\in I_k$ for all $1\leq k\leq n$. In other words, $\bar h_p$ is the discrete-time departure rate. Naturally, we let $\bar h=\big(\bar h_p,\,p\in\mathcal{P}\big)\in\mathbb{R}_+^{n\times|\mathcal{P}|}$ be the vector of all the discrete-time departure rates. We define the feasible set of the discrete-time path departure rates as
\begin{equation}\label{barLambdadef}
\bar\Lambda^n~\doteq~\left\{\bar h\in \mathbb{R}_+^{n\times|\mathcal{P}|}:~~\delta t\sum_{p\in\mathcal{P}_{ij}}\sum_{k=1}^n \bar h_p(k)  ~=~ Q_{ij}\quad \forall (i,\,j)\in\mathcal{W} \right\}
\end{equation}
\noindent where $\delta t$ is the time step size.

We will next define the effective path delays in discrete time. Given $h\in\Lambda^n$, the corresponding effective path delays $\Psi_p(\cdot,\,h),\,p\in\mathcal{P}$ is, in general, not piecewise constant. We let
$$
{1\over |I_k|}\int_{I_k}\Psi_p(t,\,h)\,dt\qquad\forall 1\leq k\leq n, \quad\forall p\in\mathcal{P}
$$
be the average value of the effective delay on path $p$ corresponding to departure interval $I_k$, where $|I_k|$ is the length of $I_k$.  Let us define the discrete-time effective path delay operator according to the following chain of mappings
$$
\bar h\in \bar\Lambda^n~\rightarrow~ h\in\Lambda^n~\rightarrow~\big(\Psi_p(\cdot,\,h),~p\in\mathcal{P}\big)~\rightarrow~\left({1\over |I_k|}\int_{I_k}\Psi_p(t,\,h)\,dt,~ 1\leq k\leq n, ~p\in\mathcal{P}  \right),
$$
\noindent which defines a mapping 
$$
\bar\Psi: \bar\Lambda^n \to \mathbb{R}_+^{n\times|\mathcal{P}|},\qquad \bar h \mapsto \left(\bar\Psi_p(k,\,\bar h), ~1\leq k\leq n, ~ p\in\mathcal{P}\right)
$$
where 
\begin{equation}\label{barPsipkdef}
\bar\Psi_p(k,\,\bar h)~\doteq~{1\over |I_k|}\int_{I_k}\Psi_p(t,\,h)\,dt
\end{equation}
$\bar \Psi$ is the discrete-time counterpart of the effective path delay operator.

With these preliminaries, we are now ready to introduce discrete-time versions of the DUE and VT-BR-DUE problems. 
\begin{definition}\label{defdtDUE}{\bf (Discrete-time DUE problem)}
For each $n\geq 1$, a vector of path departure rates $\bar h^{*}\in\bar\Lambda^n$ is a solution of the finite-dimensional DUE problem if for any $(i,\,j)\in\mathcal{W}$, 
\begin{equation}
\bar h^{*}_p(k)~>~0,~ p\in\mathcal{P}_{ij}~\Longrightarrow~\bar\Psi_p(k,\,\bar h^{*})~=~v_{ij}(\bar h^*)\qquad\forall 1\leq k\leq n
\end{equation}
\noindent where $\displaystyle v_{ij}(\bar h^*)=\min_{1\leq k\leq n,\,p\in\mathcal{P}_{ij}}\bar\Psi_p(k,\,\bar h^{*})$.
\end{definition}

It is not difficult to see that the problem defined above is equivalent to the following finite-dimensional variational inequality:
\begin{equation}\label{fdduevi}
\left.
\begin{array}{c}
\hbox{find}~\bar h^{*}\in \bar\Lambda^n~~\hbox{such that}
\\
\displaystyle \sum_{p\in\mathcal{P}}\sum_{k=1}^n\bar\Psi_p(k,\,\bar h^{*})(\bar h_p(k)-\bar h_p^{*}(k))~\geq~0
\\
\forall~\bar h\in\bar\Lambda^n
\end{array}
\right\}VI\big(\bar\Psi,\,\bar\Lambda^n,\,[t_0,\,t_f]\big)
\end{equation}
The discrete-time VT-BR-DUE problem is similarly defined. To do this, we let $\bar \varepsilon_{ij}^p(\cdot),\, p\in\mathcal{P}_{ij},\,(i,\,j)\in\mathcal{W}$ be a set of mappings from $\bar\Lambda^n$ into $\mathbb{R}_{++}$, such that $\bar \varepsilon_{ij}^p(\bar h)=\varepsilon_{ij}^p(h)$, where $\bar h\in\bar\Lambda^n$ is the discrete-time counterpart of $h\in\Lambda^n$. 
\begin{definition}\label{defdtVT-BR-DUE}{\bf (Discrete-time VT-BR-DUE problem)}
Given $\bar\varepsilon=\big(\bar\varepsilon^p_{ij}(\cdot):~p\in\mathcal{P}_{ij},\,(i,\,j)\in\mathcal{W}\big)$, a vector $\bar h^{*}\in\bar\Lambda^n$ is a solution of the discrete-time VT-BR-DUE problem if, for any $(i,\,j)\in\mathcal{W}$, 
\begin{equation}
\bar h_p^*(k)~>~0,~ p\in\mathcal{P}_{ij}~\Longrightarrow~\bar\Psi_p(k,\,\bar h^*)\in\left[v_{ij}(\bar h^*),\,v_{ij}(\bar h^*)+\bar\varepsilon_{ij}^p(\bar h^*)\right]
\end{equation}
\noindent where $\displaystyle v_{ij}(\bar h^*)=\min_{1\leq k\leq n,\,p\in\mathcal{P}_{ij}}\bar\Psi_p(k,\,\bar h^{*})$.
\end{definition}
The following finite-dimensional VI characterizes the solution of the discrete-time VT-BR-DUE problem. 
\begin{equation}\label{fdvtbrduevi}
\left.
\begin{array}{c}
\hbox{find}~\bar h^{*}\in \bar\Lambda^n~~\hbox{such that}
\\
\displaystyle \sum_{p\in\mathcal{P}}\sum_{k=1}^n\bar\Phi^{\bar\varepsilon}_p(k,\,\bar h^{*})(\bar h_p(k)-\bar h_p^{*}(k))~\geq~0
\\
\forall~\bar h\in\bar\Lambda^n
\end{array}
\right\}VI\big(\bar\Phi^{\bar\varepsilon},\,\bar\Lambda^n,\,[t_0,\,t_f]\big)
\end{equation}
where 
\begin{equation}\label{barPhidef1}
\bar\Phi^{\bar\varepsilon}_p(k,\,\bar h^*)=\max\left\{\bar\Psi_p(k,\,\bar h^*),~ v_{ij}(\bar h^*)+\bar\varepsilon_{ij}^p(\bar h^*)\right\}-\left(\bar\varepsilon_{ij}^p(\bar h^*)-\min_{q\in\mathcal{P}_{ij}}\left\{\bar\varepsilon_{ij}^q(\bar h^*)\right\}\right)\quad 1\leq k\leq n,~ p\in\mathcal{P}_{ij}
\end{equation}
\noindent which defines the operator $\bar\Phi^{\bar\varepsilon}: \bar\Lambda^n\to \mathbb{R}_+^{n\times|\mathcal{P}|}$, ~$\bar h\mapsto \big(\bar\Phi_p^{\bar\varepsilon}(k,\,\bar h), ~1\leq k\leq n,~p\in\mathcal{P}\big)$.

The following lemma regarding the continuity of the discrete-time operators is useful as it substantiates the assumptions made in subsequent analysis. 
\begin{lemma}\label{lemmabarphicont}{\bf (Continuity of $\bar\Psi$ and $\bar\Phi^{\bar\varepsilon}$)} If $\Psi: \Lambda\to \big(L_+^2[t_0,\,t_f]\big)^{|\mathcal{P}|}$ and $\varepsilon_{ij}^p(\cdot): \Lambda\to \mathbb{R}_+$ are continuous, then so are $\bar\Psi: \bar\Lambda^n \to \mathbb{R}_+^{n\times|\mathcal{P}|}$ and  $\bar\Phi^{\bar\varepsilon}: \bar\Lambda^n \to \mathbb{R}_+^{n\times|\mathcal{P}|}$.
\end{lemma}
\begin{proof}
The proof is postponed until \ref{subsecappbarcont}.
\end{proof}

\subsection{Characterization of the solution set of VT-BR-DUE}

We begin with the first theorem that characterizes the solution set as being closed and bounded in the finite-dimensional Euclidean space $\mathbb{R}^{n\times|\mathcal{P}|}$.

\begin{proposition}\label{thmclosedness}{\bf (Compactness of the solution set)}
Let $\bar\Phi^{\bar\varepsilon}$ be continuous as shown in Lemma \ref{lemmabarphicont}, then all the solutions of the discrete-time VT-BR-DUE problem form a closed and bounded, and thus compact, set in $\bar\Lambda^n\subset \mathbb{R}^{n\times|\mathcal{P}|}$ for every $n\geq 1$.
\end{proposition}
\begin{proof}
The proof is deferred to  \ref{subsecappcompact}
\end{proof}

Having established that for each $n\geq 1$, the discrete-time solution set, denoted by $\Omega^n$, is compact, we would like to further characterize its interior points. An interior point of a set is such that there exists a ball centered at this point that is completely contained in the set. In other words, if a solution $\bar h^*\in\Omega^n$ is an interior point, then one can perturb such a point in any direction by a small amount and find another solution.  Additionally, we say that a point $x$ is an interior point of a set $S$ relative to another set $X\supset S$ if there exists a ball $\mathcal{B}_{x}^{\delta}$ centered at $x$ with radius $\delta>0$ such that $\mathcal{B}_x^{\delta}\cap X\subset S$.

\begin{example}{\bf (Nonexistence of interior points of $\Omega^n$ relative to the whole space)}
We first observe that $\Omega^n$ does not have any interior point relative to the whole space $\mathbb{R}^{n\times|\mathcal{P}|}$.  This is because any ball $\mathcal{B}_{\bar h^*}^{\delta}\subset \mathbb{R}^{n\times|\mathcal{P}|}$ centered at a solution $h^*\in\Omega^n$ with radius $\delta>0$ obviously contains points that violate the demand satisfaction constraints \eqref{barLambdadef} and hence do not belong to the set $\Omega^n$. 
\end{example}

The next question we would like to ask is whether $\Omega^n$ has any interior points relative to $\bar\Lambda^n$. Again, the answer is no, meaning that a point $\bar h\in\bar\Lambda^n$ arbitrarily close to a solution $\bar h^*$ may fail to be a solution. This will be illustrated in the following example.

\begin{example}\label{nointexample}{\bf (Nonexistence of interior points of $\Omega^n$ relative to $\bar\Lambda^n$)}
We consider a network consisting of just one path $p$. Let the time horizon $[t_0,\,t_f]$ be large enough such that for any solution $\bar h_p^*\in\Omega^n$, the effective path delay corresponding to the first time interval is far greater than the experienced effective path delays for utilized departure intervals. That is,
$$
\bar h^*_p(k)~>~0~\Longrightarrow~\bar\Psi_p(k,\,\bar h_p^*)~\ll \bar\Psi_p(1,\,\bar h_p^*)
$$
This corresponds to the situation where a driver departures very early, say at 2am in the morning when the target arrival time at work is 8am; as a result, he/she experiences a great travel cost even though there is little or no congestion. See Figure \ref{figNointerior} for an illustration of this situation. 

To see that a point in $\bar\Lambda^n$ close enough to $\bar h_p^*$ may not be a VT-BR-DUE solution, we simply move some traffic from the departure window of $\bar h_p^*$ to the first time interval, and call the resulting departure pattern $\bar h_p^{**}$, which clearly belongs to $\bar\Lambda^n$ but is not a VT-BR-DUE solution,  despite the fact that it differs from $\bar h_p^*$ by as little as one wants. This shows that $\Omega^n$ has no interior points when viewed as a subset of $\bar\Lambda^n$.
\end{example}

\begin{figure}[h!]
\centering
\includegraphics[width=\textwidth]{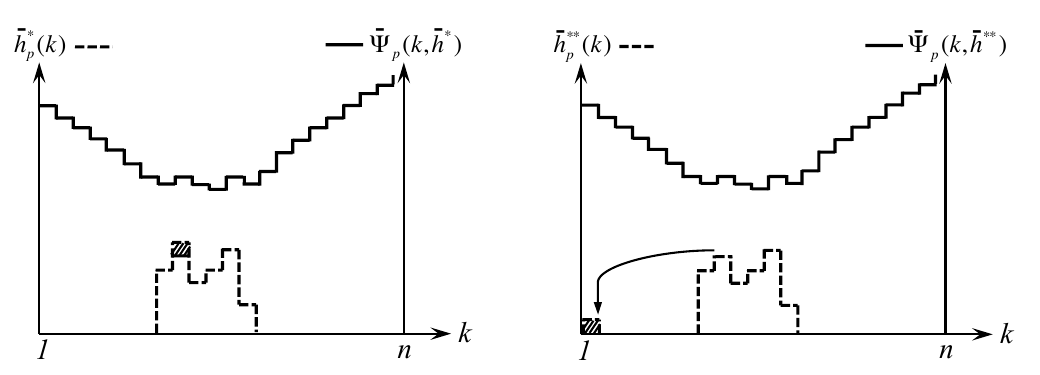}
\caption{\small An illustration that $\Omega^n$ does not have an interior point when viewed as a subset of $\bar\Lambda^n$ (i.e., a point in $\bar\Lambda^n$ arbitrarily close to a solution may fail to be a solution). One can move some traffic from the departure window of a solution $\bar h^*$ to the first time interval, and the resulting departure profile $\bar h^{**}_p\in\bar\Lambda^n$ is not a VT-BR-DUE, despite that $\left\|\bar h^*_p-\bar h^{**}_p\right\|_{2}$ can be arbitrarily small. Here $\|\cdot\|_2$ denotes the standard 2-norm in Euclidean spaces.}
\label{figNointerior}
\end{figure}

Example \ref{nointexample} suggests that, starting from a given solution, even the search is restricted within the feasible set $\bar\Lambda^n$, one cannot always find a solution of the VT-BR-DUE problem. Clearly, additional constraints for the search direction are needed to guarantee one or more solutions are found. The rest of this section is dedicated to the articulation of such conditions and the procedure for finding infinitely many solutions. The following concept turns out to be crucial.

\begin{definition}
A discrete-time VT-BR-DUE solution $\bar h^*$ is said to have the {\bf (P)} property if  for all $(i,\,j)\in\mathcal{W}$, 
$$
\bar h_p^*(k)~>~0,~p\in\mathcal{P}_{ij}~\Longrightarrow \bar\Psi_p(k,\,\bar h^*)~<~v_{ij}(\bar h^*)+\bar\varepsilon_{ij}^p(\bar h^*)\qquad\forall 1\leq k\leq n
$$
where $\displaystyle v_{ij}(\bar h^*)\doteq \min_{1\leq k\leq n,\,p\in\mathcal{P}_{ij}}\bar\Psi_p(k,\,\bar h^*)$ is the minimum effective delay within O-D $(i,\,j)$.
\end{definition}

The {\bf (P)} property means that in a VT-BR-DUE solution no driver experiences a travel cost that reaches his/her maximum tolerance (that is, $v_{ij}(\bar h^*)+\bar \varepsilon_{ij}^p(\bar h^*)$). As a special case of VT-BR-DUE, all normal DUE solutions satisfy the {\bf (P)} property. As we shall show later in Proposition \ref{propsamesupport}, there are in fact infinitely many solutions with the {\bf (P)} property.

We fix an ordering of the finite set $\left\{(I_k,\, p):~1\leq k\leq n, ~p\in\mathcal{P}\right\}$, which is viewed as a bijective mapping $\mathcal{O}$: $\left\{(I_k,\, p):~1\leq k\leq n, ~p\in\mathcal{P}\right\}\to \{1,\,2,\,\ldots,\,n\times|\mathcal{P}|\}$. Simply put, $\mathcal{O}$ assigns a label between $1$ and $n\times|\mathcal{P}|$ to any pair $(I_k,\, p)$, where $I_k$ is some time interval and $p$ is some path.  For any $\bar h^*\in\Omega^n$ that satisfies the {\bf (P)} property, we define  
$$
\mathcal{F}(\bar h^*)~\doteq~\left\{\mathcal{O}(I_k,\,p):~ \bar\Psi_p(k,\,\bar h^*)<v_{ij}(\bar h^*)+\bar\varepsilon_{ij}^p(\bar h^*),~p\in\mathcal{P}_{ij} \right\}~\subset~ \{1,\,2,\,\ldots,\,n\times|\mathcal{P}|\}
$$
\noindent In prose, $\mathcal{F}(\bar h^*)$ identifies the pairs $(I_k,\,p)$ whose corresponding effective path delays are strictly less than the maximum tolerable cost. Due to the fact that $\bar h^*$ satisfies the {\bf (P)} property, $\mathcal{F}(\bar h^*)$ is nonempty. We have the following crucial result.

\begin{proposition}\label{propsamesupport} 
Fix $n\geq 1$, let the effective path delay operator $\bar\Psi$ and the mappings $\bar\varepsilon_{ij}^p(\cdot),\, p\in\mathcal{P}_{ij},\,(i,\,j)\in\mathcal{W}$ be continuous. The following hold.
\begin{enumerate}
\item There exists at least one solution that satisfies the {\bf (P)} property. 
\item  For every solution $\bar h^*$ that satisfies the {\bf (P)} property, there exists $\delta>0$ such that 
$$
\mathcal{B}_{\bar h^*}^{\delta}~\cap~\bar\Lambda^n~\cap~ span\big\{\mathbf{e}_l:~ l\in \mathcal{F}(\bar h^*) \big\} ~\subset~\Omega^n
$$
where $\big\{\mathbf{e}_l\big\}_{l=1}^{n\times|\mathcal{P}|}$ is the natural basis  of $\mathbb{R}^{n\times|\mathcal{P}|}$, and $span\big\{\mathbf{e}_l:~ l\in \mathcal{F}(\bar h^*) \big\}$ is the linear subspace spanned by vectors $\mathbf{e}_l$, $l\in \mathcal{F}(\bar h^*)$. 
\item The set $\mathcal{B}_{\bar h^*}^{\delta}~\cap~\bar\Lambda^n~\cap~ span\big\{\mathbf{e}_l:~ l\in \mathcal{F}(\bar h^*) \big\}$ is infinite, and all points in this set have the {\bf (P)} property.
\end{enumerate}
\end{proposition} 
\begin{proof}
The proof is deferred to \ref{subsecapppropsamesupport}.
\end{proof}
\noindent In Proposition \ref{propsamesupport}, each statement is logically dependent on the preceding statement. We thus present them in the way they are despite the fact that the third statement makes the first one redundant.  

We interpret Proposition \ref{propsamesupport} as follows. It first ensures the existence of at least one solution with the {\bf (P)} property. Then, given a solution $\bar h^*$ with the {\bf (P)} property, one can find infinitely many VT-BR-DUE solutions by searching nearby points in the subset $\bar\Lambda^n\cap span\big\{\mathbf{e}_l:\, l\in \mathcal{F}(\bar h^*) \big\}$. As an immediate corollary, we have the following result by observing that the set $\mathcal{B}_{\bar h^*}^{\delta}~\cap~\bar\Lambda^n~\cap~ span\big\{\mathbf{e}_l:~ l\in \mathcal{F}(\bar h^*) \big\}$  is infinite and convex.

\begin{corollary}
For every solution $\bar h^*$ that satisfies the {\bf (P)} property, there exists an infinite, and convex solution set in $\mathbb{R}^{n\times|\mathcal{P}|}$ that contains $\bar h^*$.
\end{corollary}

We use a three-dimensional space to illustrate the structure of the solution set, although there is no fundamental difficulty to extend what is visualized to a very high-dimensional space. As shown in Figure \ref{figThreed}, only three dimensions are explicitly plotted. The set $\bar\Lambda^n$ is convex shown as the triangular set (which is analogous to a simplex in the three-dimensional space). Let $\bar h^*$ be a solution with the {\bf (P)} property. Assume that $\mathcal{F}(\bar h^*)=\{1,\,2\}$, then the point $\bar h^*$ lies in the plane spanned by $\mathbf{e}_1$ and $\mathbf{e}_2$.  The convex set containing VT-BR-DUE solutions, including $\bar h^*$ itself, is highlighted as the red line segment. All relevant solutions are within $\delta$-distance from $\bar h^*$.

\begin{figure}[h!]
\centering
\includegraphics[width=.55\textwidth]{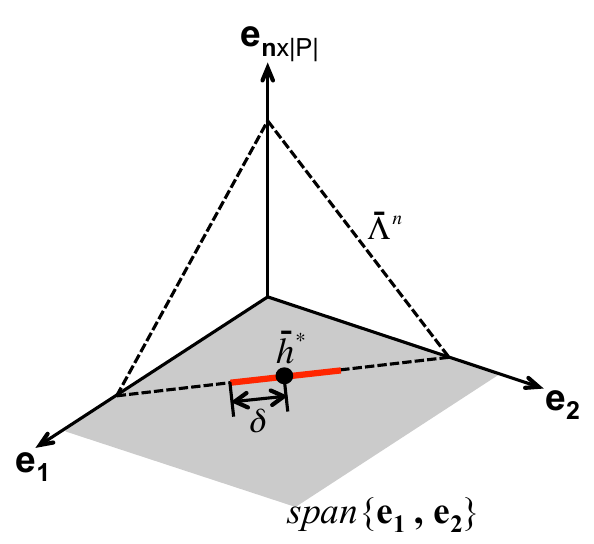}
\caption{\small A visualization of the solution set by searching in the neighborhood of $\bar h^*$ along the direction determined as the intersection of  $\bar\Lambda^n$ and $span\{\mathbf{e}_1,\,\mathbf{e}_2\}$.}
\label{figThreed}
\end{figure}

\subsection{Constructing connected subset of the solution set $\Omega^n$}

Proposition \ref{propsamesupport} suggests a way of expanding the solution set based on a given solution $\bar h^*$, obtained possibly through a particular computational algorithm (computational methods will be introduced in Section \ref{seccomputation}). In this section we will extend such a technique to obtain connected subsets of the solution set. 

To fix the idea, we start with a given solution $\bar h^1$ with the {\bf (P)} property. According to the proof of item 1 from Proposition \ref{propsamesupport} (see \ref{subsecapppropsamesupport}), such a point can be found by solving a modified VT-BR-DUE problem. We then search locally for more solutions that satisfy the {\bf (P)} property according to the procedure described in Proposition \ref{propsamesupport}, and call the resulting infinite and convex solution set be $\mathcal{S}(\bar h^1)$. For every $\bar h^2\in\mathcal{S}(\bar h^1)$, we repeat the same procedure to find $\mathcal{S}(\bar h^2)$. Such a process will continue until no more points can be included in the solution set. This procedure is illustrated in Figure \ref{figSolutionset}.

\begin{figure}[h!]
\centering
\includegraphics[width=.9\textwidth]{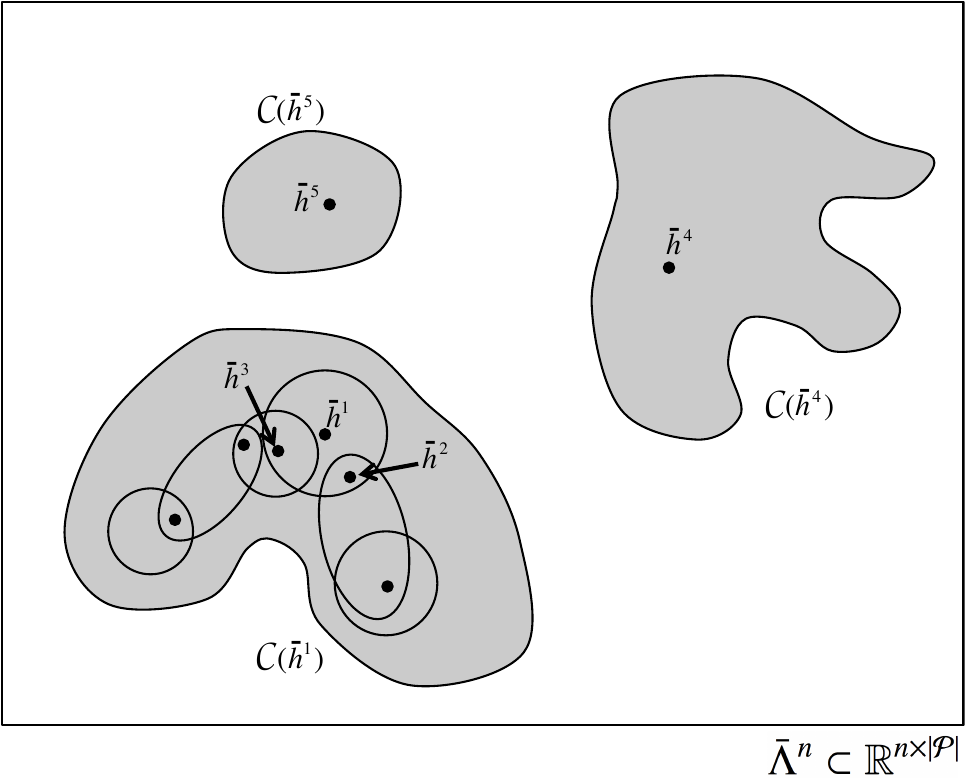}
\caption{\small Constructing connected subsets of the solution set $\Omega^n$. Each connected subset is generated from a single solution with the {\bf (P)} property.}
\label{figSolutionset}
\end{figure}

Once the procedure stops, we call the set of solutions obtained in this way the {\it child set} of $\bar h^*$, denoted by $\mathcal{C}(\bar h^*)$. The next proposition shows that any child set is connected. We begin with a precise mathematical notion of connectedness.

\begin{definition}{\bf (Connected set)}
A connected set is a set that cannot be divided into two nonempty subsets which are open in the relative topology
\end{definition}
\noindent Given a topological space $X$ and a subset $S\subset X$, a subset of $S$ is open in the relative topology if and only if it is an intersection of $S$ with an open set in $X$. 

\begin{proposition}\label{propconnected}
Let $\bar h^*$ be a solution with the {\bf (P)} property, then $\mathcal{C}(\bar h^*)$ is connected.
\end{proposition}
\begin{proof}
The proof is postponed until \ref{subsecapppropconnected}.
\end{proof}

\noindent For a given solution $\bar h^*$ with the {\bf (P)} property, when the procedure of expanding the solution set described above terminates, one can further take the closure of the set $\mathcal{C}(\bar h^*)$ to include all the boundary points since the limit of any converging sequence of solutions is also a solution. This last statement is due to closedness of the solution set (Proposition \ref{thmclosedness}).

The entire solution set may have one or more connected components, and is generally not convex. In order to provide a global characterization of the solution set, more information on the delay operator, such as generalized monotonicity, is required.

\section{Computation of the SRDT VT-BR-DUE and BR-DUE problems}\label{seccomputation}
This section focuses on the computational aspect of the VT-BR-DUE problems. In particular, we present three computational algorithms based on the variational inequality (VI) and {\it differential variational inequality} (DVI) formulations. The first algorithm is a {\it fixed-point} algorithm, which is an immediate consequence of the DVI formulation; the second algorithm is a {\it self-adaptive projection} method, which is adapted from \cite{HL2002} to solve the VI representation of the VT-BR-DUE problem; the third algorithm is called {\it proximal point method} \citep{Konnov}, which relies on successive regularization of the original VI problem to improve convergence. As we shall see later, these three methods rely on increasingly relaxed generalized monotonicity of the principle operator to ensure convergence.

\subsection{Fixed point algorithm based on the DVI formulation}\label{secFPP}
The DVI representation of dynamic user equilibrium with fixed travel demand has been demonstrated by \cite{DODG} and \cite{FKKR} using optimal control theory. Its application to the bounded rationality extension of DUE is straightforward by virtue of the VI formulation proposed in this paper. 

In order to formulate the VT-BR-DUE problem as a DIV, let us introduce the function $y_{ij}(\cdot): ~[t_0,\,t_f]\rightarrow [0,\,Q_{ij}]$ for each O-D pair $(i,\,j)$ where $Q_{ij}$ denotes the fixed travel demand. The quantity $y_{ij}(t)$ represents the total traffic volume that has departed from origin $i$ with the intent of reaching destination $j$ by time $t$. Note that the set of feasible path departure rates $\Lambda$ defined in \eqref{chapVI:lambda} can be equivalently written as a two-point boundary value problem:

\begin{equation}\label{eqn4}
\Lambda_1~\doteq~\left\{h(\cdot)\geq 0: ~~{d y_{ij}(t) \over dt}=\sum_{p\in\mathcal{P}_{ij}}h_p(t),~~ y_{ij}(t_0)~=~0,~~ y_{ij}(t_f)~=~Q_{ij}\quad \forall (i,\,j)\in\mathcal{W}\right\}
\end{equation}

The following equivalence theorem is adapted from the DVI-DUE equivalence theorem \citep{FKKR} by replacing the effective delay operator $\Psi$ with the new operator $\Phi^{\varepsilon}$. Theorem \ref{brduedvithm} is significant in that it relates the VT-BR-DUE problem, together with the BR-DUE problem as a special case, to the still emerging mathematical paradigm of DIVs and computational algorithms associated therein \citep{PS}.

\begin{theorem}\label{brduedvithm}{\bf (VT-BR-DUE equivalent to a DVI and a fixed-point problem)} Given $\varepsilon(\cdot): \Lambda_1\rightarrow \mathbb{R}_+^{|\mathcal{P}|}$, let the operator $\Phi^{\varepsilon}$ be defined via \eqref{Phidef} and \eqref{Phidef1}.  A vector of path departure rates $h^*\in\Lambda_1$ is a VT-BR-DUE if and only if $h^*$ solves the following differential variational inequality 
\begin{equation}\label{brduedvi}
\left.\begin{array}{c}
\hbox{find}~h^*\in\Lambda_1~\hbox{such that}
\\
\displaystyle \sum_{p\in\mathcal{P}}\int_{t_0}^{t_f}\Phi^{\varepsilon}_p(t,\,h^*)(h_p-h_p^*)\,dt~\geq~0
\\
\forall ~h\in\Lambda_1
\end{array}
\right\} DVI\big(\Phi^{\varepsilon},\,\Lambda_1,\,[t_0,\,t_f]\big)
\end{equation}

\noindent Moreover, this DVI problem is equivalent to the following fixed-point problem in a Hilbert space:
\begin{equation}\label{brduefp}
h^*~=~P_{\Lambda_1}\big[h^*-\alpha \Phi^{\varepsilon}(t,\,h^*)\big]
\end{equation}
where $P_{\Lambda_1}[\cdot]$ is the minimum norm projection onto $\Lambda_1$ and $\alpha>0$ is a fixed constant.
\end{theorem}
\begin{proof}
The proof is quite similar to those of Theorem 1 and 2 in \cite{FKKR}, and is omitted here. 
\end{proof}

\noindent The following fixed-point algorithm follows immediately from the fixed-point formulation \eqref{brduefp}. Its derivation relies on explicitly solving a linear-quadratic optimal control problem, which can be found in \cite{FHNMY}. Notice that the following algorithm also applies to BR-DUE with fixed tolerances by replacing $\Phi^{\varepsilon}$ with $\phi^{\varepsilon}$.

\begin{framed}
\noindent {\bf Fixed-point method}
\begin{description} 

\item[Step 0] Identify an initial feasible solution $h^0\in\Lambda_1$.  Set the iteration counter $k=0$. 

\item[Step 1] Solve the dynamic network loading problem with path departure rates given by $h^k$, and obtain the effective path delays $\Psi_p(\cdot,\,h^k),~\forall p\in\mathcal{P}$. Let $v_{ij}(h^k)$ be the minimum effective delay within O-D pair $(i,\,j)$. Then compute: 
$$
\Phi^{\varepsilon}_p\big(t,\,h^k\big)~=~\max\left\{\Psi_p\big(t,\,h^k\big),~ v_{ij}(h^k)+\varepsilon_{ij}^p(h^k)\right\}-\Big( \varepsilon_{ij}^p(h^k)-\min_{q\in\mathcal{P}_{ij}}\big\{\varepsilon_{ij}^q(h^k)\big\} \Big)
$$

\item[Step 2] For each $(i,\,j)\in\mathcal{W}$, solve the following equation for $\mu_{ij}$, using root-search algorithms (here $[x]_+\doteq\max\{0,\,x\}$). 
$$
\sum_{p\in \mathcal{P}_{ij}}\int_{t_0}^{t_f}\left[ h_{p}^{k}(t)
-\alpha \Phi^{\varepsilon} _{p}\big(t,\,h^{k}\big) +\mu_{ij}\right] _{+}dt~=~Q_{ij}
$$
\noindent Then update the next iterate $h^{k+1}=\{h_p^{k+1}: p\in\mathcal{P}\}$ where
$$
h_p^{k+1}(t)~=~\left[ h_{p}^{k}(t)
-\alpha \Phi^{\varepsilon} _{p}\big(t,\,h^{k}\big) +\mu_{ij}\right] _{+}\qquad\forall t\in[t_0,\,t_f],~~ p\in\mathcal{P}_{ij},~~(i,\,j)\in\mathcal{W}
$$

\item[Step 3] Terminate the algorithm with output $h^*\approx h^k$ if
$$
\left\|h^{k+1}-h^{k}\right\|_{L^2}\Big/ \left\|h^k\right\|_{L^2}~\leq~\epsilon
$$
where $\epsilon\in\mathbb{R}_{++}$ is a prescribed termination threshold, and the norm $\|\cdot\|_{L^2}$ is defined in \eqref{ipdef}-\eqref{l2normdef}. Otherwise, set $k=k+1$ and repeat Step 1 through Step 3. 
 
 \end{description}
 \end{framed}

Convergence of the fixed-point algorithm, as in all other algorithms, depends on properties of the principal operator $\Phi^{\varepsilon}$. As pointed out by \cite{FKKR} and \cite{Mounce}, a sufficient condition for convergence of the fixed-point algorithm is strong monotonicity together with Lipschitz continuity of such an operator.  Monotonicity is a relatively strong assumption that may hold in some very specific circumstances \citep{Mounce, PR}. It is likely to fail for general networks and traffic dynamics; for example, non-monotonicity of the delay operator for the Vickrey model is shown in \cite{MS2007}. Thus, the fixed-point method, along with the two methods presented in subsequent sections, should be considered as heuristics when convergence cannot be rigorously assured by the underlying network model. Their effectiveness in computing VT-BR-DUE solutions with satisfactory convergence will be demonstrated in our numerical examples.

To articulate the convergence condition, we introduce the following definitions.
\begin{definition}
The operator $\Phi^{\varepsilon}$ is Lipschitz continuous on $\Lambda_1$ if there exists $K>0$ such that
$$
\left\|\Phi^{\varepsilon}(h^1)  - \Phi^{\varepsilon}(h^2)\right\|_{L^2}~\leq~K \left\|h^1-h^2\right\|_{L^2}\qquad\forall h^1,\,h^2\in\Lambda_1
$$
\end{definition}

\begin{definition}
The operator $\Phi^{\varepsilon}$ is strongly monotone if there exists $\delta >0$ such that
$$
\left<\Phi^{\varepsilon}(h^1) -\Phi^{\varepsilon}(h^2),~~h^1-h^2\right>~\geq~\delta \left\|h^1-h^2\right\|^2_{L^2}\qquad\forall h^1,\,h^2\in\Lambda_1
$$
\end{definition}

\begin{theorem}\label{thmfpaconvgence}{\bf (Convergence of the fixed-point algorithm)}
Let the operator $\Phi^{\varepsilon}$ be Lipschitz continuous and strongly monotone. Then the sequence $\{h^{k}\}\subset\Lambda_1$ generated by the algorithm converge strongly to a VT-BR-DUE solution $h^*$.
\end{theorem}
\begin{proof}
The proof is very similar to that of Theorem 4 in \cite{FKKR}, and is omitted here.
\end{proof}

\subsection{Self-adaptive projection method based on the VI formulation}
We present a self-adaptive projection method for VT-BR-DUE problems based on its VI formulation. Such a projection method is originally proposed by \cite{HL2002} for solving generic VIs, and relies on pseudo monotonicity of the principal operator, which is less restrictive than strong monotonicity.

We begin with some notations. As before, $P_{\Lambda}[\cdot]$ denotes the minimum-norm projection onto the convex set $\Lambda\subset\big(L^2[t_0,\,t_f]\big)^{|\mathcal{P}|}$. Define the residual 
\begin{equation}\label{algresidual}
r(h;\,\beta)~\doteq~h-P_{\Lambda}\left[h-\beta\Phi^{\varepsilon}(h)\right]\qquad h\in\Lambda,~\beta>0
\end{equation}
\noindent Notice that the residual $r(h,\,\beta)$ is  zero if and only if $h$ is a solution of the VI, which is a consequence of the fixed-point equivalence result \eqref{brduefp}.  Given $\alpha,\,\beta>0$, let 
\begin{align}\label{algd}
d(h; \alpha, \beta)~\doteq~&\alpha r(h;\beta) +\beta \Phi^{\varepsilon}\left(h-\alpha r(h;\beta)\right)
\\
\label{algg}
g(h;\alpha,\beta)~\doteq~&\alpha\Big[r(h; \beta)-\beta\Big(\Phi^{\varepsilon}(h)-\Phi^{\varepsilon}(h-\alpha r(h; \beta))\Big)  \Big]
\\
\rho(h;\alpha, \beta)~\doteq~&{\left<r(h;  \beta),\, g(h; \alpha, \beta)\right>\over \left\|d(h; \alpha, \beta)\right\|_{L^2}^2}
\end{align}
where $\left<\cdot,\,\cdot\right>$, defined in \eqref{ipdef}, is the inner product in the space $\big(L^2[t_0,\,t_f]\big)^{|\mathcal{P}|}$ .\\

\begin{framed}
\noindent {\bf Self-adaptive projection method}
\begin{description} 

\item[Step 0] Choose fixed parameters $\mu\in (0,\,1),\,\gamma\in(0,\,2),\,\theta>1$, and $L\in(0,\,1)$. Let $\epsilon>0$ be the termination threshold. Identify an initial feasible solution $h^0\in\Lambda$ and set iteration counter $k=0$. Let $\alpha_k=1$.

\item[Step 1] Set $\beta_k=\min\{1,\,\theta \alpha_k\}$. Compute the residual $r(h^k; \beta_k)$ according to \eqref{algresidual}. If ${\left\|r(h^k; \beta_k)\right\|_{L^2}\over \|h^k\|_{L^2}}\leq \epsilon$, terminate the algorithm; otherwise, continue to Step 2.

\item[Step 2]  Find the smallest non-negative integer $m_k$ such that $\alpha_{k+1}\doteq\beta_k \mu^{m_k}$ satisfies
\begin{equation}\label{algfindint}
\beta_k\left\|\Phi^{\varepsilon}(h^k)-\Phi^{\varepsilon}\big(h^k-\alpha_{k+1}r(h^k; \beta_k)\big)\right\|_{L^2}~\leq~L\left\|r(h^k; \beta_k)\right\|_{L^2}
\end{equation}

\item[Step 3] Compute 
\begin{equation}
h^{k+1}~=~P_{\Lambda}\big[  h^k-\gamma\rho(h^k; \alpha_{k+1}, \beta_k)d(h^k; \alpha_{k+1}, \beta_k)  \big]
\end{equation}
Set $k=k+1$ and go to Step 1.
 
 \end{description}
 \end{framed}

 \eqref{algfindint} requires evaluation of $\Phi^{\varepsilon}$ at the point $h^k-\alpha_{k+1} r(h^k; \beta_k)$. We need to show this point always belongs to $\Lambda$ so that the dynamic network loading procedure can be properly carried out. Indeed, $r(h^k; \beta_k)=h^k-P_{\Lambda}[h^k-\beta_k\Phi^{\varepsilon}(h^k)]$, thus 
$$
h^k-\alpha_{k+1} r(h^k; \beta_k)~=~(1-\alpha_{k+1})h^k+\alpha_{k+1}P_{\Lambda}[h^k-\beta_k\Phi^{\varepsilon}(h^k)]\in \Lambda
$$
\noindent since both $h^k$ and $P_{\Lambda}[h^k-\beta_k\Phi^{\varepsilon}(h^k)]$ belong to the convex set $\Lambda$.

In Step 2 of the self-adaptive projection algorithm,  one is required to test a range of integers, starting from zero, in order to find the smallest integer $m_k$. We show below that such a procedure can always terminate within finite number of trials. Notice that $\Phi^{\varepsilon}$ is a continuous operator and that $\alpha_{k+1} \to 0$ as $m_k \to +\infty$. Thus there exists $N>0$ such that for every $m_k>N$ there holds
$$
\left\|\Phi^{\varepsilon}(h^k)-\Phi^{\varepsilon}(h^k-\alpha_{k+1}r(h^k; \beta_k))\right\|_{L^2}~\leq~{L\epsilon\over \beta_k}~\leq~{L\left\|r(h^k; \beta_k)\right\|_{L^2}\over \beta_k}
$$
\noindent which is \eqref{algfindint}. In case $m_k>1$, the algorithm requires more than one evaluation of the operator (that is, more than one dynamic network loading procedure) within an iteration, which is less efficient than the fixed-point algorithm. However, as we show below, convergence of such an algorithm relies on a weaker form of monotonicity than the fixed-point method.

\begin{definition}{\bf (Pseudo monotonicity)} The operator $\Phi^{\varepsilon}$ is pseudo monotone if, for arbitrary $h^1,\,h^2\in\Lambda$, the following holds
\begin{equation}\label{Phipseudodef}
\big<\Phi^{\varepsilon}(h^2),\,h^1-h^2\big>~\geq~0~\Longrightarrow~\big<\Phi^{\varepsilon}(h^1),\,h^1-h^2\big>~\geq~0
\end{equation}
\end{definition}
It is well-known that pseudo monotonicity is a consequence of monotonicity, and thus is one type of generalized monotonicity.  The convergence of the self-adaptive projection algorithm requires the following property of the operator $\Phi^{\varepsilon}$:
\begin{equation}\label{Phiconvcond}
\big<\Phi^{\varepsilon}(h),\, h-h^*\big>~\geq~0\qquad\forall h\in\Lambda
\end{equation}
where $h^*$ is a solution of the original VI. Notice that \eqref{Phiconvcond} holds if $\Phi^{\varepsilon}$ is monotone or pseudo monotone.

We are now ready to state the convergence result, which is due to  \cite{HL2002}.

\begin{theorem}\label{thmprojalgconv}{\bf (Convergence of the self-adaptive projection method)}.  Let $\Phi^{\varepsilon}:\,\Lambda\to \big(L^2[t_0,\,t_f]\big)^{|\mathcal{P}|}$ be continuous and satisfy \eqref{Phiconvcond}. Then the sequence $\{h^k\}$ generated by the algorithm converge to a solution of the VT-BR-DUE problem.
\end{theorem}
\begin{proof}
The proof follows closely the one in \cite{HL2002}. 
\end{proof}

\subsection{Proximal point method based on the VI formulation}

The proximal point method (PPM) \citep{Konnov} replaces the original VI problem with a sequence of regularized VI problems, each of which may be solved with standard algorithms (such as the projection algorithm) due to improved regularity. The algorithm is summarized as follows.

\begin{framed}
\noindent {\bf Proximal point method}
\begin{description} 

\item[Step 0] Identify an initial feasible solution $h^0\in\Lambda$. Fix a large constant $a>0$ and a parameter $\delta>0$. Set the iteration counter $k=0$. 

\item[Step 1] Solve the following variational inequality for $h^{k+1}$: 
\begin{equation}\label{ppmvi}
\big<\Phi^{\varepsilon}(h^{k+1}) + a(h^{k+1}-h^k)~,~  h - h^{k+1}\big>~\geq~0\qquad\forall h\in\Lambda
\end{equation}

\item[Step 2] Terminate the algorithm if $\|h^{k+1}-h^k\|_{L^2} \leq {\delta\over a D}$, where $D$ is the diameter of the set $\Lambda$.  Otherwise, set $k=k+1$ and repeat Step 1 through Step 2. 
  \end{description}
 \end{framed}

In the above algorithm, the parameter $\delta$ will be used in the convergence analysis.  The key step of the PPM is to solve the VI \eqref{ppmvi}, which  enjoys a significantly improved regularity than the original VI problem. To see this, we rewrite $\Phi^{\varepsilon}(h^{k+1}) + a(h^{k+1}-h^k)$  as $(\Phi^{\varepsilon}+a I)(h^{k+1}) -a h^k$, where $I$ is the identity map. If $\Phi^{\varepsilon}$ is weakly monotone with constant $-K$ \footnote{All Lipschitz continuous operators are weakly monotone; see \cite{E-DUE} for a proof.}, that is,
\begin{equation}\label{weakmono}
\left<\Phi^{\varepsilon}(h^1)-\Phi^{\varepsilon}(h^2)~,~h^1-h^2\right>~\geq~-K \left\|h^1-h^2\right\|_{L^2}^2
\end{equation}
\noindent then $(\Phi^{\varepsilon}+a I)(h^{k+1}) -a h^k$ is a strongly monotone operator acting on $h^{k+1}$ provided that $a>K$. Thus, by choosing $a$ large enough, the VI \eqref{ppmvi} can be solved with any existing algorithm with satisfactory convergence result. 

We now provide a convergence result for the PPM, which is due to \cite{E-DUE}. To prepare for this result, we need the notion of semistrictly quasi monotonicity.

\begin{definition}{\bf (Semistrictly quasi monotone)}
The operator $\Phi^{\varepsilon}$ is quasi monotone if, for arbitrary $h^1,\,h^2\in\Lambda$, 
$$
\left<\Phi^{\varepsilon}(h^2)~,~h^1-h^2\right>>0~\Longrightarrow~\left<\Phi^{\varepsilon}(h^1)~,~h^1-h^2\right>\geq0
$$
The operator $\Phi^{\varepsilon}$ is semistrictly quasi monotone if it is quasi monotone and, for every $h^1,\,h^2\in\Lambda$,  
$$
\left<\Phi^{\varepsilon}(h^2)~,~h^1-h^2\right>>0~\Longrightarrow~\left<\Phi^{\varepsilon}(h^3)~,~h^1-h^2\right>>0 
$$
for some $h^3\in \Big\{h:~h=h^1+\lambda(h^2-h^1),~~\lambda\in(0,\,1/2) \Big\}$.
\end{definition}
The following proposition states that semistrictly quasi monotonicity is weaker than pseudo monotonicity.
\begin{proposition}
If an operator is pseudo monotone, then it is semistrictly quasi monotone.
\end{proposition}
\begin{proof}
See Lemma 3.1 of \cite{Konnov1998}.
\end{proof}
We are now ready to present the convergence result.
\begin{theorem}{\bf (Convergence of the PPM)}\label{ppmthm}
Assume that  $\Phi^{\varepsilon}$ is continuous and semistrictly quasi monotone. Let the set $\Lambda$ be bounded with diameter $D<\infty$. Then for any $\delta>0$, when the PPM algorithm terminates, i.e. when $\|h^{k+1}-h^k\|_{L^2}\leq {\delta\over aD}$ for the first time, then
\begin{equation}\label{ppmconveqn2}
\big<\Phi^{\varepsilon}(h^{k+1})~,~h-h^{k+1}\big>~\geq~-\delta \qquad\forall h\in\Lambda
\end{equation}
\end{theorem}
\begin{proof}
The reader is referred to \cite{E-DUE} for a proof.
\end{proof}

\begin{remark}
Unlike the convergence results established for the previous two methods, which focus on the asymptotic behavior of the generated sequence $\{h^k\}$ as $k\to\infty$,  the convergence result developed in Theorem \ref{ppmthm} is concerned with finding a solution of the approximate VI (that is, with $-\delta$ on the right hand side) within finite iteration. Such a convergence result is quite practical for numerical computations as it searches for a solution of the approximate VI within finite number of iterations.
\end{remark}

\section{Numerical studies}\label{secnumerical}
In this section the three computational methods proposed previously are tested in terms of solution quality, convergence, and computational efficiency. Three test networks with varying sizes are considered. For the dynamic network loading procedure, we employ the link transmission model \citep{LTM} with vehicle spillback explicitly captured. All of the computations reported were performed on a standard laptop with 2.7 GHz processor and 4 GB RAM.

\subsection{The seven-arc, six-node network}
Our first example is meant to illustrate the VT-BR-DUE solution in which the user tolerances are endogenous, depending on the actual (realized) departure rates.  In particular, we will demonstrate how such dependence affects the final outcome of the model.

\begin{figure}[h!]
\centering
\includegraphics[width=.9\textwidth]{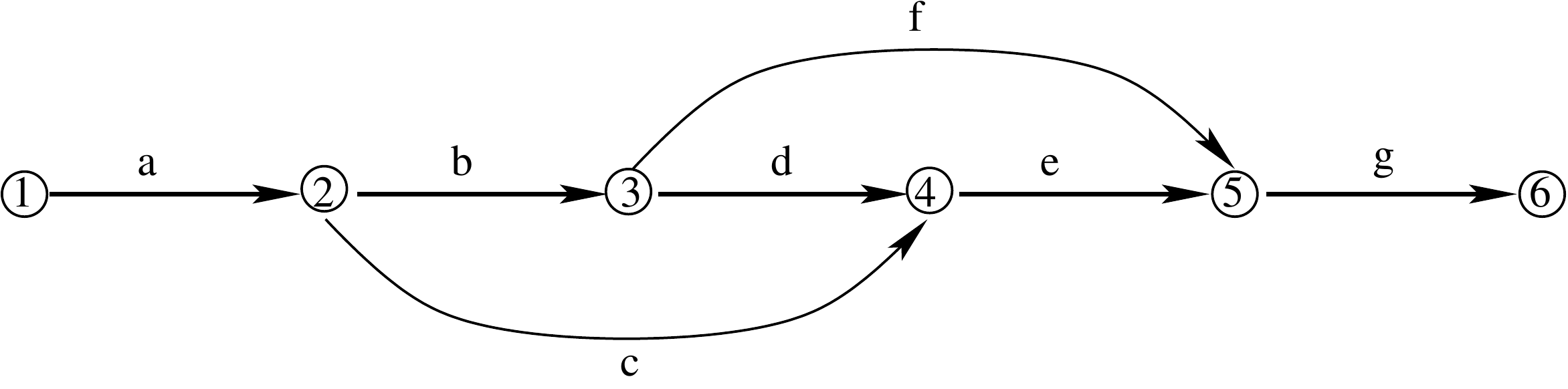}
\caption{\small The seven-arc, six-node test network}
\label{figthreepaths}
\end{figure}

The test network shown in Figure \ref{figthreepaths} has one O-D pair $(1,\,6)$, and three paths $p_1=\{a,\,c,\,e,\,g\}$, $p_2=\{a,\,b,\,d,\,e,\,g\}$, $p_3=\{a,\,b,\,f,\,g\}$.  The fixed travel demand is 2000 (in vehicle). We consider two cases wherein the following two sets of tolerance functions are considered:
\begin{itemize}
\item[] [Case I] ~$\varepsilon_1=\varepsilon_3\equiv 0.1$,~ $\varepsilon_2=0.15\left(1-{100\over 100+V_2}\right)$ 

\item[] [Case II] ~$\varepsilon_1=\varepsilon_3\equiv0.1$, ~$\varepsilon_2=0.2\left(1-{100\over 100+V_2}\right)$
\end{itemize}
\noindent where $\varepsilon_i$ denotes the user tolerance associated with path $p_i$, $i=1,\,2,\,3$; and $V_2\doteq\int_{t_0}^{t_f}h_{p_2}(t)\,dt$ is the total traffic volume on path $p_2$. In both cases, the cost tolerance $\varepsilon_2$ is a functional of $h_{p_2}(\cdot)$; in fact, it is an increasing function of $V_2$. In addition, $\varepsilon_2(\cdot)$ is bounded from above by a fixed constant; see Figure \ref{figepsilonfcn} for these functional forms. Notice that Case II yields a higher tolerance along $p_2$ than Case I, provided the same value of $V_2$. By comparing the two cases, we will show how the slight difference in the tolerance functions manifest itself in the solution. We further note that these chosen functional forms for the tolerances are for illustration purposes only, while further study is clearly needed to formulate and calibrate those specific functional forms in order to accurately capture the network flows. This is beyond the scope of this paper.

\begin{figure}[h!]
\centering
\includegraphics[width=.85\textwidth]{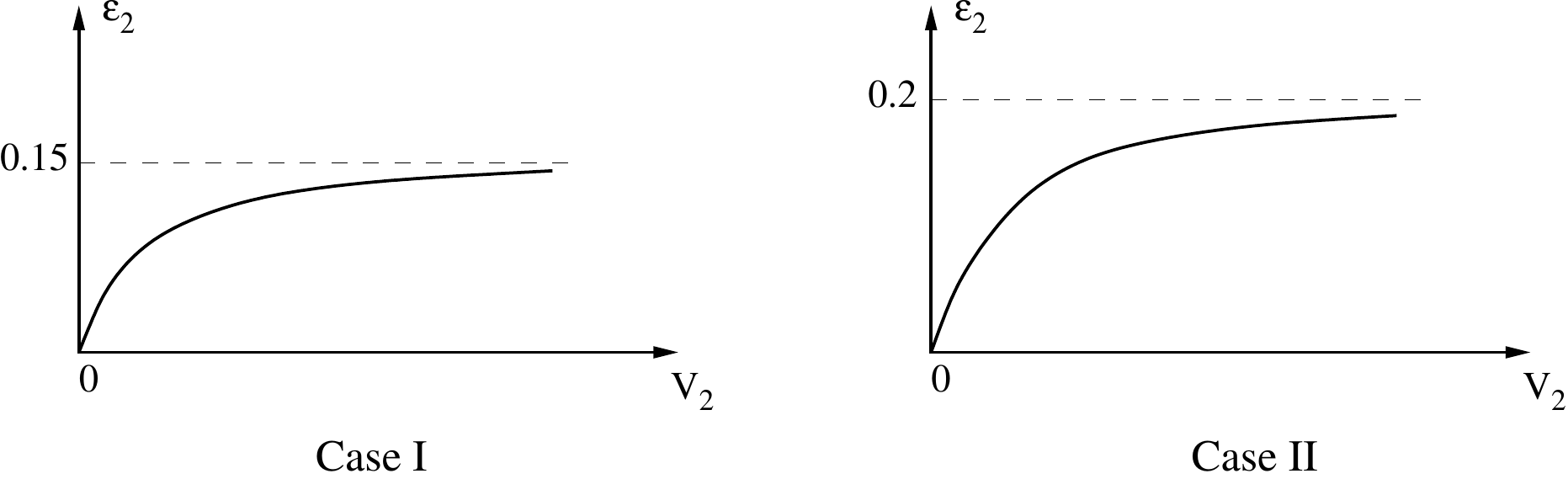}
\caption{\small The VT-BR-DUE problem: the functional forms selected for $\varepsilon_2(\cdot)$.}
\label{figepsilonfcn}
\end{figure}

The VT-BR-DUE problems were solved with the fixed-point algorithm (Section \ref{secFPP}), and the results are displayed in Figure \ref{figsmallsoln}, where we show the departure rates along the three paths and the corresponding effective path delays. In these numerical solutions, the traffic volumes on path $p_2$ are respectively $V_2=1105$ (veh) in Case I and $V_2=1178$ (veh) in Case II. Accordingly, the corresponding tolerances are:
\begin{itemize}
\item[] [Case I] ~$\varepsilon_1=\varepsilon_3=0.1$, ~$\varepsilon_2=0.1376$ 
\item[] [Case II]~ $\varepsilon_1=\varepsilon_3=0.1$, ~$\varepsilon_2=0.1844$ 
\end{itemize}
\noindent The minimum effective delay between the O-D pair, denoted $v_{16}$, and the tolerance thresholds $v_{16}+\varepsilon_i$, $i=1,\,2,\,3$, are shown in Figure \ref{figsmallsoln}. From this figure we see that the computed solutions are indeed solutions of the VT-BR-DUE problems since $h^*_{p_i}(t)>0$ implies that $\Psi(t,\,h^*_{p_i})\leq v_{16}+\varepsilon_i$, $i=1,\,2,\,3$. We also observe that as a result of the chosen forms of $\varepsilon_2(\cdot)$, the total traffic volume on path $p_2$ is smaller in Case I than in Case II. The reason is that, if given the same traffic volume $V_2$, drivers following path $p_2$ have a lower tolerance in Case I than in Case II, and thus more drivers are likely to switch to the other paths in Case I.

Finally, for both computational scenarios, the fixed-point algorithm converged after a finite number of iterations. This can be seen from Figure \ref{figsmallconv}, where the relative gap is expressed by $\left\|h^{k+1}-h^{k}\right\|_{L^2} \big/ \left\|h^k\right\|_{L^2}$.

\begin{figure}[h!]
\centering
\begin{minipage}[c]{0.49\textwidth}
\centering
\includegraphics[width=\textwidth]{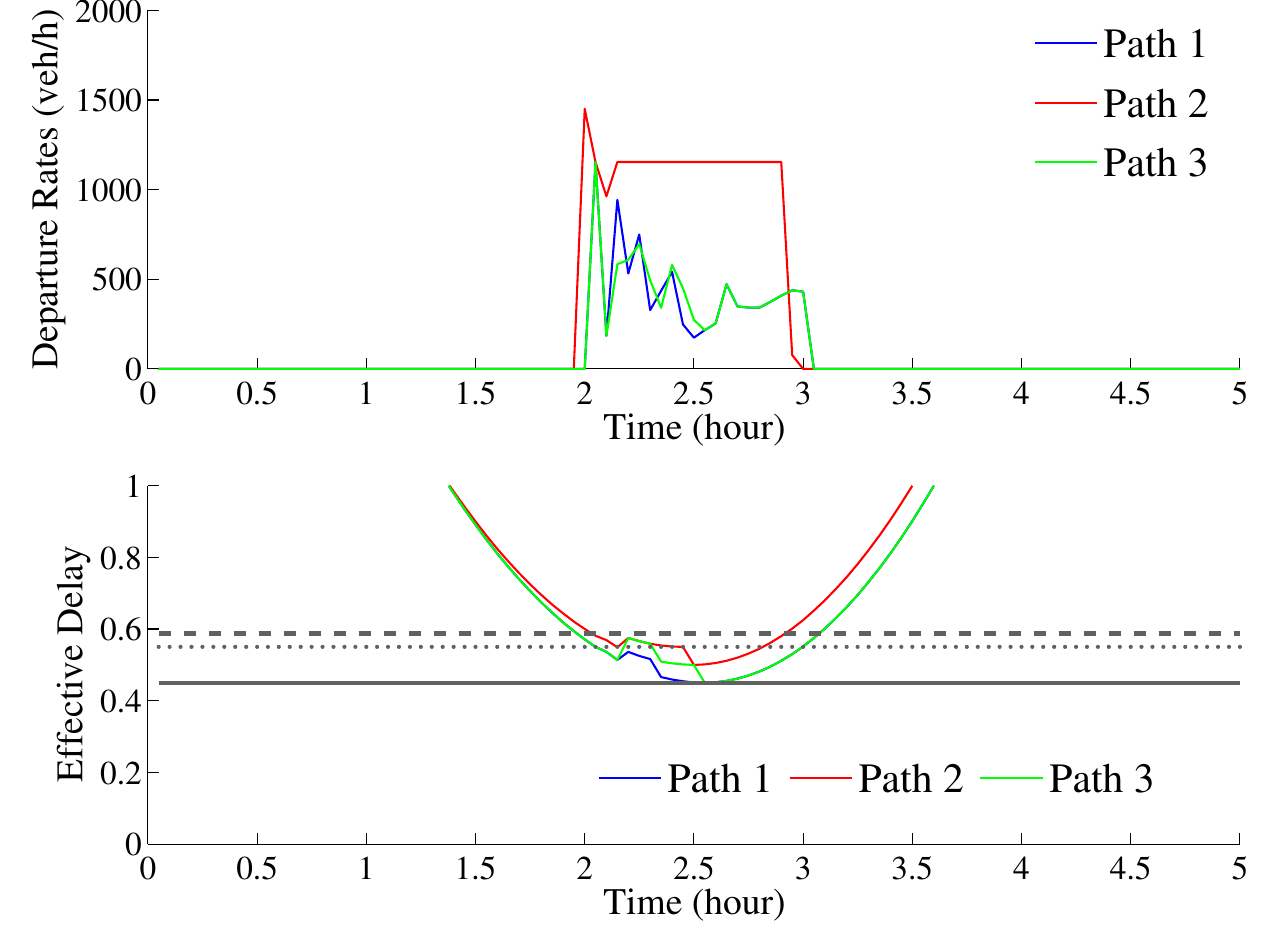}
\end{minipage}
\begin{minipage}[c]{0.49\textwidth}
\centering
\includegraphics[width=\textwidth]{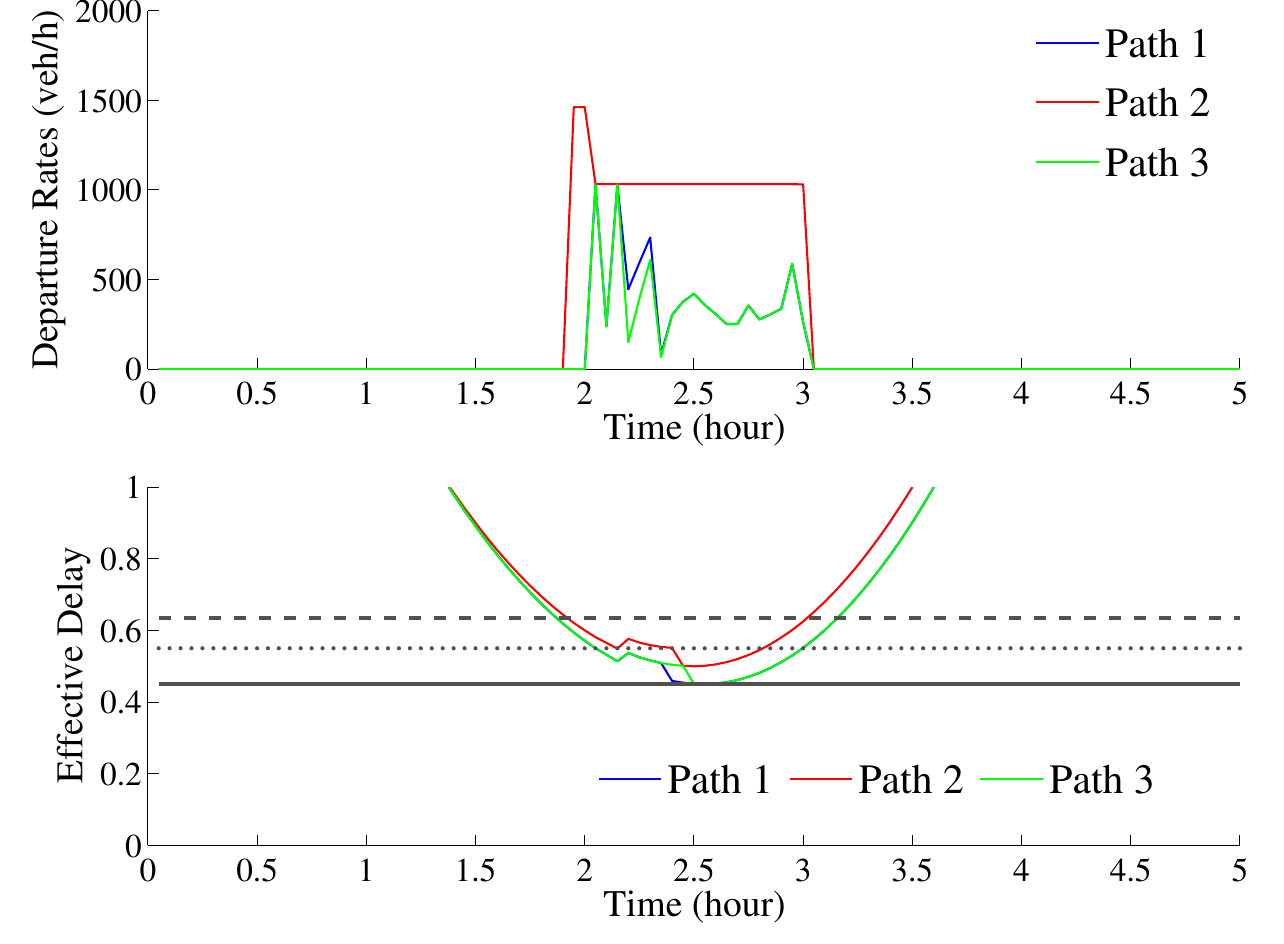}
\end{minipage}
\caption{\small Solutions of the VT-BR-DUE problems in Case I (left) and Case II (right). In the second row, the solid horizontal line represents the minimum effective delay $v_{16}$, the dashed lines represents $v_{16}+\varepsilon_2$, and the dotted lines represent $v_{16}+\varepsilon_1$ (or $v_{16}+\varepsilon_3$).}
\label{figsmallsoln}
\end{figure}

\begin{figure}[h!]
\centering
\includegraphics[width=.8\textwidth]{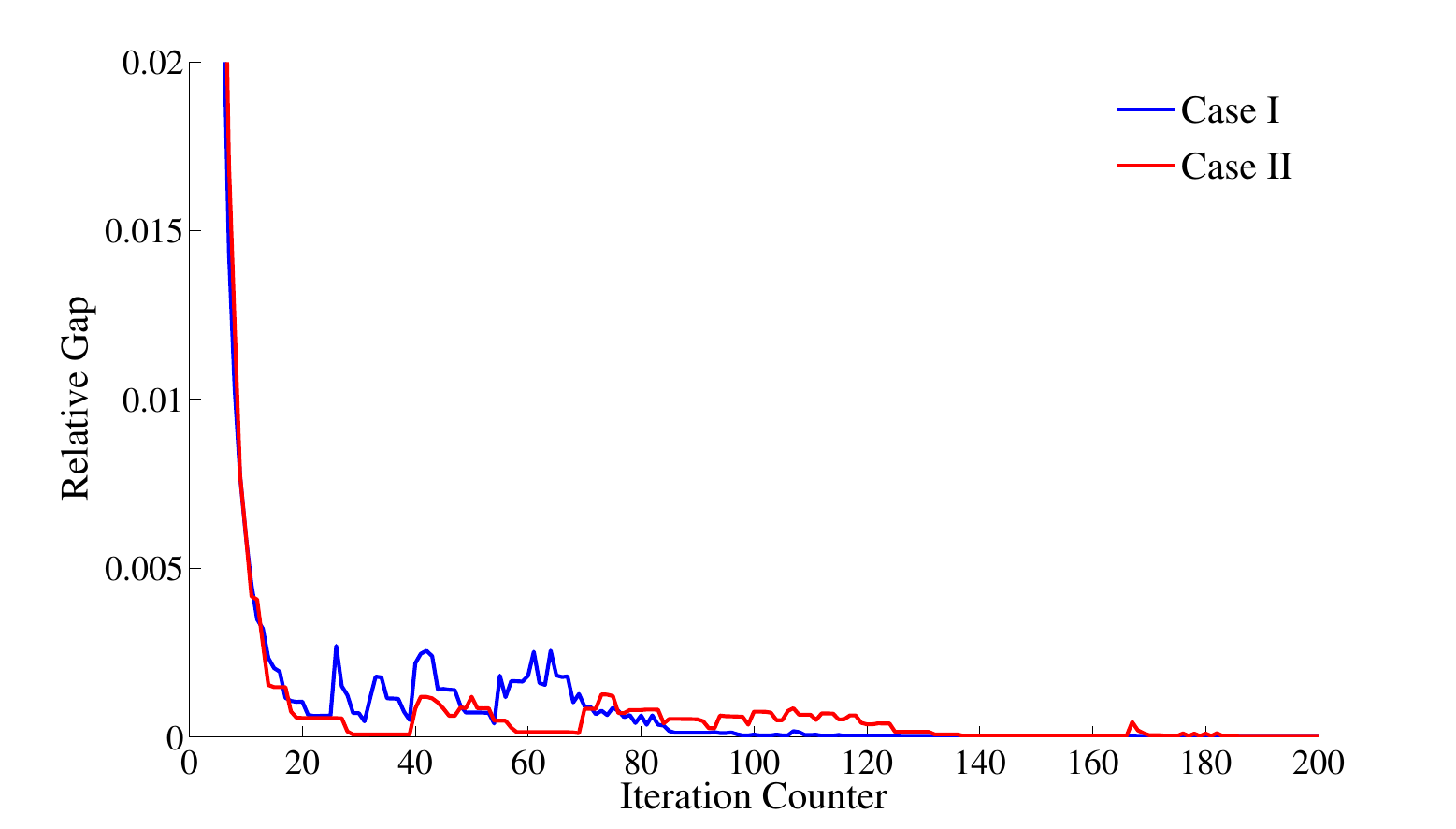}
\caption{\small Convergence of the fixed-point algorithm for the seven-arc, six-node network.}
\label{figsmallconv}
\end{figure}

\subsection{The 19-arc, 13-node network}

The second example evaluates the impact of the {\it fixed} cost tolerance $\varepsilon$ on the final solution of the BR-DUE problem. The numerical results herein demonstrates the sensitivity of the BR-DUE solution to the choices of $\varepsilon$; they also highlight the importance of a well-calibrated indifference band when BR-DUE models are applied.

The second test network is shown in Figure \ref{figmedium}. Four origin-destination pairs, $(1,\,2)$, $(1,\,3)$, $(4,\,2)$, and $(4,\,3)$ are considered, each with a fixed demand of 2000 (veh). There are totally 24 paths among these four O-D pairs. We apply a fixed tolerance $\varepsilon$ for all the O-D pairs and all the paths, so that the problem becomes a BR-DUE problem. Two values of $\varepsilon$ are considered: (1) $\varepsilon=0.4$, and (2) $\varepsilon=0.2$. The corresponding BR-DUE solutions are partially illustrated in Figure \ref{figmediumsoln1} for $\varepsilon=0.4$ and in Figure \ref{figmediumsoln2} for $\varepsilon=0.2$. Notice that the {\it revised} effective delays in all these figures refer to the function $\phi_p^{\varepsilon}(\cdot,\,h^*)$ defined in \eqref{smallphidef}. Recall that a BR-DUE solution $h^*$ must solve the following variational inequality
$$
\big<\phi^{\varepsilon}(\cdot,\,h^*),\,h(\cdot)-h^*(\cdot)\big>~\geq~0\qquad \forall h\in\Lambda
$$
\noindent According to our earlier discussion following Corollary \ref{brduevilemma}, we are assured that the solutions presented in Figure \ref{figmediumsoln1} and Figure \ref{figmediumsoln2} are indeed BR-DUE solutions, since $h_p^*(t)>0$ implies that $\phi_p^{\varepsilon}(t,\,h^*)$ is equal and minimal.

\begin{figure}[h!]
\centering
\includegraphics[width=.5\textwidth]{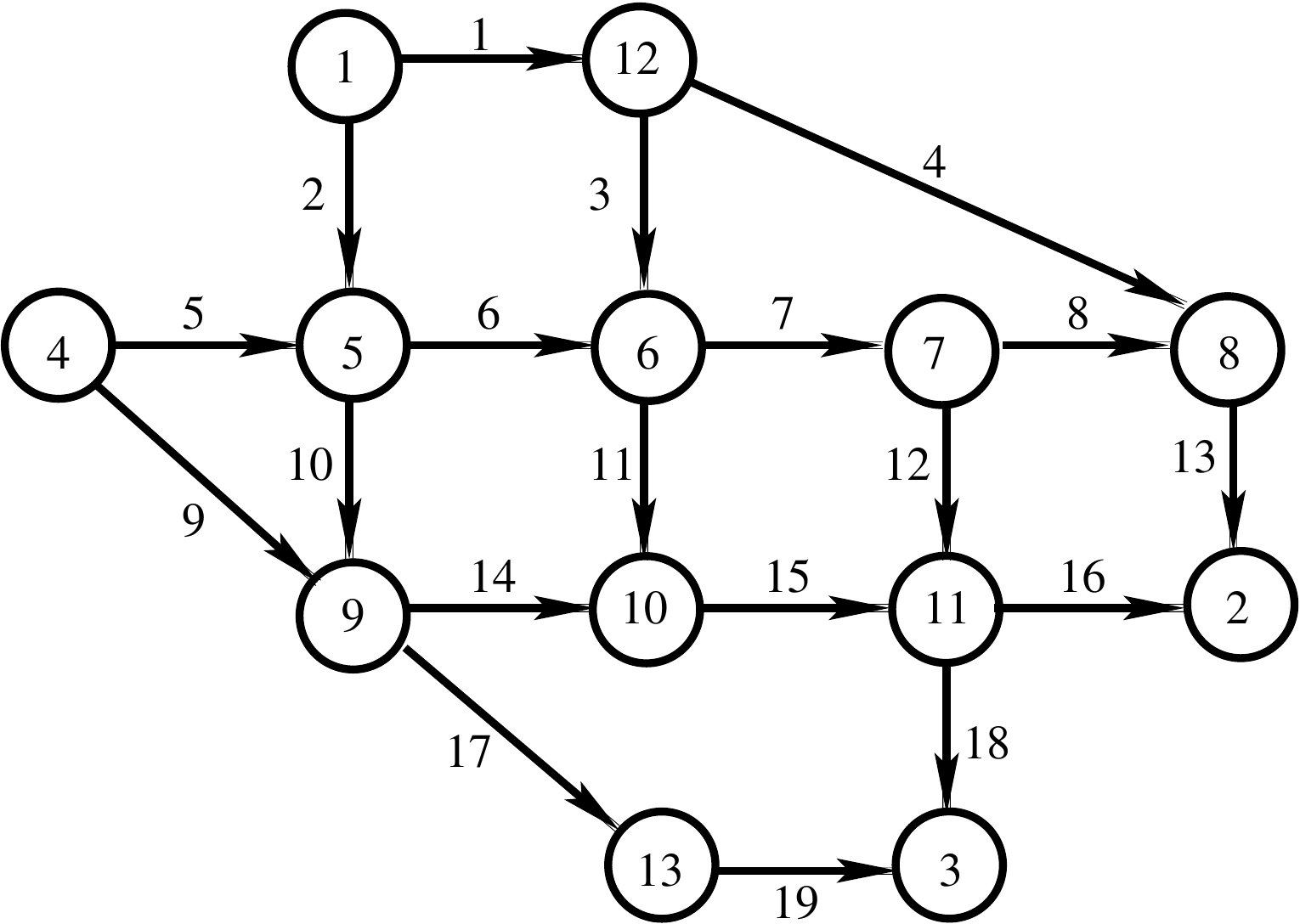}
\caption{\small The 19-arc, 13-node test network}
\label{figmedium}
\end{figure}

\begin{figure}[h!]
\centering
\includegraphics[width=.95\textwidth]{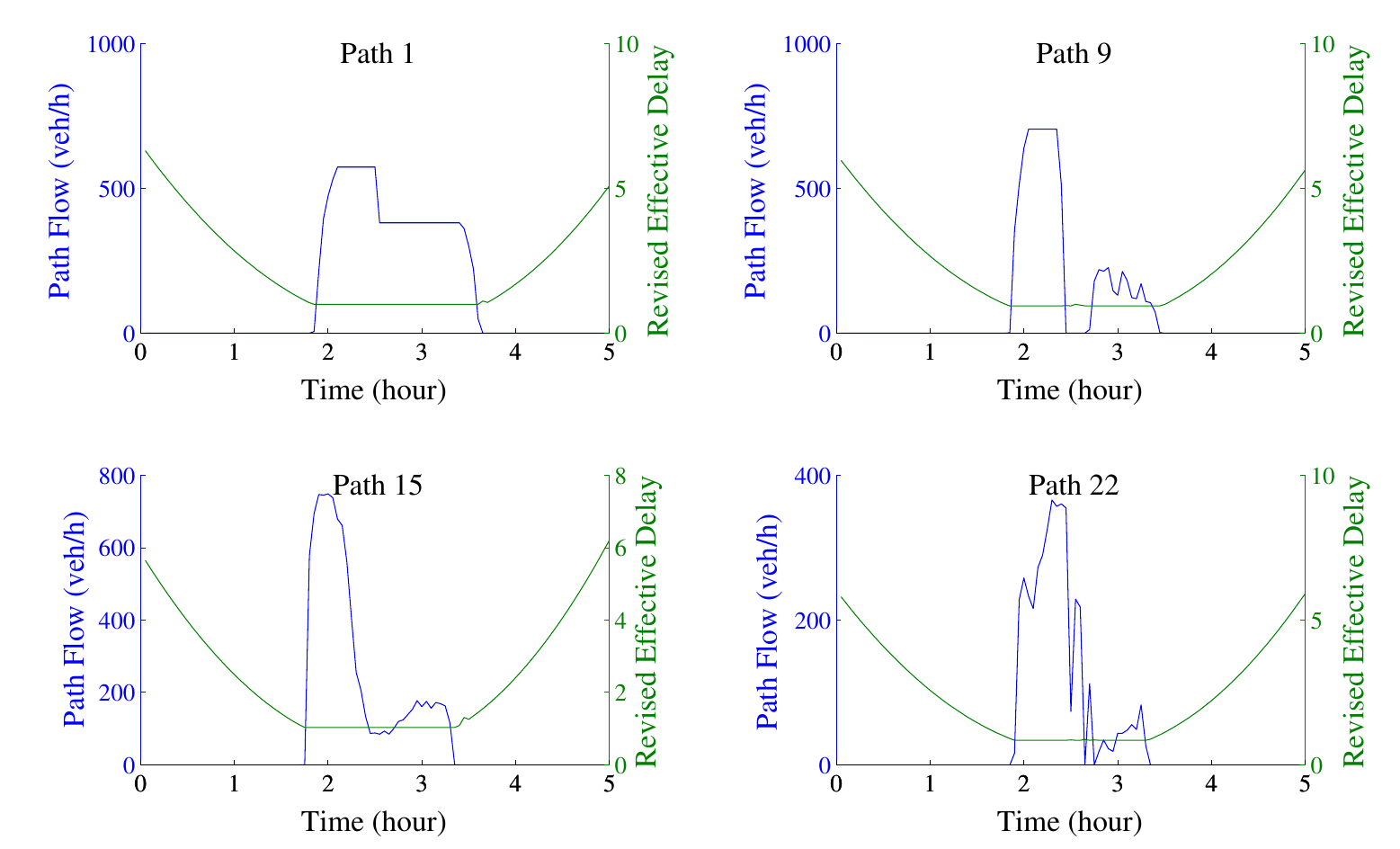}
\caption{\small The 19-arc network with $\varepsilon=0.4$: BR-DUE path departure rates and the corresponding revised effective path delays $\phi^{\varepsilon}$.}
\label{figmediumsoln1}
\end{figure}

\begin{figure}[h!]
\centering
\includegraphics[width=.95\textwidth]{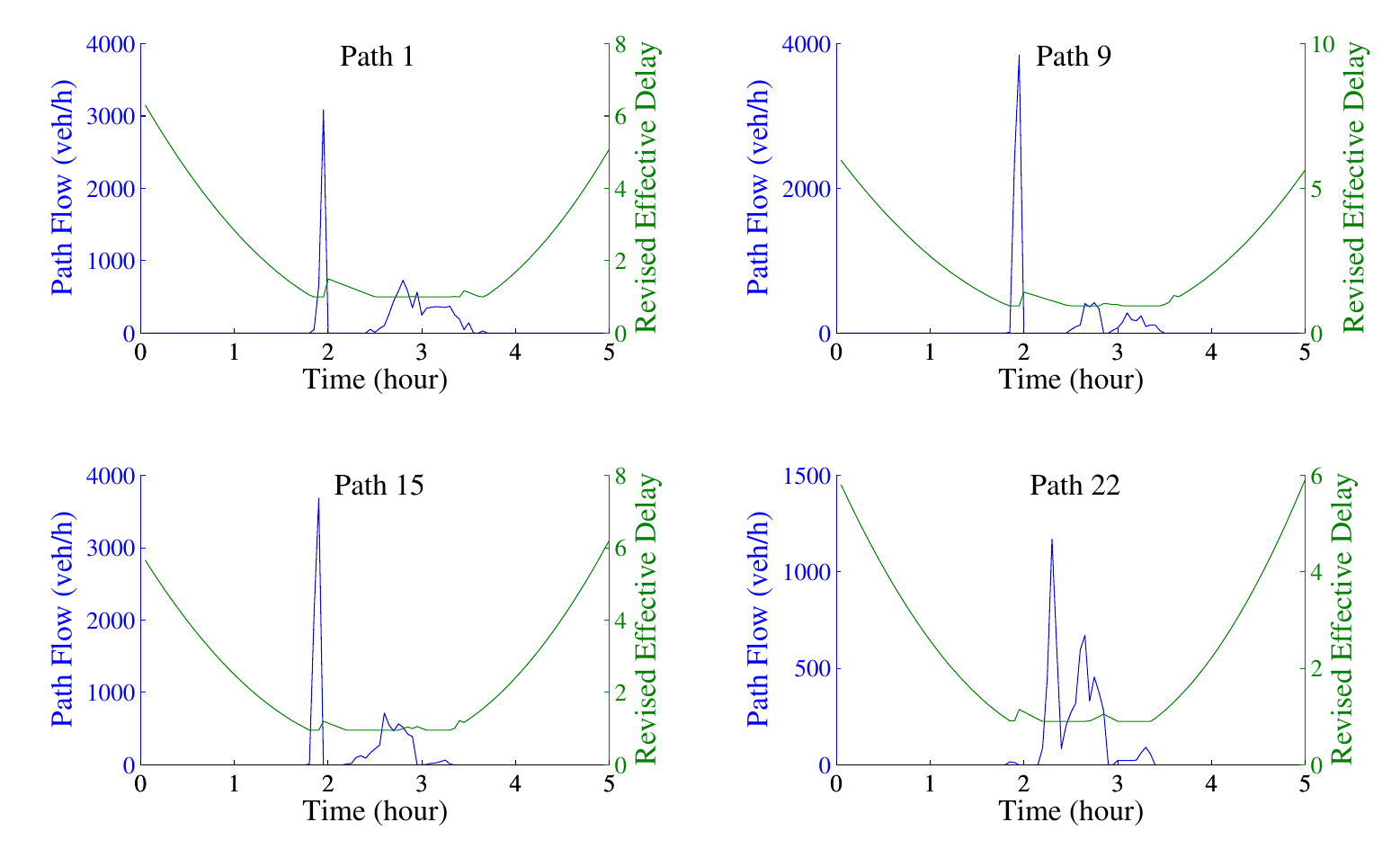}
\caption{\small The 19-arc network with $\varepsilon=0.2$: BR-DUE path departure rates and the corresponding revised effective path delays $\phi^{\varepsilon}$.}
\label{figmediumsoln2}
\end{figure}

We also see from a comparison between Figure \ref{figmediumsoln1} and Figure \ref{figmediumsoln2} that the BR-DUE solutions are very different as a result of different values of $\varepsilon$, indicating a high sensitivity of the solution to the tolerances. Therefore, it is crucial to identify appropriate values of the cost tolerance in order to accurately predict and describe the path departure rates. In addition, as we shall demonstrate later in Section \ref{secsioux}, the larger the value of $\varepsilon$, the fewer iterations the fixed-point algorithm usually takes to converge.

\subsection{The Sioux Falls network}\label{secsioux}

Our third test network is the 76-arc, 24-node Sioux Falls network illustrated in Figure \ref{figsioux2}. In this network we consider six origin-destination pairs: $(1,\,20)$, $(2,\,20)$, $(3,\,20)$, $(4,\,20)$, $(5,\,20)$ and $(6,\,20)$, among which 119 paths are selected.

\begin{figure}[h!]
   \centering
 \includegraphics[width=.4\textwidth]{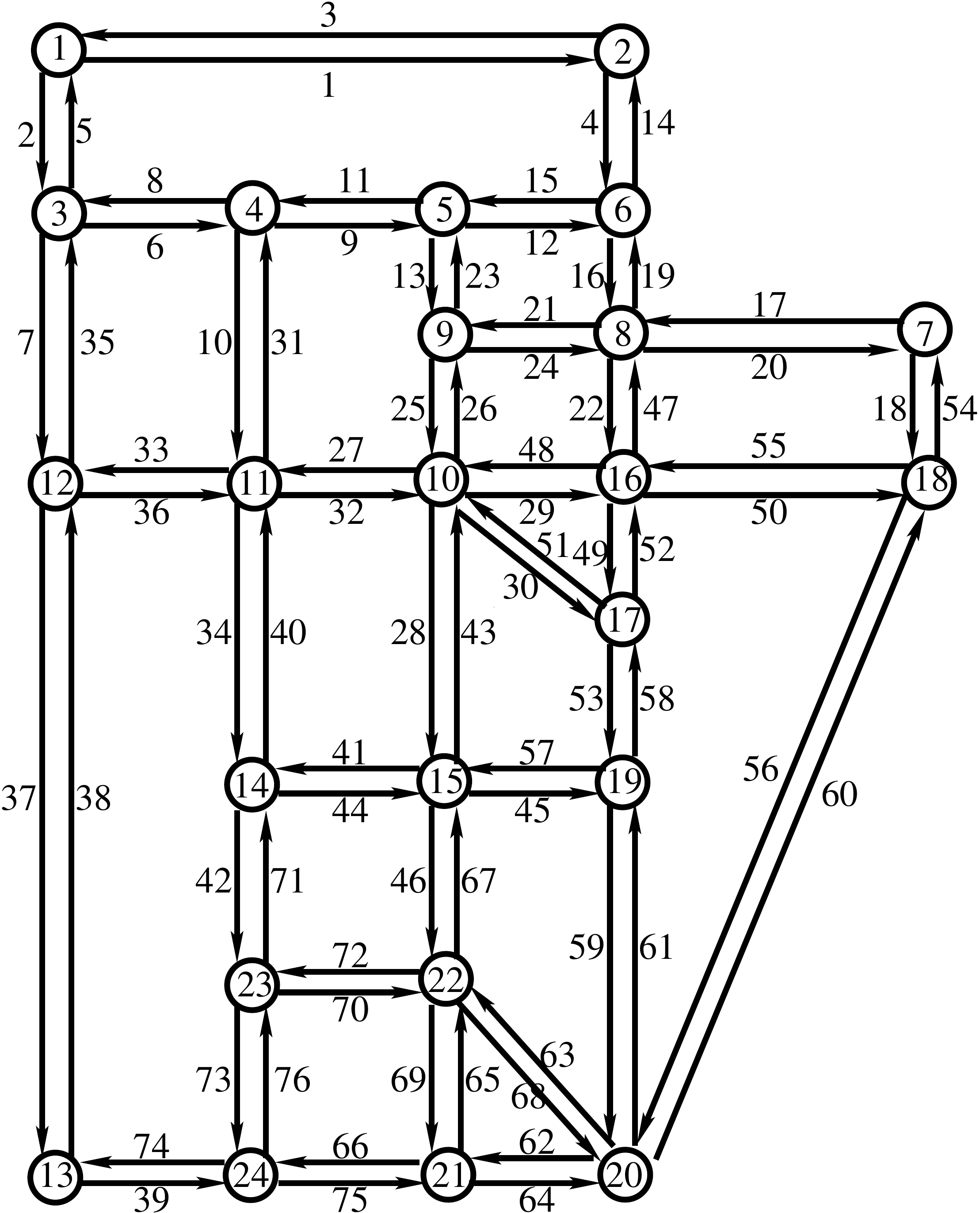}
   \caption{\small The Sioux Falls network}
   \label{figsioux2}
\end{figure}

The main purpose of this numerical study is to evaluate and compare the performances of the three algorithms proposed in this paper, in terms of their convergence and computational efficiency. To make our numerical study more comprehensive, we also consider the seven-arc network and 19-arc network as our test networks. Moreover, we will compare the convergence results with varying values of $\varepsilon$. To simplify the problem, we assume that the cost tolerances are fixed  and equal for all O-D pairs and paths.  The following termination criteria are employed for the three methods:
\begin{align}
&{\|h^{k+1}-h^k\|_{L^2}\over \|h^k\|_{L^2}}\leq10^{-4} ~~\text{(fixed-point method; proximal point method)} \label{relgapdef1}
\\
& {\left\|r(h^k; \beta_k)\right\|_{L^2}\over \|h^k\|_{L^2}}\leq 10^{-4}~~\text{(self-adaptive projection method)}  \label{relgapdef2}
\end{align}

\begin{figure}[h!]
   \centering
 \includegraphics[width=\textwidth]{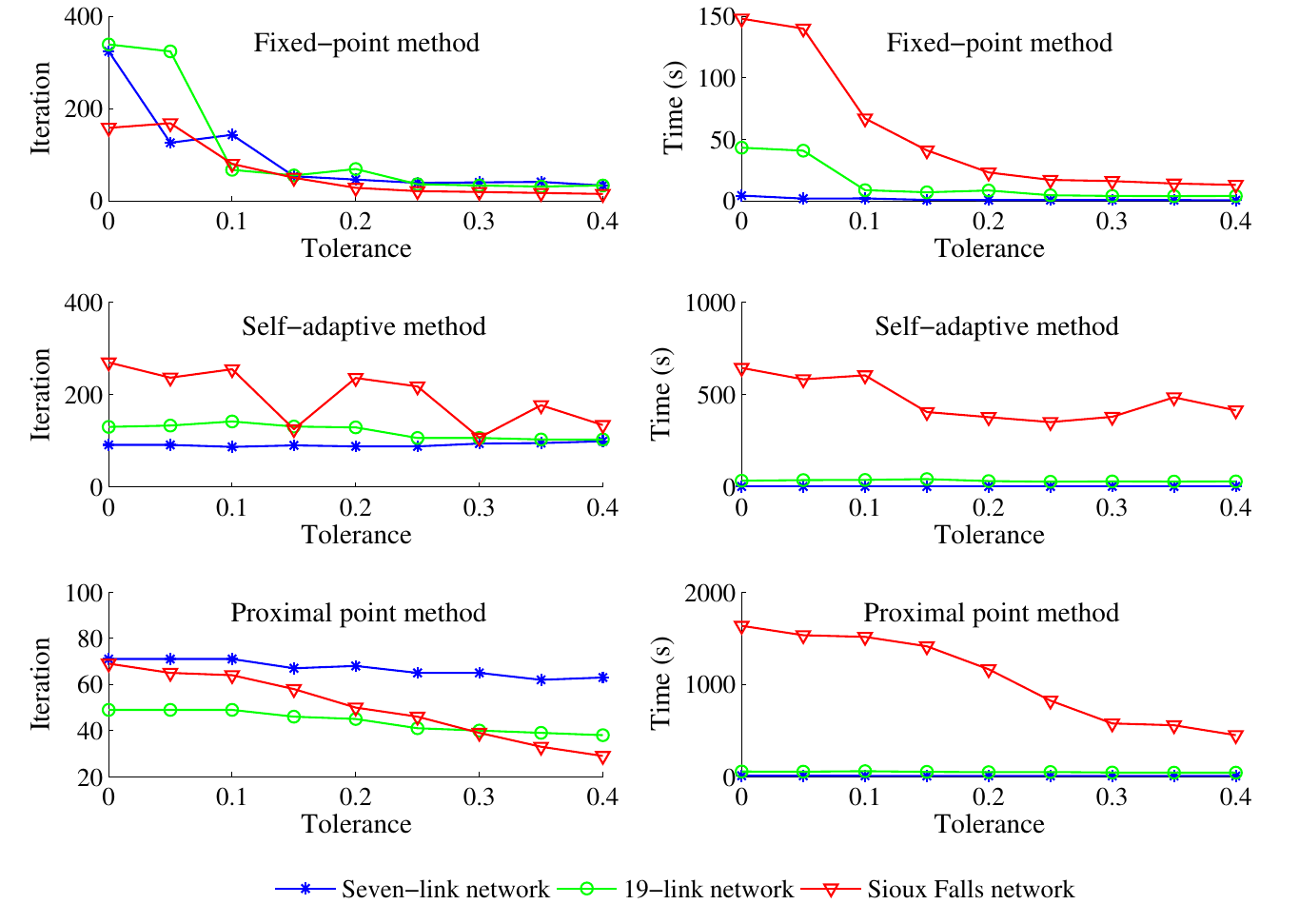}
   \caption{\small Comparison of the three algorithms on the three test networks with a range of $\varepsilon$.}
   \label{figY}
\end{figure}

The convergence results and the computational times of the three algorithms implemented on different networks are shown in Figure \ref{figY}, where we consider a range of values for the tolerance $\varepsilon$. The three algorithms show qualitatively different convergence trends.  For the fixed-point algorithm and the proximal point method, as the tolerance $\varepsilon$ becomes larger, both algorithms take fewer iterations to converge, although the iterations needed by the proximal point method is less sensitive to the tolerance than the fixed-point method. This is understood given that the larger the tolerance, the more likely the traffic is equilibrated. For the self-adaptive projection method, however, there is no discernible effect of $\varepsilon$ on the convergence of the algorithm, especially for the smaller networks.  We also see, from the first and the third rows of Figure \ref{figY}, that as the network size increases, despite the increase in the dimension of the problem (i.e., the number of paths), the number of iterations required by both algorithms remain roughly the same. This implies good scalability and dimension-free nature of the fixed-point and proximal point methods.

With regard to the computational time, the three algorithms differ significantly. In particular, for the same network, the computational time of the fixed-point method is proportional to the number of iterations needed. This is obviously due to the fact that each fixed-point iteration requires exactly one DNL procedure. When the computational times of the fixed-point algorithm are compared across different networks, the larger the network the more time needed even if the iteration numbers remain more or less the same. This is because the DNL procedure for larger networks consumes more time. For the self-adaptive method and the proximal point method, the computational times are disproportional to the iteration numbers; this is because each iteration in these algorithms may require multiple DNL procedures (see Step 2 of the self-adaptive projection method and Step 1 of the proximal point method). As a result,  their computational times are significantly larger than the fixed-point algorithm, especially for the largest network (Sioux Falls). We may conclude that the self-adaptive method and the proximal point method in general take more time to reach a given level of convergence than the fixed-point algorithm, despite the fact that they enjoy more relaxed convergence conditions than the latter. This highlights the potential trade-off between improved convergence and computational efficiency.

To further analyze the convergence patterns of these three methods, we shown in Figure \ref{figSiouxconv} the relative gaps defined in \eqref{relgapdef1}-\eqref{relgapdef2} at each iteration. This figure is based on the calculation on the Sioux Falls network with tolerance $\varepsilon=0.2$. We see that the self-adaptive projection and the proximal point method have smoother convergence than the fixed-point method. The latter, however, has a faster convergence rate. Overall, the fixed-point method and the proximal point method converge faster than the self-adaptive projection method within \underline{finite iterations}, and they are able to achieve a relative gap of $10^{-6}$ within 300 iterations. However, distinctions need to be made between the asymptotic convergence of an algorithm, and the convergence to a given precision within finite iterations. Most existing convergence results are established in the former sense, while only a few (such as Theorem \ref{ppmthm}) are concerned with convergence to an approximate solution in a computationally practical way.

\begin{figure}[h!]
   \centering
 \includegraphics[width=.9\textwidth]{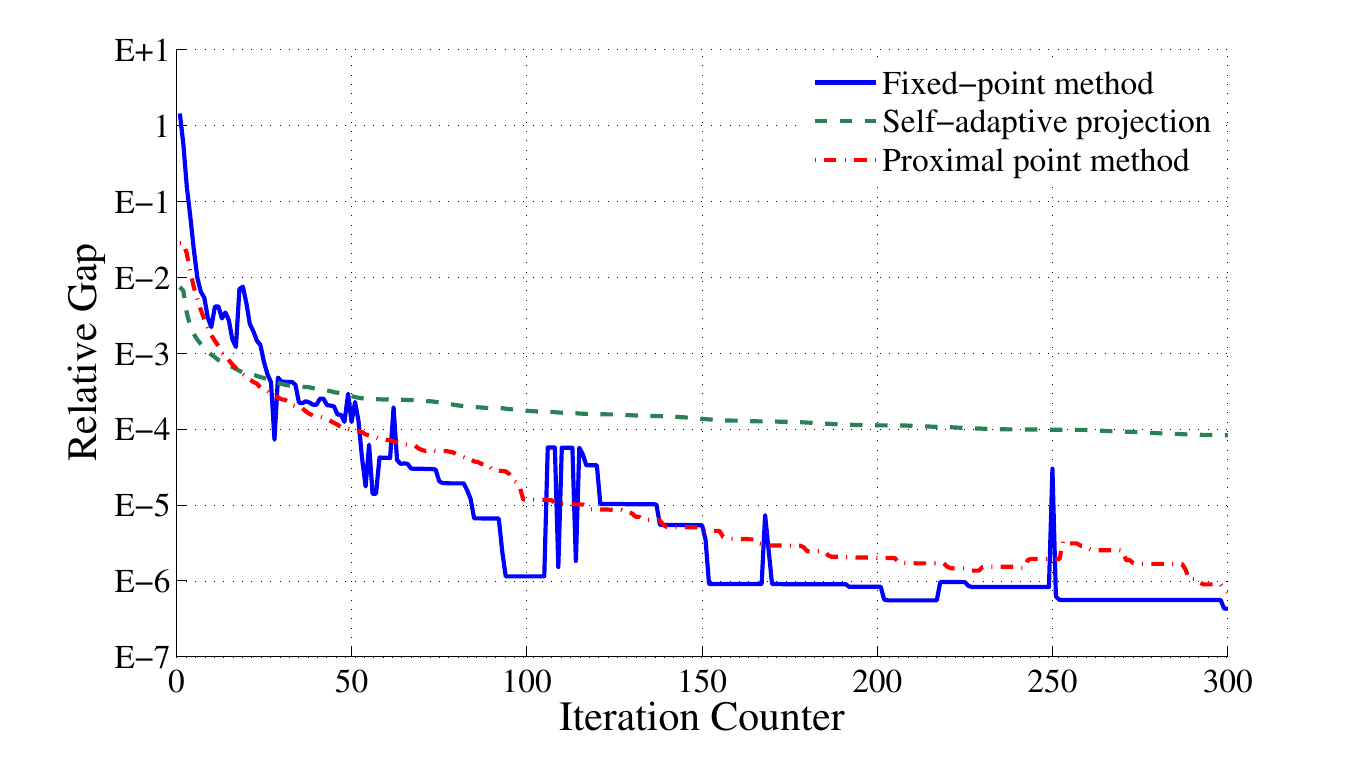}
   \caption{\small Convergence of the three algorithms on the Sioux Falls network with tolerance $\varepsilon=0.2$. On the y-axis E-i means $10^{-i}$.}
   \label{figSiouxconv}
\end{figure}

\section{Concluding remarks}\label{secconclusion}

This paper analyzes the simultaneous route-and-departure-time (SRDT) dynamic user equilibrium (DUE) with bounded rationality (BR). Specifically, we consider DUEs with fixed tolerances (BR-DUE) or endogenously determined tolerances (VT-BR-DUE), where the former is a special case of the latter. This paper makes significant contributions in four areas: problem formulation, existence of solutions, characterization of solution set, and computation.

We first show that the VT-BR-DUE problem is equivalent to a variational inequality (VI) problem using measure-theoretic argument. A key ingredient of the VI formulation is the newly introduced principal operator, which simultaneously encapsulates the dynamic network loading sub-model and the heterogenous and possibly endogenous user tolerances. This operator is shown to be well-defined and continuous, which is crucial for the existence and computation of VT-BR-DUE. We then provide an existence result for VT-BR-DUE based on the VI formulation. It is notable that this existence result relies on conditions that are weaker than those ensuring the existence of normal DUEs. Furthermore, this paper is the first to provide a characterization of the solution set of VT-BR-DUEs. In particular, under a discrete-time setting, we show compactness of the solution set, analyze its interior points, and provide an analytical procedure to find infinitely many solutions and construct connected components of the solution set without resort to numerical computations. These results constitute the first characterization of the solution set for DUE problems that incorporate bounded rationality. Finally, the computability of the VT-BR-DUE models is demonstrated with three new algorithms: a fixed-point algorithm, a self-adaptive projection algorithm, and a proximal point algorithm, all of which enjoy rigorous convergence results based on generalized monotonicity, and are assessed in terms of solution quality, convergence, and computational efficiency.

Although this paper is mainly concerned with dynamic modeling, the techniques used to convert a BR-DUE problem into a generic VI form is applicable to static problems, i.e., boundedly rational user equilibrium (BR-UE). In addition, the route-choice dynamic user equilibrium with bounded rationality can be treated in a very similar way; and it is expected that the VI formulation, existence results, and computational methods will become available for this type of problems as well. Due to space limitation, those results are not elaborated here but will be mentioned in future research.

The variable tolerance BR-DUE model considered by this paper needs to be further validated and calibrated using empirical studies, which should identify and distinguish a range of factors that affect users' perception of travel cost and hence describe the indifference band on a more refined and comprehensive level.  Although this paper does not provide direct empirical evidence of the variable tolerances, which is too much for the current paper, it provides a general modeling platform with significant theoretical and computational results to facilitate future studies on bounded rationality.

\section{Acknowledgement} 
The work described in this paper was partially supported by a grant from the Research Grants Council of the Hong Kong Special Administrative Region, China (HKU 716312E), a grant (201311159123) from the University Research Committee from the University of Hong Kong, and a grant from the National Natural Science Foundation of China (71271183). The authors are grateful to the three reviewers for their constructive comments


 \appendix

 \section{Proof of  Theorem \ref{vtbrduevithm}}\label{secapp1}

\begin{proof}
(i) [Necessity] Let $h^*\in\Lambda$ be a VT-BR-DUE solution, then for any  $h\in\Lambda$ and any $(i,\,j)\in\mathcal{W}$,
\begin{align}
&\sum_{(i,\,j)\in\mathcal{W}}\sum_{p\in\mathcal{P}_{ij}}\int_{t_0}^{t_f}\Phi^{\varepsilon}_p(t,\,h^*)\,h_p(t)\,dt~\geq~\sum_{(i,\,j)\in\mathcal{W}}\sum_{p\in\mathcal{P}_{ij}}\int_{t_0}^{t_f} \mu^{\varepsilon}_{ij}(h^*)\,h_p(t)\,dt \nonumber
\\
\label{eqn1}
~=~&\sum_{(i,\,j)\in\mathcal{W}}\mu^{\varepsilon}_{ij}(h^*)\sum_{p\in\mathcal{P}_{ij}}\int_{t_0}^{t_f}h_p(t)\,dt~=~\sum_{(i,\,j)\in\mathcal{W}}\mu^{\varepsilon}_{ij}(h^*)Q_{ij}
\end{align}
where 
\begin{equation}\label{mudef1}
\mu^{\varepsilon}_{ij}(h^*)~\doteq~\min_{p\in\mathcal{P}_{ij}}\left\{\mu^{\varepsilon}_p(h^*)\right\},\qquad \mu^{\varepsilon}_p(h^*)~\doteq~\underset{t\in[t_0,\,t_f]}{\hbox{essinf}}\left\{\Phi^{\varepsilon}_p(t,\,h^*)\right\}
\end{equation}

We claim that $\mu^{\varepsilon}_{ij}(h^*)=v_{ij}(h^*)+\min\limits_{q\in\mathcal{P}_{ij}}\left\{\varepsilon^q_{ij}(h^*)\right\}$. Indeed, we first notice from \eqref{Phidef1} that 
\begin{align*}
\mu_p^{\varepsilon}(h^*)&~=~\underset{t\in[t_0,\,t_f]}{\hbox{essinf}}\left\{\max\big\{\Psi_p(t,\,h^*),\,v_{ij}(h^*)+\varepsilon_{ij}^p(h^*)\big\} -\varepsilon_{ij}^{p}(h^*)+\min_{q\in\mathcal{P}_{ij}}\left\{\varepsilon_{ij}^q(h^*)\right\}\right\}
\\
&~=~\underset{t\in[t_0,\,t_f]}{\hbox{essinf}}\Bigg\{\max\big\{\Psi_p(t,\,h^*),\,v_{ij}(h^*)+\varepsilon_{ij}^p(h^*)\big\} \Bigg\}  -\varepsilon_{ij}^{p}(h^*)+\min_{q\in\mathcal{P}_{ij}}\left\{\varepsilon_{ij}^q(h^*)\right\}
\end{align*}
\noindent We then distinguish  two cases. If $p\in\mathcal{P}_{ij}$ is such that $\underset{t\in[t_0,\,t_f]}{\hbox{essinf}}\left\{\Psi_p(t,\,h^*)\right\}\geq v_{ij}(h^*)+\varepsilon_{ij}^p(h^*)$, then 
\begin{equation}\label{rveqn1}
\mu_p^{\varepsilon}(h^*)~=~\underset{t\in[t_0,\,t_f]}{\hbox{essinf}}\left\{\Psi_p(t,\,h^*)\right\}-\varepsilon_{ij}^p(h^*)+\min\limits_{q\in\mathcal{P}_{ij}}\left\{\varepsilon_{ij}^q(h^*)\right\}\geq v_{ij}(h^*)+\min\limits_{q\in\mathcal{P}_{ij}}\left\{\varepsilon_{ij}^q(h^*)\right\}
\end{equation}
\noindent On the other hand, if $p\in\mathcal{P}_{ij}$ is such that $\underset{t\in[t_0,\,t_f]}{\hbox{essinf}}\left\{\Psi_p(t,\,h^*)\right\}< v_{ij}(h^*)+\varepsilon_{ij}^p(h^*)$, then
\begin{equation}\label{rveqn2}
\mu_p^{\varepsilon}(h^*)~=~v_{ij}(h^*)+\varepsilon_{ij}^p(h^*)-\varepsilon_{ij}^{p}(h^*)+\min_{q\in\mathcal{P}_{ij}}\left\{\varepsilon_{ij}^q(h^*)\right\}~=~v_{ij}(h^*)+\min_{q\in\mathcal{P}_{ij}}\left\{\varepsilon_{ij}^q(h^*)\right\}
\end{equation}
\noindent \eqref{rveqn1} and \eqref{rveqn2} combined show that $\mu_{ij}^{\varepsilon}=\min\limits_{p\in\mathcal{P}_{ij}}\left\{\mu^{\varepsilon}_p(h^*)\right\}\geq v_{ij}(h^*)+\min\limits_{q\in\mathcal{P}_{ij}}\left\{\varepsilon_{ij}^q(h^*)\right\}$. Finally, notice that there exists at least one path $p$ such that \eqref{rveqn2} is true (e.g., the path at which the essential infimum $v_{ij}(h^*)$ is attained). We thus conclude that $\mu_{ij}^{\varepsilon}=v_{ij}(h^*)+\min\limits_{q\in\mathcal{P}_{ij}}\left\{\varepsilon_{ij}^q(h^*)\right\}$. Our claim is substantiated.

In view of \eqref{vtbrduedef1} and \eqref{Phidef1}, we perform the following deduction for every $p\in\mathcal{P}_{ij}$:
\begin{align*}
h_p^*(t)~>~0&~\Longrightarrow~\Psi_p(t,\,h^*)~\leq~v_{ij}(h^*)+\varepsilon_{ij}^p(h^*)
\\
&~\Longrightarrow~\Phi_p^{\varepsilon}(t,\,h^*)~=~v_{ij}(h^*)+\varepsilon_{ij}^p(h)-\left(\varepsilon_{ij}^p(h^*)-\min_{q\in\mathcal{P}_{ij}}\left\{\varepsilon_{ij}^q(h^*)\right\}\right)~=~\mu_{ij}^{\varepsilon}(h^*)
\end{align*}
Therefore, according to the non-negativity of $h$, we have
\begin{equation}\label{eqn2}
\sum_{(i,\,j)\in\mathcal{W}}\sum_{p\in\mathcal{P}_{ij}}\int_{t_0}^{t_f} \Phi^{\varepsilon}_p(t,\,h^*)\,h_p^*(t)\,dt~=~\sum_{(i,\,j)\in\mathcal{W}}\sum_{p\in\mathcal{P}_{ij}}\int_{t_0}^{t_f}\mu_{ij}^{\varepsilon}(h^*)\,h_p^*(t)\,dt   
~=~\sum_{(i,\,j)\in\mathcal{W}}\mu^{\varepsilon}_{ij}(h^*)Q_{ij}
\end{equation}
In view of \eqref{eqn1} and \eqref{eqn2}, we have 
$$
\sum_{(i,\,j)\in\mathcal{W}}\sum_{p\in\mathcal{P}_{ij}}\int_{t_0}^{t_f}\Phi_p^{\varepsilon}(t,\,h^*)\,h_p^*(t)\,dt~\leq~\sum_{(i,\,j)\in\mathcal{W}}\sum_{p\in\mathcal{P}_{ij}}\int_{t_0}^{t_f}\Phi_p^{\varepsilon}(t,\,h^*)\,h_p(t)\,dt\qquad \forall h\in \Lambda
$$
\noindent which is recognized as the variational inequality \eqref{newvtbrduevi}. \\

\noindent (ii) [Sufficiency]  Let $h^*\in\Lambda$ be a solution of the variational inequality. Clearly, $\Phi_p^{\varepsilon}(\cdot,\,h)$ is measurable and positive for any $p\in\mathcal{P}$ and any $h\in\Lambda$. We invoke the same proof of Theorem 2 from \cite{Friesz1993} to show that $h^*$ satisfies 
\begin{equation}\label{eqn3}
h_p^*(t)~>~0,~ p\in\mathcal{P}_{ij}~\Longrightarrow~ \Phi_p^{\varepsilon}(t,\,h^*)~=~\mu^{\varepsilon}_{ij}(h^*) \qquad \hbox{for almost every}~t\in[t_0,\,t_f],\quad \forall (i,\,j)\in\mathcal{W}
\end{equation}
where $\mu_{ij}^{\varepsilon}(h^*)$ is given by \eqref{mudef1} and is equal to $v_{ij}(h^*)+\min_{q\in\mathcal{P}_{ij}}\{\varepsilon_{ij}^q(h^*)\}$. We readily deduce that
\begin{align*}
&h_p^*(t)~>~0,~p\in\mathcal{P}_{ij}~\Longrightarrow~\Phi_p^{\varepsilon}(t,\,h^*)~=~v_{ij}(h^*)+\min_{p\in\mathcal{P}_{ij}}\{\varepsilon_{ij}^p(h^*)\}
\\
~\Longrightarrow~&\max\left\{\Psi_p(t,\,h^*),\,v_{ij}(h^*)+\varepsilon_{ij}^p(h^*)\right\}-\varepsilon_{ij}^p(h^*)+\min_{q\in\mathcal{P}_{ij}}\{\varepsilon_{ij}^q(h^*)\}~=~v_{ij}(h^*)+\min_{q\in\mathcal{P}_{ij}}\{\varepsilon_{ij}^q(h^*)\}
\\
~\Longrightarrow~&\max\left\{\Psi_p(t,\,h^*),\,v_{ij}(h^*)+\varepsilon_{ij}^p(h^*)\right\}~=~v_{ij}(h^*)+\varepsilon_{ij}^p(h^*)
\\
~\Longrightarrow~&v_{ij}(h^*)~\leq~\Psi_p(t,\,h^*)~\leq~v_{ij}(h^*)+\varepsilon_{ij}^p(h^*)  
\end{align*}
\noindent for almost every $t\in[t_0,\,t_f]$, $\forall p\in\mathcal{P}_{ij}$, $\forall (i,\,j)\in\mathcal{W}$. Therefore, $h^*$ solves the VT-BR-DUE problem. \end{proof}

\section{Proofs of Proposition \ref{propwelldefine} and Theorem \ref{propcontphi}.}\label{secappcontproof}

\subsection{Proof of Proposition \ref{propwelldefine}}\label{subsecappthmwelldefined}
\begin{proof}
Given any $h\in\Lambda$, by virtue of the effective path delay operator there exists a unique vector-valued function $\big(\Psi_p(\cdot,\,h),\,p\in\mathcal{P}\big)$ of $t$, where $\Psi_p(t,\,h)$ is given by \eqref{cost}. Moreover, such a vector-valued function belongs to the set $\big(L_+^2[t_0,\,t_f]\big)^{|\mathcal{P}|}$. Then, according to \eqref{Phidef1} there exists a unique vector-valued function $\big(\Phi^{\varepsilon}_p(\cdot,\,h),\,p\in\mathcal{P}\big)$ of $t$. Moreover, due to \eqref{epsilonunib} there exists an upper bound $M<+\infty$ for all the functionals $\varepsilon_{ij}^p(\cdot),\,p\in\mathcal{P}_{ij},\,(i,\,j)\in\mathcal{W}$. Then we have, for every $p\in\mathcal{P}$, that 
\begin{align*}
\int_{t_0}^{t_f}\left[\Phi_{p}^{\varepsilon}(t,\,h)\right]^2\,dt&~\leq~\int_{t_0}^{t_f}\left[\Psi_p(t,\,h) + M\right]^2\,dt
\\
&~=~\int_{t_0}^{t_f}\left[\Psi_p(t,\,h)\right]^2\,dt +2M\int_{t_0}^{t_f}\Psi_p(t,\,h)\,dt +(t_f-t_0)M^2~<~+\infty
\end{align*}
Here we have used the fact that a square-integrable function on a compact set is also integrable. Therefore, $\Phi^{\varepsilon}(h)\in\big(L_+^2[t_0,\,t_f]\big)^{|\mathcal{P}|}$. We conclude that the operator $\Phi^{\varepsilon}$ exists and is well defined. 
\end{proof}

\subsection{Proof of Theorem \ref{propcontphi}}\label{subsecappcontproof}

Before we begin the continuity proof, we make a very simple yet crucial observation regarding the  effective path delay  $\Psi_p(\cdot,\,h)$,  viewed as a function of departure time $t$. 
\begin{lemma}
Assuming that {\bf A0} holds. Then under the {\it first-in-first-out} (FIFO) rule, there must hold that
\begin{equation}\label{PsionesideL}
\Psi_p(t_2,\,h)-\Psi_p(t_1,\,h)~\geq~-L_{ij}(t_2-t_1)\qquad \forall p\in\mathcal{P}_{ij},\,\forall h\in\Lambda
\end{equation}
for any $t_0\leq t_1\leq t_2\leq t_f$. Here, $L_{ij}$ is the Lipschitz constant associated with the function $\phi_{ij}$ as articulated in assumption {\bf A0}.
\end{lemma}
\begin{proof}
According to FIFO, a later departure time implies a later arrival time along the same path; we have, for any $t_1\leq t_2$, that
$$
t_1+D_p(t_1,\,h)~\leq~t_2+D_p(t_2,\,h)\qquad \forall p\in\mathcal{P}_{ij},\,\forall h\in\Lambda
$$
According to the alternative representation of the effective path delay and assumption {\bf A0}, we deduce that
\begin{align*}
\Psi_p(t_2,\,h)-\Psi_p(t_1,\,h)=& \phi_{ij}(t_2)+\psi_{ij}\big(t_2+D_p(t_2,\,h)\big) - \phi_{ij}(t_1)-\psi_{ij}\big(t_1+D_p(t_1,\,h)\big)
\\
\geq&-L_{ij}(t_2-t_1) +\psi_{ij}\big(t_2+D_p(t_2,\,h)\big) - \psi_{ij}\big(t_1+D_p(t_1,\,h)\big)~\geq~-L_{ij}(t_2-t_1) 
\end{align*}
\end{proof}

We now begin the proof of Theorem \ref{propcontphi}.

\begin{proof} 
The proof is divided into several parts.

\noindent {\bf Part 1.} Consider an arbitrary O-D pair $(i,\,j)\in\mathcal{W}$.  We show in this part that if $h^1,\,h^2\in\Lambda$ and $\|h^1-h^2\|_{L^2}\to 0$, then the essential infima satisfy $|v_{ij}(h^1)-v_{ij}(h^2)|\to 0$, where these essential infima are defined in \eqref{essinfdef1}-\eqref{essinfdef2}.

By continuity of the effective delay operator, given any $\epsilon>0$, there exists $\delta>0$ such that whenever $h^1, h^2\in\Lambda$ and $\|h^1-h^2\|_{L^2}<\delta$, there holds $\|\Psi(h^1)-\Psi(h^2)\|_{L^2}<{\epsilon^{3/2}\over \sqrt{8L_{ij}}}$, where $L_{ij}$ is the Lipschitz constant from {\bf A0}. Without loss of generality we let $v_{ij}(h^1)\leq v_{ij}(h^2)$. Then we claim that 
\begin{equation}\label{contproofcontr}
v_{ij}(h^1)\geq v_{ij}(h^2)-\epsilon
\end{equation}
We proceed by contradiction. Assume that $v_{ij}(h^1)<v_{ij}(h^2)-\epsilon$, 
and that $v_{ij}(h^1)$ is attained at some time $\hat t\in[t_0,\,t_f]$ for some path $\hat p\in\mathcal{P}_{ij}$. Then, for every $t\in [\hat t-{\epsilon\over 2L_{ij}},\, \hat t]$, according to \eqref{PsionesideL}, 
\begin{align*}
&\Psi_{\hat p}(\hat t,\,h^1)-\Psi_{\hat p}(t,\,h^1)~\geq~-L_{ij}(\hat t-t)~\Longrightarrow~\Psi_{\hat p}(t,\,h^1)~\leq~\Psi_{\hat p}(\hat t,\,h^1)+L_{ij}(\hat t-t)
\\
~=~&v_{ij}(h^1)+L_{ij}(\hat t-t)~<~v_{ij}(h^2)-\epsilon+L_{ij}(\hat t-t)~\leq~v_{ij}(h^2)-\epsilon+L_{ij}\cdot {\epsilon\over 2L_{ij}}
\\
~=~& v_{ij}(h^2)-{\epsilon\over 2}~\leq~ \Psi_{\hat p}(t,\,h^2)-{\epsilon\over 2}
\end{align*}
Thus, we have that
\begin{align*}
{\epsilon^{3/2}\over \sqrt{8L_{ij}}}~>~\|\Psi(h^1)-\Psi(h^2)\|_{L^2}&~=~\left(\sum_{p\in\mathcal{P}}\int_{t_0}^{t_f}\big|\Psi_p(t,\,h^1)-\Psi_p(t,\,h^2)\big|^2\,dt\right)^{1/2}
\\
&~\geq~\left( \int^{\hat t}_{\hat t-{\epsilon\over 2L_{ij}}} \big|\Psi_{\hat p}(t,\,h^1)-\Psi_{\hat p}(t,\,h^2)\big|^2\,dt \right)^{1/2}~\geq~ \left( {\epsilon\over 2L_{ij}} \cdot {\epsilon^2\over 4} \right)^{1/2}
\end{align*}
\noindent which leads to a contradiction. Thus \eqref{contproofcontr} must hold. We have shown that if $\|h^1-h^2\|_{L^2}\to 0$ then $|v_{ij}(h^1)-v_{ij}(h^2)|\to 0$ for every $(i,\,j)\in\mathcal{W}$.\\

\noindent {\bf Part 2.} In this part, we show that if $h^{(n)}$ is a sequence that converges to $h^*$ in the $L^2$-norm, then $\|\Phi^{\varepsilon}(h^{(n)})-\Phi^{\varepsilon}(h^*)\|_{L^2} \to 0$ as $n\to +\infty$, thereby establishing the desired continuity result. Indeed, according to \eqref{Phidef3},  for every $t\in[t_0,\,t_f]$ and every $p\in\mathcal{P}_{ij}$,
\begin{align*}
\big|\Phi_p^{\varepsilon}(t,\,h^{(n)})-\Phi_p^{\varepsilon}(t,\,h^*)\big|~\leq~&\max\Bigg\{\big|\Psi_p(t,\,h^{(n)})-\Psi_p(t,\,h^*)\big|~,~ \big|v_{ij}(h^{(n)})-v_{ij}(h^*)+\varepsilon_{ij}^p(h^{(n)})-\varepsilon_{ij}^p(h^*)\big| \Bigg\}
\\
~+~&\big|\varepsilon_{ij}^p(h^{(n)})-\varepsilon_{ij}^p(h^*)-\min_{q\in\mathcal{P}_{ij}}\big\{\varepsilon_{ij}^q(h^{(n)})\big\}+\min_{q\in\mathcal{P}_{ij}}\big\{\varepsilon_{ij}^q(h^*)\big\}\big|
\\
~\leq~&\big|\Psi_p(t,\,h^{(n)})-\Psi_p(t,\,h^*)\big|+A_{ij}^{(n)}
\end{align*}
where 
$$
A_{ij}^{(n)}\doteq \big|v_{ij}(h^{(n)})-v_{ij}(h^*)+\varepsilon_{ij}^p(h^{(n)})-\varepsilon_{ij}^p(h^*)\big|+\big|\varepsilon_{ij}^p(h^{(n)})-\varepsilon_{ij}^p(h^*)-\min_{q\in\mathcal{P}_{ij}}\big\{\varepsilon_{ij}^q(h^{(n)})\big\}+\min_{q\in\mathcal{P}_{ij}}\big\{\varepsilon_{ij}^q(h^*)\big\}\big|
$$
\noindent tends to zero as $n \to +\infty$ according to the result established in {\bf Part 1}, and the fact that each $\varepsilon_{ij}^p(\cdot)$ is continuous. As a consequence, we deduce that 
\begin{align*}
\big\|\Phi^{\varepsilon}(h^{(n)})-\Phi^{\varepsilon}(h^*)\big\|^2_{L^2}~=~&\sum_{p\in\mathcal{P}}\int_{t_0}^{t_f}\big|\Phi_p^{\varepsilon}(t,\,h^{(n)})-\Phi_p^{\varepsilon}(t,\,h^*)\big|^2\,dt
\\
~\leq~&\sum_{(i,\,j)\in\mathcal{W}} \sum_{p\in\mathcal{P}_{ij}}\int_{t_0}^{t_f} \Big(\big|\Psi_p(t,\,h^{(n)})-\Psi_p(t,\,h^*)\big|+A_{ij}^{(n)}\Big)^2\,dt
\\
~\leq~&\sum_{(i,\,j)\in\mathcal{W}}\sum_{p\in\mathcal{P}_{ij}}\int_{t_0}^{t_f} \big|\Psi_p(t,\,h^{(n)})-\Psi_p(t,\,h^*)\big|^2\,dt
\\
~+~& \sum_{(i,\,j)\in\mathcal{W}}\sum_{p\in\mathcal{P}_{ij}}\int_{t_0}^{t_f}2A_{ij}^{(n)}\big|\Psi_p(t,\,h^{(n)})-\Psi_p(t,\,h^*)\big| +\Big(A_{ij}^{(n)}\Big)^2\,dt
\\
~\to~& 0,\quad\hbox{as } ~n\to +\infty
\end{align*}
\end{proof}

\section{Proof of Theorem \ref{existencethembrdue}}\label{secapp4}
\begin{proof}
We consider, for each natural number $n\geq 1$,  a uniform partition of the compact interval $[t_0,\,t_f]$ into $n$ sub-intervals $I_1,\, \ldots,\,I_n$ with the size of each being $(t_f-t_0)/n$. We then consider the following subsets:
$$
\Lambda^n~\doteq~\left\{ h\in\Lambda:~~ h_p(\cdot) \hbox{ is constant on }~ I_i,\quad\forall i=1,\,\ldots,\, n, \quad \forall p\in\mathcal{P}\right\}\subset\Lambda\qquad\forall n\geq 1
$$
\noindent Notice that each $\Lambda^n$ is the intersection of $\Lambda$ and the space of piecewise constant functions. Since $\Lambda$ is expressed via linear constraints, each set $\Lambda^n$ is clearly convex. In addition, due to the finite-dimensional nature of $\Lambda^n$,  it is also compact. A detailed proof of compactness for $\Lambda^n$ based on sequential compactness has been presented in \cite{existence}, and will be omitted here.

It follows from assumption 2 and Theorem \ref{mainthm} that for each $n\geq 1$, there exists $h^{n, *}\in\Lambda^n$ such that 
\begin{equation}\label{nvi}
\left<\Phi^{\varepsilon}(h^{n,*}),\, h^n-h^{n,*}\right>\geq 0\quad \hbox{or}\quad \sum_{p\in\mathcal{P}}\int_{t_0}^{t_f}\Phi^{\varepsilon}_p(t,\,h^{n,*})\left(h^n_p(t)-h^{n,*}_p(t)\right)\,dt \geq 0, \qquad\forall h^n\in\Lambda^n
\end{equation}
\noindent Since both $h^n$ and $h^{n,*}$ are piecewise constant, it follows from \eqref{nvi} that for any $(i,\,j)\in\mathcal{W}$,
\begin{equation}\label{contra}
h_p^{n,*}(t)>0,~~ t\in I_k ~~\Longrightarrow~~ \int_{I_k}\Phi^{\varepsilon}_p(t,\,h^{n,*})\,dt~=~\min_{q\in\mathcal{P}_{ij}}\min_{l=1,\ldots, n} \int_{I_l} \Phi^{\varepsilon}_q(t,\,h^{n,*})\,dt
\end{equation}
\noindent for all $p\in\mathcal{P}_{ij}$ and $k=1,\,\ldots,\,n$. In other words, within the same origin-destination pair, the integrals of the function $\Phi_p^{\varepsilon}(\cdot,\,h^{n,*})$ for all utilized paths and departure time intervals are equal and minimal. 

In view of assumptions 1 and 3, we choose $n\geq 1$  such that the size of the subintervals, $\delta_n\doteq {t_0-t_f\over n}$, satisfies $\delta_n < {\varepsilon^{min}\over 2\max_{(i,\,j)\in\mathcal{W}} L_{ij}}$. Fixing any origin-destination pair $(i,\,j)\in\mathcal{W}$, we denote by $v_{ij}(h^{n,*})$ the essential infimum of the path effective delays, which is attained at some point $\hat t\in I_l$ corresponding to some path $q\in\mathcal{P}_{ij}$.  Clearly, by taking the time horizon large enough one can always assume that $l>1$. That is, $I_l$ is not the first time interval because the corresponding early arrival penalty would be very large. Then we consider the interval $I_{l-1}$ and deduce the following based on \eqref{PsionesideL}, which is a consequence of {\bf A0},
\begin{multline}
\Psi_q(\hat t,\,h^{n,*})-\Psi_q(t,\,h^{n,*})~\geq~-L_{ij}(\hat t -t )
\\
~\Longrightarrow~\Psi_q(t,\,h^{n,*})~\leq~ v_{ij}(h^{n,*})+L_{ij}(\hat t-t) ~\leq~ v_{ij}(h^{n,*})+2L_{ij}\delta_n~<~ v_{ij}(h^{n,*})+ \varepsilon^{min}
\end{multline}
\noindent for all $t\in I_{l-1}$. We immediately have that $\forall t\in I_{l-1}$,
$$
\Phi^{\varepsilon}_q(t,\,h^{n,*})~=~v_{ij}(h^{n,*})+\varepsilon_{ij}^{q}(h^{n,*})-\left(\varepsilon_{ij}^q(h^{n,*})  -\min_{q'\in\mathcal{P}_{ij}}\{\varepsilon_{ij}^{q'}(h^{n,*})\} \right)~=~v_{ij}(h^{n,*})+\min_{q'\in\mathcal{P}_{ij}}\{\varepsilon_{ij}^{q'}(h^{n,*})\}
$$

\noindent For any $p\in\mathcal{P}_{ij}$ and any interval $I_k$ such that $h_p^{n,*}(t)>0$, $t\in I_k$, \eqref{contra} implies that 
\begin{equation}\label{existenceprooff1}
\int_{I_k}\Phi_p^{\varepsilon}(t,\,h^{n,*})\,dt~\leq~\int_{I_{l-1}}\Phi_q^{\varepsilon}(t,\,h^{n,*})\,dt~=~\delta_n\cdot \left(v_{ij}(h^{n,*})+\min_{q'\in\mathcal{P}_{ij}}\{\varepsilon_{ij}^{q'}(h^{n,*})\}\right)
\end{equation}
On the other hand, by definition 
\begin{equation}\label{existenceprooff2}
\Phi_p^{\varepsilon}(t,\,h^{n,*})~\geq~v_{ij}(h^{n,*})+\varepsilon_{ij}^{p}(h^{n,*})-\left(\varepsilon_{ij}^p(h^{n,*})  -\min_{q'\in\mathcal{P}_{ij}}\{\varepsilon_{ij}^{q'}\} \right)~=~v_{ij}(h^{n,*})+\min_{q'\in\mathcal{P}_{ij}}\{\varepsilon_{ij}^{q'}\}
\end{equation}
\eqref{existenceprooff1} and \eqref{existenceprooff2} together implies that $\Phi_p^{\varepsilon}(t,\,h^{n,*})= v_{ij}(h^{n,*})+\min_{q'\in\mathcal{P}_{ij}}\{\varepsilon_{ij}^{q'}\}$ for almost every $t\in I_k$. Finally, by definition \eqref{Phidef1}, such an equality holds if and only if $\Psi_p(t,\,h^{n,*})\leq v_{ij}(h^{n,*})+\varepsilon_{ij}^p(h^{n,*})$ for almost every $t\in I_k$. Since $(i,\,j)$, $p$ and $I_k$ are arbitrary, we have established that $h^{n,*}$ is a VT-BR-DUE. 
\end{proof}

\section{Proofs of Lemma \ref{lemmabarphicont}, Proposition  \ref{thmclosedness}, Proposition \ref{propsamesupport}, and Proposition \ref{propconnected}}\label{secappchara}

\subsection{Proof of Lemma \ref{lemmabarphicont}}\label{subsecappbarcont}
\begin{proof}
For each $n\geq 1$, fix any point $\bar h\in\bar\Lambda^n$, and consider an arbitrary sequence $\bar h^{(m)}\in \bar\Lambda^n$, $m\geq 1$, that converge to $\bar h$ in the Euclidean norm. We let $h\in \Lambda^n$ and $h^{(m)}\in\Lambda^n$ be the continuous-time counterparts of $\bar h$ and $\bar h^{(m)}$, respectively. It is easy to verify that $h^{(m)}\to h$ in the $L^2$-norm in $\big(L^2[t_0,\,t_f]\big)^{|\mathcal{P}|}$. Thus, according to the continuity of $\Psi$, we have
$$
\left\|\Psi(h)-\Psi(h^{(m)}) \right\|^2_{L^2}~=~\sum_{p\in\mathcal{P}}\int_{t_0}^{t_f}\left|\Psi_p(t,\,h)-\Psi_p(t,\,h^{(m)})\right|^2\,dt ~\to~0,\quad\hbox{as } m\to +\infty
$$
\noindent Next, in order to show that $\bar\Psi$ is continuous it suffices to verify that $\bar \Psi_p(k,\,\bar h^{(m)})\to \bar \Psi_p(k,\,\bar h)$ for every $p\in\mathcal{P}$ and $1\leq k\leq n$, as $m$ tends to infinity. Indeed, recalling from \eqref{barPsipkdef} that
\begin{align*}
\left|\bar\Psi_p(k,\,\bar h) -\bar\Psi_p(k,\,\bar h^{(m)})\right|~=~&{1\over |I_k|} \left| \int_{I_k}\Psi_p(t,\,h)-\Psi_p(t,\,h^{(m)})\,dt  \right|
\\
~\leq~&{1\over |I_k|}\int_{I_k}\left|\Psi_p(t,\,h)-\Psi_p(t,\,h^{(m)})\right|\,dt
\\
~\leq~&{1\over |I_k|}\cdot \sqrt{|I_k|}\cdot\left( \int_{I_k} \left|\Psi_p(t,\,h)-\Psi_p(t,\,h^{(m)})\right|^2\,dt \right)^{1\over 2}~\to~0\quad \hbox{as } m\to +\infty
\end{align*}
for every $k$ and every $p$. Notice that the last inequality is a consequence of Jensen's inequality. This establishes the continuity of $\bar\Psi$. 

To see the continuity of $\bar\Phi^{\bar\varepsilon}$, we notice that the continuity of functionals $\varepsilon_{ij}^p(\cdot)$, $p\in\mathcal{P}_{ij},\,(i,\,j)\in\mathcal{W}$, immediately leads to the continuity of the functions $\bar\varepsilon_{ij}^p(\cdot)$ by definition. The rest of the proof simply follows from  \eqref{barPhidef1}.
\end{proof}

\subsection{Proof of Proposition \ref{thmclosedness}}\label{subsecappcompact}
\begin{proof}
We let $\bar h^{(m)}\in \bar\Lambda^n$, $m\geq 1$ be an arbitrary sequence of solutions that converge to some $\bar h^*\in\mathbb{R}_+^{n\times|\mathcal{P}|}$. Moreover,  for every $(i,\,j)\in\mathcal{W}$,
$$
\delta t\sum_{p\in\mathcal{P}_{ij}}\sum_{k=1}^n\bar h^*_p(k)~=~\delta t\sum_{p\in\mathcal{P}_{ij}}\sum_{k=1}^n\lim_{m\to+\infty}\bar h_p^{(m)}(k)~=~\lim_{m\to+\infty}\delta t\sum_{p\in\mathcal{P}_{ij}}\sum_{k=1}^n\bar h^{(m)}_p(k)~=~Q_{ij}
$$
\noindent where $\delta t$ is the time step size. This means that $\bar h^*\in\bar \Lambda^n$. We will next show that $\bar h^*$ is a solution of the VT-BR-DUE problem by proving that it satisfies the VI \eqref{fdvtbrduevi}. Indeed, for any $\bar h\in\bar\Lambda^n$, we have that 
$$
\sum_{p\in\mathcal{P}}\sum_{k=1}^n\bar\Phi_p^{\bar\varepsilon}(k,\,\bar h^{(m)})\big(\bar h_p(k)-\bar h_p^{(m)}(k)\big) ~\geq~0 \qquad\forall m\geq 1
$$
According to the continuity result provided by Lemma \ref{lemmabarphicont}, we deduce that
\begin{align*}
0~\leq~\lim_{m\to+\infty}\sum_{p\in\mathcal{P}}\sum_{k=1}^n\bar\Phi_p^{\bar\varepsilon}(k,\,\bar h^{(m)})\big(\bar h_p(k)-\bar h_p^{(m)}(k)\big)~=~&\sum_{p\in\mathcal{P}}\sum_{k=1}^n \lim_{m\to+\infty}\bar\Phi_p^{\bar\varepsilon}(k,\,\bar h^{(m)})\big(\bar h_p(k)-\bar h_p^{(m)}(k)\big)
\\
~=~&\sum_{p\in\mathcal{P}}\sum_{k=1}^n   \bar\Phi_p^{\bar\varepsilon}(k,\,\bar h^*)\big(\bar h_p(k)-\bar h_p^*(k)\big)
\end{align*}
for all $\bar h\in \bar \Lambda^n$. Thus $\bar h^*$ is a solution. We have thus shown that the set of discrete-time VT-BR-DUE solutions is closed.

To see that the set is also bounded, we have the following estimate for a given $n\geq 1$:
$$
\bar h_p(k)~\leq~ {\max_{(i,\,j)\in\mathcal{W}} Q_{ij} \over (t_f-t_0)/n}   \qquad\forall 1\leq k\leq n,~~\forall p\in\mathcal{P}
$$
\noindent for all $\bar h\in\bar\Lambda^n$. In other words, the vectors in $\bar\Lambda^n$ are element-wise uniformly bounded, thus their norms are also uniformly bounded. 
\end{proof}

\subsection{Proof of Proposition \ref{propsamesupport}}\label{subsecapppropsamesupport}

\begin{proof}
(i) According to assumption 3 in Theorem \ref{existencethembrdue}, there exists $\varepsilon^{min}>0$ such that $\varepsilon_{ij}^p(h)\geq \varepsilon^{min}$ for all $h\in\Lambda$. Such a property easily transfers to the finite-dimensional counterpart by  definition; that is, $\bar\varepsilon_{ij}^p(\bar h)\geq \varepsilon^{min}$ holds for all $\bar h\in\bar\Lambda^n$ and $p\in\mathcal{P}_{ij}$, $(i,\,j)\in\mathcal{W}$. Fix any number $0<\delta <\varepsilon^{min}$, we define a new set of tolerance functions $\tilde\varepsilon_{ij}^p(\cdot):\bar\Lambda^n\to \mathbb{R}_{++}$ such that
$$
\tilde\varepsilon_{ij}^p(\bar h)~=~\bar\varepsilon_{ij}^p(\bar h)-\delta\qquad\forall \bar h\in\bar\Lambda^n,\quad\forall p\in\mathcal{P}_{ij},\quad \forall (i,\,j)\in\mathcal{W}
$$
\noindent Clearly, these new tolerance functions are continuous. Hence the mapping $\bar\Phi^{\tilde\varepsilon}$, defined via \eqref{barPhidef1} by replacing $\bar\varepsilon_{ij}^p(\cdot)$ with $\tilde\varepsilon_{ij}^p(\cdot)$, is continuous as well. We then apply Browder's existence theorem to obtain a $\bar h^*\in\bar\Lambda^n$ such that 
$$
\sum_{p\in\mathcal{P}}\sum_{k=1}^n\bar \Phi^{\tilde\varepsilon}_p(k,\,\bar h^*)\big(\bar h_p(k)-\bar h_p^*(k)\big)~\geq~0\qquad\forall \bar h\in\bar\Lambda^n
$$
\noindent The VI above clearly leads to the following statement: for all $(i,\,j)\in\mathcal{W}$,
$$
\bar h^*_p(k)~>~0,~p\in\mathcal{P}_{ij}~\Longrightarrow~\bar\Psi_p(k,\,\bar h^*)~<~v_{ij}(\bar h^*)+\tilde\varepsilon_{ij}^p(\bar h^*)~<~v_{ij}(\bar h^*)+\bar\varepsilon_{ij}^p(\bar h^*)
$$
\noindent where $v_{ij}(\bar h^*)$ is the minimum effective delay within $(i,\,j)$. Thus $\bar h^*$ is a solution with the {\bf (P)} property.\\

\noindent (ii) Let $\bar h^*$ be a solution with the {\bf (P)} property. We define
\begin{equation}\label{proofpropsamesupporteqn0}
\sigma~\doteq~\min_{\mathcal{O}(I_k,\,p)\in\mathcal{F}(\bar h^*)}\left\{v_{ij}(\bar h^*)+\bar\varepsilon_{ij}^p(\bar h^*)-\bar\Psi_p(k,\,\bar h^*) ,~~ p\in\mathcal{P}_{ij}  \right\}~>~0
\end{equation}
According to the continuity of $\bar\Psi(\cdot)$ and $\bar\varepsilon_{ij}^p(\cdot)$, there exists a $\delta>0$ such that whenever $\left\|\bar h-\bar h^*\right\|_2<\delta$ there holds 
\begin{equation}\label{proofpropsamesupporteqn1}
\left\|\bar\Psi(\bar h)-\Psi(\bar h^*)\right\|_2<\sigma/3,\qquad \left|\bar\varepsilon_{ij}^p(\bar h)-\bar\varepsilon_{ij}^p(\bar h^*)\right|<\sigma/3   \qquad \forall p\in\mathcal{P}_{ij},\,\forall (i,\,j)\in\mathcal{W}
\end{equation}
\noindent  where $\|\cdot\|_2$ is the Euclidean norm.  Fix any point $\bar h\in \mathcal{B}_{\bar h^*}^{\delta}\cap\bar\Lambda^n\cap span\left\{\mathbf{e}_{l}:~l\in\mathcal{F}(\bar h^*)\right\}$. We show that $\bar h$ is a solution. First of all, notice that since $\left\|\bar h-\bar h^*\right\|_2<\delta$, we must have that 
\begin{equation}\label{proofpropsamesupporteqn2}
\left|\bar\Psi_p(k,\,\bar h)-\bar\Psi_p(k,\,\bar h^*)\right|~\leq~\left\|\bar\Psi(\bar h)-\bar\Psi(\bar h^*)\right\|_2~<~\sigma/3 \qquad \forall 1\leq k\leq n,~\forall p\in\mathcal{P}
\end{equation}
\noindent Consequently, we must also have
\begin{equation}\label{proofpropsamesupporteqn3}
\left|v_{ij}(\bar h)-v_{ij}(\bar h^*)\right|~<~\sigma/3\qquad \forall (i,\,j)\in\mathcal{W}
\end{equation}

\noindent Given the fact that $\bar h\in span\left\{\mathbf{e}_{l}:~l\in\mathcal{F}(\bar h^*)\right\}$, for any $(i,\,j)$,
\begin{equation}\label{proofpropsamesupporteqn4}
\bar h_p(k)~>~0,\, p\in\mathcal{P}_{ij}~\Longrightarrow~\bar\Psi_p(k,\,\bar h^*)~<~v_{ij}(\bar h^*)+\bar\varepsilon_{ij}^p(\bar h^*)
\end{equation}
\noindent A consequence of the right hand side of \eqref{proofpropsamesupporteqn4}, along with \eqref{proofpropsamesupporteqn0}-\eqref{proofpropsamesupporteqn3}, is that 
\begin{align*}
\bar \Psi_p(k,\,\bar h)~<~\sigma/3+\bar\Psi_p(k,\,\bar h^*)~\leq~&\sigma/3+v_{ij}(\bar h^*)+\bar \varepsilon_{ij}^p(\bar h^*)-\sigma
\\
~\leq~&\sigma/3+\sigma/3+v_{ij}(\bar h)+\sigma/3+\bar\varepsilon_{ij}^p(\bar h)-\sigma
\\
~=~&v_{ij}(\bar h)+\bar\varepsilon_{ij}^p(\bar h)
\end{align*}
\noindent We have established the following:
\begin{equation}\label{proofpropsamesupporteqn5}
\bar h_p(k)~>~0,\, p\in\mathcal{P}_{ij}~\Longrightarrow~\bar \Psi_p(k,\,\bar h)~<~v_{ij}(\bar h)+\bar\varepsilon_{ij}^p(\bar h), \quad \forall (i,\,j)\in\mathcal{W},
\end{equation}
\noindent which shows that $\bar h_p(k)$ is a solution with the {\bf (P)} property.\\

\noindent (iii) We note that the set $\mathcal{B}_{\bar h^*}^{\delta}~\cap~\bar\Lambda^n~\cap~ span\Big\{\mathbf{e}_l:~ l\in \mathcal{F}(\bar h^*) \Big\}$ is infinite and, by \eqref{proofpropsamesupporteqn5}, every point in this set has the {\bf (P)} property.
\end{proof}

\subsection{Proof of Proposition \ref{propconnected}}\label{subsecapppropconnected}
\begin{proof}
By contradiction, if $\mathcal{C}(\bar h^*)$ is not connected, then there exist two subsets, $A$ and $B$, of $\mathcal{C}(\bar h^*)$ such that $A\cap B=\emptyset$ and $A\cup B=\mathcal{C}(\bar h^*)$. Moreover, both $A$ and $B$ are open in the relative topology. That is, there exist open sets $A^0,\,B^0\subset \mathbb{R}^{n\times|\mathcal{P}|}$ such that
$$
A~=~\mathcal{C}(\bar h^*)~\cap~A^0,\qquad B~=~\mathcal{C}(\bar h^*)~\cap~B^0
$$ 
\noindent Without loss of generality, we assume $\bar h^*\in A$. According to the way $\mathcal{C}(\bar h^*)$ is constructed, there exists at least one point $\bar h^2 \in B$ such that $\bar h^2\in\mathcal{S}(\bar h^1)$ for some $\bar h^1\in A$. Since $\mathcal{S}(\bar h^1)$ is convex, it is connected. We thus consider two nonempty subsets of $\mathcal{S}(\bar h^1)$: $A'\doteq\mathcal{S}(\bar h^1)\cap A$ and $B'\doteq \mathcal{S}(\bar h^1)\cap B$. Clearly $A'\cap B'=\emptyset$ and  $A'\cup B'=\mathcal{S}(\bar h^1)$. In addition, 
\begin{align*}
A'~=~&\mathcal{S}(\bar h^1)~\cap~\mathcal{C}(\bar h^*)~\cap~A^0~=~\mathcal{S}(\bar h^1)~\cap~A^0
\\
B'~=~&\mathcal{S}(\bar h^1)~\cap~\mathcal{C}(\bar h^*)~\cap~B^0~=~\mathcal{S}(\bar h^1)~\cap~B^0
\end{align*}
\noindent which shows that both $A'$ and $B'$ are open in the relative topology. Thus $\mathcal{S}(\bar h^1)$ is not connected and we have reached a contradiction. 
\end{proof}

\end{document}